\documentclass{article}

\usepackage[english]{babel}

\usepackage[letterpaper, left=1.2truein, right=1.2truein, top = 1.2truein, bottom = 1.2truein]{geometry}

\usepackage{graphicx}

\usepackage{amsmath, amsfonts, amsthm, amssymb, MnSymbol}
\usepackage{mathtools}
\usepackage{hyperref} 
\hypersetup{colorlinks,citecolor=blue,linkcolor=blue}

\usepackage{accents}
\usepackage[margin=1cm]{caption}

\usepackage{algorithm}
\usepackage{algpseudocode}

\usepackage[normalem]{ulem}
\usepackage{enumerate}
\usepackage{float}

\usepackage{xcolor}
\newcommand{\blue}[1]{#1}

\usepackage[blocks, affil-it]{authblk}
\usepackage[round]{natbib}

\usepackage{etoolbox}

\theoremstyle{plain}

\newtheorem{theorem}{Theorem}
\newtheorem{lemma}[theorem]{Lemma}
\newtheorem{corollary}[theorem]{Corollary}
\newtheorem{prop}[theorem]{Proposition}
\newtheorem{remark}[theorem]{Remark}

\theoremstyle{remark}

\newcommand{\R}{\mathbb{R}}

\newcommand{\F}{\mathbb{F}}

\newcommand{\EE}{\mathbb{E}}
\newcommand{\E}{\mathbb{E}}

\newcommand{\PP}{\mathbb{P}}

\renewcommand{\>}{\rangle}

\newcommand{\III}[1]{{\left\vert\kern-0.25ex\left\vert\kern-0.25ex\left\vert #1 
    \right\vert\kern-0.25ex\right\vert\kern-0.25ex\right\vert}}

\newcommand{\simiid}{\stackrel{\text{i.i.d.}}{\sim}}
\newcommand{\simind}{\stackrel{\text{ind.}}{\sim}}

\newcommand{\ca}{\mathcal{A}}  \newcommand{\cc}{\mathcal{C}} \newcommand{\cd}{\mathcal{D}}    \newcommand{\ch}{\mathcal{H}}    \newcommand{\cl}{\mathcal{L}} \newcommand{\cm}{\mathcal{M}} \newcommand{\cn}{\mathcal{N}}  \newcommand{\cp}{\mathcal{P}}         \newcommand{\cz}{\mathcal{Z}}

\newcommand{\eps}{\varepsilon}

\newcommand{\conv}{\textrm{conv}} 

\makeatletter
\providecommand*{\diff}%
	{\@ifnextchar^{\DIfF}{\DIfF^{}}}
\def\DIfF^#1{%
	\mathop{\mathrm{\mathstrut d}}%
		\nolimits^{#1}\gobblespace}
\def\gobblespace{%
	\futurelet\diffarg\opspace}
\def\opspace{%
	\let\DiffSpace\!%
	\ifx\diffarg(%
		\let\DiffSpace\relax
	\else
		\ifx\diffarg[%
			\let\DiffSpace\relax
	\else
		\ifx\diffarg\{%
			\let\DiffSpace\relax
		\fi\fi\fi\DiffSpace}

\renewcommand{\P}{{\mathbb P}}

\newcommand{\M}{{\mathcal{M}}}

\newcommand{\Ps}{{\mathcal{P}}}

\newcounter{rcnt}[section]

\def\argmin{\mathop{\rm argmin}}
\def\argmax{\mathop{\rm argmax}}
\def\cp{\mathcal{P}}

\def\eps{\varepsilon}

\newcommand{\aprior}{G}
\newcommand{\trueprior}{\aprior^*}
\newcommand{\npmle}{{\widehat{\aprior}_n}}

\newcommand{\avec}{\theta}
\newcommand{\truevec}{\avec^*} 
\newcommand{\estvec}{\hat\avec} 
\newcommand{\orvec}{\hat\avec^*} 

\newcommand{\p}{d} 
\newcommand{\khat}{\hat{k}} 

\begin{document}

\title{Multivariate, Heteroscedastic Empirical Bayes\\via Nonparametric Maximum Likelihood}

\author[1]{Jake A. Soloff\thanks{Work done while at UC Berkeley.}}
\author[2]{Adityanand Guntuboyina}
\author[3]{Bodhisattva Sen}
\affil[1]{Department of Statistics, University of Chicago}
\affil[2]{Department of Statistics, University of California, Berkeley}
\affil[3]{Department of Statistics, Columbia University}

\date{\today}

\maketitle

\begin{abstract}
Multivariate, heteroscedastic errors complicate statistical inference in many large-scale denoising problems. Empirical Bayes is attractive in such settings, but standard parametric approaches rest on assumptions about the form of the prior distribution which can be hard to justify and which introduce unnecessary tuning parameters. We extend the nonparametric maximum likelihood estimator (NPMLE) for Gaussian location mixture densities to allow for multivariate, heteroscedastic errors. NPMLEs estimate an arbitrary prior by solving an infinite-dimensional, convex optimization problem; we show that this convex optimization problem can be tractably approximated by a finite-dimensional version. 

The empirical Bayes posterior means based on an NPMLE have low regret, meaning they closely target the oracle posterior means one would compute with the true prior in hand. We prove an oracle inequality implying that the empirical Bayes estimator performs at nearly the optimal level (up to logarithmic factors) for denoising without prior knowledge. We provide finite-sample bounds on the average Hellinger accuracy of an NPMLE for estimating the marginal densities of the observations. We also demonstrate the adaptive and nearly-optimal properties of NPMLEs for deconvolution. We apply our method to two denoising problems in astronomy, constructing a fully data-driven color-magnitude diagram of 1.4 million stars in the Milky Way and investigating the distribution of 19 chemical abundance ratios for 27 thousand stars in the red clump. We also apply our method to hierarchical linear models, illustrating the advantages of nonparametric shrinkage of regression coefficients on an education data set and on a microarray data set.
\vspace{.25em}


\end{abstract}

\section{Introduction}\label{sec-intro}

Consider a $\p$-dimensional ($\p \ge 1$), heteroscedastic normal observation model
\begin{align}\label{eq-obs-model}
X_i\mid \truevec_i &\simind \cn(\truevec_i, \Sigma_i), \qquad\mbox{with } \truevec_i \simiid \trueprior, \qquad \;\mbox{for } i\in \{1,\dots, n\},
\end{align}
where $(\Sigma_i)_{i=1}^n$ is a known sequence of $\p\times \p$ positive-definite covariance matrices, and the underlying mean vectors $(\truevec_i)_{i=1}^n$ are additionally assumed to be drawn from a common prior $\trueprior$, where $\trueprior$ belongs to the collection $\Ps(\R^\p)$ of all probability measures on $\R^\p$. In settings where~$\trueprior$ is known, model~\eqref{eq-obs-model} fully specifies a Bayesian model; this paper studies the common empirical Bayes setting where~$\trueprior$ must be estimated. The main goal of the paper is to nonparametrically estimate $\trueprior$ and the sequence $(\truevec_i)_{i=1}^n$ from the observed data $(X_i, \Sigma_i)_{i=1}^n$.

By allowing for arbitrary prior distributions $\trueprior\in \Ps(\R^\p)$, model~\eqref{eq-obs-model} captures a range of important structural assumptions on the underlying sequence~$(\truevec_i)_{i=1}^n$. For instance, the clustering problem---where the terms of $(\truevec_i)_{i=1}^n$ take on at most $k^*$ distinct values---corresponds to discrete~$\trueprior$; sparse modeling---where most of the $(\truevec_i)_{i=1}^n$ are zero---corresponds to~$\trueprior(\{0\})\approx 1$. The model also accommodates more complex manifold-like structures (see e.g. Figure~\ref{fig-cmd}) as well as substantially more heterogeneous sequences (e.g. heavy tailed~$\trueprior$). Rather than assume a particular form for the latent structure represented by~$\trueprior$, we apply the nonparametric maximum likelihood estimator (NPMLE), which searches over all possible forms of latent structure that a set of precise measurements could have, identifying the underlying structure that best explains the real, noisy observations. Through detailed theoretical analysis, simulations and real data analyses, we demonstrate that the NPMLE exhibits a remarkable ability to adapt to latent structure where it is present and back away when no structure is available. 

Empirical Bayes methods for the normal sequence model~\eqref{eq-obs-model} have been studied extensively in the univariate, homoscedastic setting where $\p=1$ and $\Sigma_i\equiv\sigma^2$ (see, e.g., \citet{james1961, efron1972empirical, efron1972limiting, efron1973combining, efron1973stein, morris1983parametric, jiang2009general, efron2012large, efron2014two} as well as \citet{johnstone2019gaussian} for a manuscript on estimation in Gaussian sequence models). Numerous methods extend empirical Bayes to the univariate, heteroscedastic case \citep[see][and references therein]{jiang2011best, xie2012sure, tan2016steinized, weinstein2018group, jiang2020general, banerjee2023nonparametric,chen2023empirical,ignatiadis2023empirical}. Relatively little attention has been given to the general case of the present paper. 

\subsection{Applications to denoising problems}\label{sec-denoising-app}
 
In the multivariate, heteroscedastic setting, model~\eqref{eq-obs-model} naturally arises in the analysis of astronomy data, where often a calibrated measurement error distribution comes attached to each observation, and typically these errors are heteroscedastic \citep{kelly2012measurement}; also see e.g.~\citet{AB96},~\citet{HDB10},~\citet{anderson2018improving}. The first part of model~\eqref{eq-obs-model} indicates that the target sequence~$(\truevec_i)_{i=1}^n$ has, due to measurement error, been corrupted by additive, zero-mean Gaussian noise, i.e.
\begin{equation*}
X_i = \truevec_i + \epsilon_i, \mbox{ where }\epsilon_i \simind \cn(0, \Sigma_i), \qquad \mbox{for } i = 1,\ldots,n.
\end{equation*}
Interestingly, the $\Sigma_i$'s above, which typically differ across $i$, are known in many applications where the measurement process is well-characterized. In many situations it is assumed that $\truevec_i$ is itself random and independent of $\epsilon_i$ for all $i$. Although each observation has a different error distribution, the~$n$ observations are tied together by the assumption that the $\truevec_i$'s are i.i.d.~from some distribution~$\trueprior$, yielding model~\eqref{eq-obs-model}. 

Our motivating example for model~\eqref{eq-obs-model} involves the construction of a precise stellar color-magnitude diagram. A color-magnitude diagram (CMD) is a scatter plot of stars, displaying their absolute magnitude (a proxy for luminosity) versus color (a proxy for surface temperature) to provide a cross-sectional view of stellar evolution. The continued expansion of available stellar measurements has made purely statistical models such as model~\eqref{eq-obs-model} increasingly attractive for denoising. One approach to estimating~$\trueprior$---which is known as {\sl Extreme Deconvolution} (XD) \citep{bovy2011extreme} and is especially popular in astronomy applications---assumes
\begin{align}\label{eq-gmm-prior}
\trueprior = \sum_{j=1}^K \alpha_j^*\cn(\mu_j^*, V_j^*)
\end{align}
and estimates the parameters $(\alpha_j^*, \mu_j^*, V_j^*)_{j=1}^K$ via the {\sl Expectation-Maximization} (EM) algorithm with split-and-merge operations designed to avoid local optima. For instance, \citet{anderson2018improving} applied XD to build a low-noise CMD with $n\approx1.4$ million de-reddened stars from the Gaia TGAS catalogue. The XD assumption~\eqref{eq-gmm-prior} that the prior~$\trueprior$ is itself a mixture of $K$-Gaussians has a number of drawbacks. Although the class of Gaussian location-scale mixtures is flexible for large $K$, the choice of $K$ requires tuning; violations of assumption~\eqref{eq-gmm-prior} for fixed~$K$ induce bias in the estimation. To our knowledge, no theoretical results for the statistical properties of XD are available, making it difficult to quantify the misspecification error. Moreover, the class of all probability distributions of the form~\eqref{eq-gmm-prior} is non-convex for finite~$K$, so even split-and-merge techniques employed within EM do not guarantee convergence to the global maximizer of the likelihood. 

\begin{figure}[t!]
\centering
\includegraphics[width=.9\textwidth]{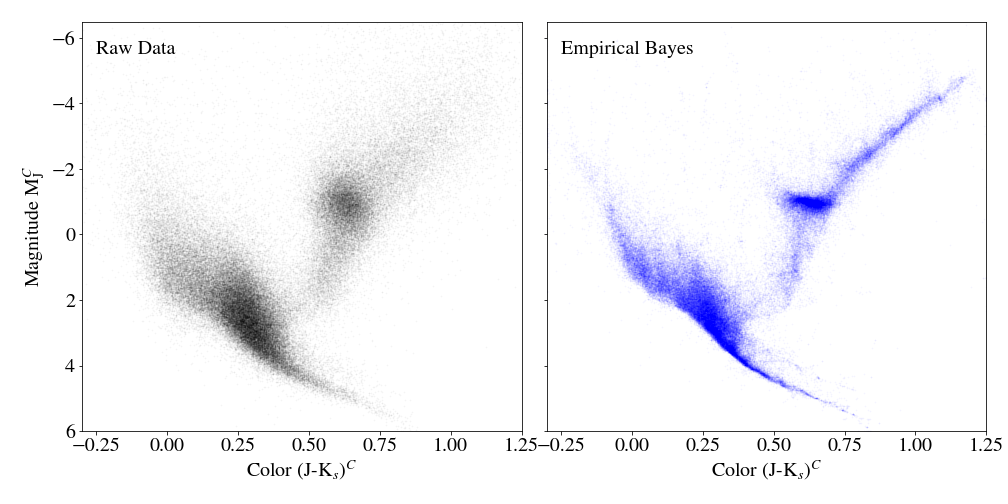} 
\caption{A noisy color-magnitude diagram (CMD) corresponding to the observations $X_i$ in model~\eqref{eq-obs-model}, with corresponding fully-nonparametric denoised estimates~$\estvec_i$ in the right panel. To avoid overplotting, we display a subsample of~$n=10^5$ stars.}\label{fig-cmd}
\end{figure}

To avoid these difficulties, we extend the \cite{kiefer1956consistency} nonparametric maximum likelihood estimator (NPMLE) to incorporate multivariate and heteroscedastic errors. An NPMLE is any~$\npmle\in \Ps(\R^\p)$ which maximizes the marginal likelihood of the observations~$(X_i)_{i=1}^n$. Marginally, the observations are independent, and the~$i^\mathrm{th}$ observation~$X_i$ is distributed according to a Gaussian location mixture with density
\begin{align}\label{eq-mixture-density}
f_{\trueprior, \Sigma_i}(x) \coloneqq \int \varphi_{\Sigma_i}(x-\avec) \diff \trueprior(\avec), \qquad \mbox{for } x \in \R^\p,
\end{align}
where $\varphi_{\Sigma_i}(x) := \frac{1}{\sqrt{\det(2\pi\Sigma_i)}}\exp\left(-\frac{1}{2}x^\mathsf{T}\Sigma_i^{-1}x\right)$ denotes the density of $\epsilon_i \sim\cn(0, \Sigma_i)$. Hence an NPMLE is any maximizer 
\begin{align}\label{eq-NPMLE}
\npmle&\in \argmax_{\aprior\in \Ps(\R^\p)} \ell_n(\aprior), \qquad\mbox{with}\qquad\ell_n(\aprior) \coloneqq \frac{1}{n}\sum_{i=1}^n\log f_{\aprior, \Sigma_i}(X_i).
\end{align}
In contrast to the parametric model used in XD, the nonparametric domain~$\Ps(\R^\p)$ is convex, so~$\npmle$ solves a convex optimization problem, and tools from convex optimization may be leveraged to find principled approximations to~$\npmle$ \citep{koenker2014convex, kim2020fast}. 

Given an estimate~$\npmle$ of the prior~$\trueprior$, empirical Bayes imitates the optimal Bayes analysis, known as the {\sl oracle} \citep{efron2019bayes}. If $\trueprior$ were known, optimal denoising of $\truevec_i$ would be achieved through the posterior distribution $\truevec_i\mid X_i$. It is well known, for instance, that the oracle posterior mean
\begin{align}\label{eq-oracle}\orvec_i \coloneqq \E_{\trueprior}\left[\truevec_i\mid X_i\right], \text{ where } \truevec_i\sim \trueprior \text{ and } X_i\mid\truevec_i\sim \cn(\truevec_i, \Sigma_i)\end{align}
minimizes the squared error Bayes risk $\E_{\trueprior}\|\mathfrak{d}_i(X_i) - \truevec_i\|_2^2$ over {\sl all} measurable $\mathfrak{d}_i : \R^\p\to\R^\p$. The NPMLE~\eqref{eq-NPMLE} yields a fully data-driven, empirical Bayes estimate of the oracle posterior mean via
\begin{align}\label{eq-posterior-mean}
\estvec_i \coloneqq \E_{\npmle}\left[\truevec_i\mid X_i\right], \mathrm{~where~} \truevec_i\sim \npmle \mathrm{~and~} X_i\mid\truevec_i\sim \cn(\truevec_i, \Sigma_i).
\end{align}

Figure~\ref{fig-cmd} shows the~$\p=2$ dimensional data set of~\citet{anderson2018improving}, where each observation has a known error distribution and may be modeled as multivariate normal after a suitable transformation. The noise in the raw CMD of Figure~\ref{fig-cmd} obscures many known features of stellar evolution, rendering the raw CMD unreliable for downstream parallax inference. The right panel of Figure~\ref{fig-cmd} displays the empirical Bayes posterior means~$(\estvec_i)_{i=1}^n$ based on the NPMLE. The substantial shrinkage of our method reveals many recognizable features of the CMD, such as the red clump and a narrow red giant branch in the upper-right region of the plot, as well as the binary sequence tail distinct from the main sequence tail in the bottom-center region. The NPMLE~$\npmle$ and corresponding posterior means~$(\estvec_i)_{i=1}^n$ thus offer a powerful approach to shrinkage estimation under minimal assumptions. In Section~\ref{sec-cmd}, we continue our analysis of the Gaia TGAS data, and in Section~\ref{sec-apogee}, we present another astronomy application for denoising~$d=19$ chemical abundance ratios of $n = 27,238$ stars.\\

\subsection{Applications to hierarchical linear modeling}\label{sec-hierarchical-linear-model}

Another common setting where the heteroscedastic covariances~$\Sigma_i$ are known or can be accurately estimated is in hierarchical linear modeling. In this setting, we observe nested data with~$n$ different groups, where the $i^{\text{th}}$ group contains~$N_i$ units, and each unit~$j\in \{1,\ldots,N_i\}$ has~$\p$ covariates $X_{ij}\in \R^\p$ and a scalar response~$y_{ij}\in \R$. Consider the model
\[
y_{ij} = \langle X_{ij}, \beta^*_i\rangle + \eps_{ij}, \qquad \text{where } \beta^*_i \simiid G^*\text{ and }\eps_{ij} \simiid \cn(0, \sigma^2).
\]
Writing $y_i = (y_{i1},\ldots,y_{iN_i}) \in \R^{N_i}$ and $X_i = (X_{i1},\ldots,X_{iN_i})^\mathsf{T}\in \R^{N_i\times \p}$, we can write the model more compactly as
\begin{align*}
y_i \mid \beta^*_i& \simind \cn(X_i\beta^*_i, \sigma^2I_{N_i}) \qquad\text{where }\beta^*_i \simiid G^*
\end{align*}
for $i=1,\ldots,n$. If $N_i > d$ for all $i$, we can summarize each individual regression problem with the ordinary least squares (OLS) solution~$b_i = (X_i^\mathsf{T}X_i)^{-1}X_i^\mathsf{T}y_i$, which satisfies
\begin{align*}
b_i \mid \beta^*_i &\simind \cn(\beta^*_i, \Sigma_i), \qquad \mbox{with } \beta^*_i \simiid G^*,  \qquad \;\mbox{for } i\in \{1,\dots, n\},
\end{align*}
where $\Sigma_i = \sigma^2(X_i^\mathsf{T}X_i)^{-1}$. Thus, the data $(b_i, \Sigma_i)_{i=1}^n$ follow the multivariate, heteroscedastic model~\eqref{eq-obs-model}, where the corresponding covariance matrices are typically known up to the scalar~$\sigma^2$. If $\sigma^2$ is not known, we can estimate it very accurately with 
\[
\hat\sigma^2 = \frac{1}{\sum_{i=1}^n (N_i-\p)} \sum_{i=1}^n\|y_i - X_ib_i\|_2^2.
\]
When $n$ is large, this will be a reliable estimate of the variance~$\sigma^2$ even if each $N_i$ is relatively close to~$\p$. From here, we proceed as in Section~\ref{sec-denoising-app}, computing the NPMLE $\npmle = \argmax_G \prod_{i=1}^n f_{G, \Sigma_i}(b_i)$ and performing shrinkage on the regression coefficients via the posterior mean 
\[\hat\beta_i := \E_{\npmle}[\beta^*_i\mid b_i],\] 
treating our estimate~$\npmle$ as the prior. 
A common assumption, for instance in meta-analysis \citep{dersimonian1986meta}, is that~$\trueprior$ is itself a Gaussian~$\trueprior = \mathcal{N}(\mu^*, V^*)$, in which case the posterior mean shrinks all~$b_i$ towards the prior mean~$\mu^*$. Our model of course allows for this possibility, but it also allows for much more complex structures~$\trueprior$, which perform nonlinear shrinkage and need not shrink all observations towards a single point. We apply the NPMLE to hierarchical linear models in two settings, an education data set (Section~\ref{sec-math-scores}) and a microarray data set (Section~\ref{sec-leukemia}).

\subsection{Computational and statistical properties}

The idea of using the NPMLE to estimate a prior distribution, due to \citet{robbins1950generalization}, has seen a resurgence in recent years \citep{jiang2009general, jiang2010empirical, koenker2014convex, dicker2016high, gu2016problem, koenker2017rebayes, feng2018approximate, kim2020fast, efron2019bayes, saha2020nonparametric, jiang2020general, polyanskiy2020self, deb2021two}. These advancements, taken together, have begun to establish the NPMLE as a formidable approach to shrinkage estimation both in theory and in practice. All this prior work has focused on either the univariate setting~$\p=1$ or the homoscedastic setting~$\Sigma_i\equiv \Sigma$, however. Our work extends the NPMLE to the practically important and more general setting of multivariate and heteroscedastic errors, uncovering a number of important differences.

Basic properties of the NPMLE that are well-understood in the univariate, homoscedastic setting~\citep{lindsay1995mixture} have not received careful attention in more complex settings. We verify in Lemma~\ref{lem-characterization} that a solution~$\npmle$ exists for every instance of the optimization problem~\eqref{eq-NPMLE}, and we record the first-order optimality conditions characterizing the solution set. Similar to the univariate, homoscedastic setting, there exists a solution~$\npmle$ which is discrete with at most~$n$ atoms, and the sequence of fitted values~$\hat{L} \equiv (\hat{L}_1,\dots,\hat{L}_n)= (f_{\npmle, \Sigma_i}(X_i))_{i=1}^n$ is unique, i.e. every solution~$\npmle$ has the same sequence of fitted likelihood values~$\hat{L}$. 

An important contribution of Lemma~\ref{lem-characterization} is our reinterpretation of the characterizing system of inequalities in terms of a natural `dual' mixture density~$\widehat\psi_n$. Specifically,~$\widehat\psi_n$ is a heteroscedastic, $n$-component mixture density---a convex combination of Gaussian bumps centered at the data points~$\cn(X_i, \Sigma_i)$ with weights inversely proportional to~$\hat{L}_{i}$ for~$i=1,\ldots, n$---such that the support of every NPMLE~$\npmle$ is contained in the set of the global maximizers of~$\widehat\psi_n$. This observation has a number of important consequences that we explore in detail in Section~\ref{sec-characterization}; in particular, tools from algebraic statistics for studying the modes of Gaussian mixtures~\citep{ray2005topography, amendola2020maximum} translate directly into results on the support set. We leverage this connection to establish that~$\npmle$ is not necessarily unique when~$\p > 1$, even in the homoscedastic case. This finding is distinctive from the univariate, homoscedastic case where it is known that~\eqref{eq-NPMLE} has a unique solution for every problem instance~\citep{lindsay1993uniqueness}. Our counterexample in Lemma~\ref{rem-nonunique} appears to be new and seems to invalidate prior claims of strict concavity of the log-likelihood \citep{marriott2002local, koenker2017rebayes}. Whereas the fitted values~$\hat{L}$ are always unique, our counterexample also demonstrates that the empirical Bayes posterior means~$(\estvec_i)_{i=1}^n$ are not necessarily unique. 

The problem of computing a solution~$\npmle$ is complicated by the presence of multivariate, heteroscedastic errors. The main difficulty in general is that the NPMLE solves an infinite-dimensional optimization problem. Since~$\npmle$ may be taken to be discrete with at most~$n$ atoms, a solution can in principle be found with a finite mixture model. In particular, defining the set of discrete distributions with at most $k\ge 1$ atoms, 
\[\Ps_k(\R^\p) = \left\{\sum_{j=1}^kw_j\delta_{a_j} : \sum_j w_j = 1, w \ge 0, a_j\in \R^\p, j=1, \ldots, k\right\},\]
maximum likelihood solutions over~$\Ps_n(\R^\p)$ are also NPMLEs. Hence, the EM algorithm can be applied to optimize~$(w_j, a_j)_{j=1}^n$, as first observed by \cite{laird1978nonparametric}, though EM over discrete distributions is prohibitively slow for moderately large~$n$ and suffers from the same non-convexity issue as XD. Many algorithms \citep{bohning1985numerical, lesperance1992algorithm, wang2007fast, liu2007partially} have been proposed for finding approximate solutions to the optimization problem~\eqref{eq-NPMLE}; \cite{koenker2014convex} identified a convex, finite-dimensional, highly scalable approximation. Instead of maximizing the log-likelihood of the data $\ell_n(G) = \frac{1}{n}\sum_{i=1}^n\log f_{\aprior, \Sigma_i}(X_i)$ over all probability measures $\aprior\in \Ps_n(\R^\p)$, the idea is to maximize the log-likelihood over~$\Ps(\ca)$, the collection of all probability measures supported on a finite set~$\ca\subset\R^\p$. If $\ca$ has $m > 0$ elements, then $\Ps(\ca)$ is isometric to the $m-1$ dimensional simplex $\Delta_{m-1} \coloneqq \{w \in \R_+^m : \sum_j w_j=1\}$, and maximizing the likelihood corresponds to optimizing over the mixing proportions~$w$, which is a convex optimization problem. When~$\p=1$, it is straightforward to see that~$\npmle$ is supported on the range of the data~$[X_{(1)}, X_{(n)}]$, so \cite{koenker2014convex} proposed taking~$\ca$ to discretize this range. \citet[][Proposition 5]{jiang2009general} bounded the discretization error in $\p=1$ dimension, establishing that optimizing the weights $w$ via EM can lead to a good approximation once $m\asymp (\log n)\sqrt{n}$. \cite{dicker2016high} further justified the discretization scheme in~$\p=1$ dimension by showing the discretized NPMLE is statistically indistinguishable from~$\npmle$ once the analyst uses at least~$m=\lfloor\sqrt{n}\rfloor$ atoms.

The discretization approach naturally extends to multivariate, heteroscedastic settings, but to our knowledge, no principled recommendations are available for choosing~$\ca\subset\R^\p$ in general. \cite{feng2018approximate} recommended taking~$\ca$ to be a grid over a compact region containing the data. We address the key questions of how to choose this compact region and how the discretization error depends on the fineness of the grid. For choosing a compact region to discretize, a natural desideratum is that the region should contain the support of~$\npmle$. To this end, in Proposition~\ref{prop-support} we present compact support bounds on the NPMLE in terms of the data~$(X_i, \Sigma_i)_{i=1}^n$. When~$\p=1$ our support bounds reduce to the range of the data, reaffirming the original suggestion of~\cite{koenker2014convex}, and when $\p>1$ but the errors are homoscedastic, it suffices to discretize the convex hull of~$(X_i)_{i=1}^n$. Interestingly, with multivariate and heteroscedastic errors, the support of the NPMLE can lie outside the convex hull of~$(X_i)_{i=1}^n$, so a different region known as the \textit{ridgeline manifold}~$\cm$ of $(X_i, \Sigma_i)_{i=1}^n$ is needed. Fortunately, this region $\cm\subset\R^\p$ is compact, and the NPMLE over $\Ps(\cm)$ agrees with the NPMLE over~$\Ps(\R^\p)$. This justifies the choice of~$\ca$ as a $\delta > 0$ cover of~$\cm$, and in Proposition~\ref{prop-approximation}, we verify that as $\delta\downarrow 0$, the log-likelihood of the discretized NPMLE approaches that of the NPMLE. We prove a quantitative bound on the gap for fixed~$\delta$, providing some guidance on how the discretization error depends on the fineness of the grid. {In Section~\ref{sec-discretization-for-statistical-purposes}, we connect these computational guarantees to our statistical guarantees, showing that it suffices to discretize~$\cm$ at a resolution~$\delta\asymp \sqrt{\frac{(\log n)^\p}{n}}$ in order to find an approximate solution that satisfies all of our statistical guarantees.}

{The standard gridding approach clearly suffers from the curse of dimensionality: if~$\ca$ is a~$\delta$-cover of the ridgeline manifold~$\cm$, the number of points in~$\ca$ typically scales like~$\left(\frac{1}{\delta}\right)^d$, so it is often difficult to compute the NPMLE over a fine grid even when~$\p=3$. An alternative---the \textit{exemplar method} proposed by~\cite{lashkari2008convex} (see also~\citet{bohning1999computer})---is to take the finite set~$\ca$ to be all of the observations~$\{X_i\}_{i=1}^n$. In Section~\ref{sec-exemplar}, we show some toy examples illustrating that the exemplar method can produce grossly suboptimal solutions, since the set~$\ca$ is not a very fine cover of the ridgeline manifold~$\cm$. We propose an extension, called the \textit{exemplar+ method}, to systematically sample additional points from the ridgeline manifold~$\cm$. We show through toy examples and simulation experiments that our proposal consistently produces approximate NPMLEs satisfying the assumptions of our theoretical guarantees, even in~$\p = 20$ dimensions.}

Our principled and efficient methods of computation facilitate simulation studies assessing the performance of the empirical Bayes estimate~$\estvec_i$ in a setting where we can actually compare to the oracle Bayes estimate~$\orvec_i$. Figure~\ref{fig-sim} illustrates the method on simulated data. The means $\truevec_i$ were drawn i.i.d.~from a circle of radius two, and the data $X_i\mid \truevec_i$ were drawn according to~\eqref{eq-obs-model} using a variety of diagonal covariance matrices $\Sigma_i = \begin{bmatrix}\sigma_{1,i}^2 & 0 \\ 0 & \sigma_{2,i}^2\end{bmatrix}$, taking each $\sigma_{j,i}^2 \in (1/2, 3/4)$. Visually, it is clear that the empirical Bayes estimates improve upon the observations by shrinking towards the underlying circle; the corresponding mean squared errors were $\frac{1}{n}\sum_{i=1}^n\|\estvec_i - \truevec_i\|_2^2 = 0.87$ and $\frac{1}{n}\sum_{i=1}^n\|X_i - \truevec_i\|_2^2 = 1.46$, respectively. The oracle, which minimizes the mean squared error in expectation, attained an error of $\frac{1}{n}\sum_{i=1}^n\|\orvec_i - \truevec_i\|_2^2 = 0.84$. While the oracle cannot be computed in practice because~$\trueprior$ is unknown, this value sets a benchmark in simulations to which we may compare the performance of bona fide estimators. The empirical Bayes estimates not only track well with this benchmark; the individual estimates also track remarkably well with the oracle. In our simulation, the {\sl regret}---defined as the mean squared error between the estimator~$(\estvec_i)_{i=1}^n$ and the oracle~$(\orvec_i)_{i=1}^n$---was $\frac{1}{n}\sum_{i=1}^n\|\estvec_i - \orvec_i\|_2^2 = 0.03$. Whereas $\estvec_i$ is a function of the observed data, the oracle~$\orvec_i$ makes optimal use of the unknown prior~$\trueprior$, making the similarity between the two especially striking. 

\begin{figure}[t!]
\centering
\includegraphics[width=.9\textwidth]{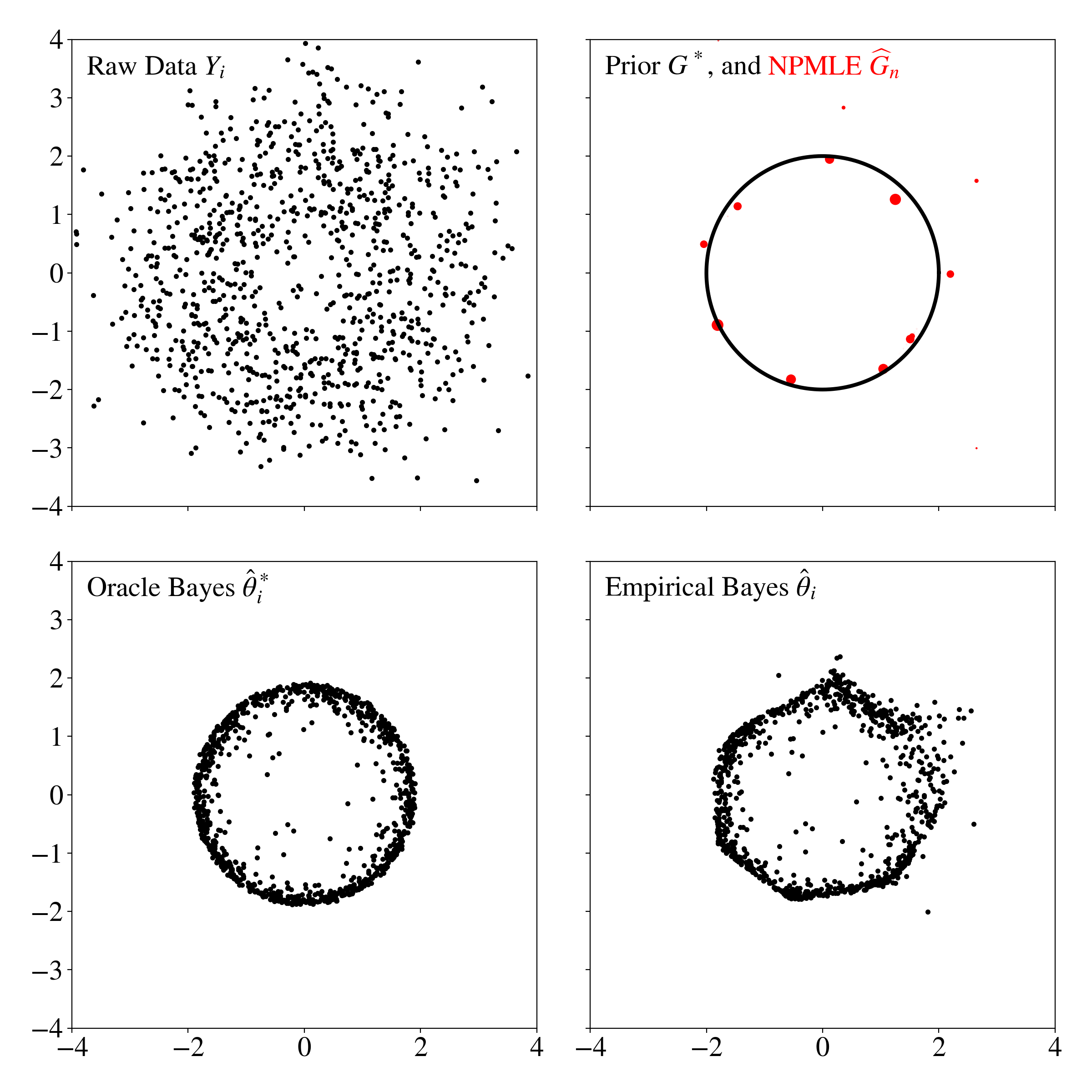} 
\caption{Toy data of size $n=1,000$ and $\p=2$. Top: observations $X_i$ (left) were generated by adding heteroscedastic Gaussian errors to the underlying means $\truevec_i\simiid \trueprior$ (right), generated i.i.d. uniformly from a circle of radius $2$. Our discrete estimate~$\npmle$ of the prior is shown in red over the prior~$\trueprior$ in black. Bottom: a comparison of oracle Bayes $\orvec_i$ (left) based on knowledge of the prior distribution $\trueprior$ and empirical Bayes $\estvec_i$ (right), a function of the observed data.}\label{fig-sim}
\end{figure}

This striking similarity between $\estvec_i$ and $\orvec_i$ affirms the empirical Bayes adage that ``{\sl large data sets of parallel situations carry within them their own Bayesian information}'' \citep{efron2016computer}. However, the setting of Figure~\ref{fig-sim} is complicated by the fact the situations are not directly parallel, in that each observation $X_i$ has a distinct error distribution. Even in heteroscedastic settings, the extent to which we glean prior information for the purpose of denoising is captured by the empirical Bayes regret $\frac{1}{n}\sum_{i=1}^n\|\estvec_i - \orvec_i\|_2^2$. Theorem~\ref{thm-denoising} develops a detailed profile of the finite-sample regret properties of the NPMLE for denoising. We show that under certain tail conditions on~$\trueprior$ the regret is bounded by a rate that is nearly parametric in~$n$, i.e.~$\frac{1}{n}$ up to logarithmic multiplicative factors. The regret still converges at a slower, nonparametric rate under less structured conditions, where~$\trueprior$ may have heavy tails. Furthermore, when~$\trueprior$ possesses finer structure, such as the clustering problem where~$\trueprior$ is a discrete measure with~$k^*$ atoms, we prove that the regret is bounded from above by~$\frac{k^*}{n}$ up to logarithmic multiplicative factors in~$n$. The clustering case is particularly remarkable, as the NPMLE is completely tuning-free, with no knowledge of~$k^*$, yet~$\npmle$ performs essentially as well as any estimator which knows the number of clusters~$k^*$. Thus, Theorem~\ref{thm-denoising} demonstrates that the NPMLE effectively discovers structure when available and also effectively learns when structure is unavailable. Theorem~\ref{thm-denoising} generalizes the regret bounds of \citet{saha2020nonparametric} and \citet{jiang2020general} who analyzed the homoscedastic~$\Sigma_i\equiv \Sigma$ setting and the univariate~$\p=1$ setting, respectively. These papers in turn built upon \cite{jiang2009general} who studied the univariate, homoscedastic setting.

A key ingredient in the analysis of the regret is a more explicit representation of the estimator~$(\estvec_i)_{i=1}^n$ and oracle~$(\orvec_i)_{i=1}^n$. The oracle posterior mean~\eqref{eq-oracle} has the following alternative expression, known as Tweedie's formula \citep{dyson1926method, robbins1956proceedings, efron2011tweedie, banerjee2023nonparametric}:
\begin{align}\label{eq-oracle-tweedie}
\orvec_i 
&= X_i + \Sigma_i\frac{\nabla f_{\trueprior, \Sigma_i}(X_i)}{f_{\trueprior, \Sigma_i}(X_i)}.
\end{align}
Similarly, our plug-in estimate can be written as
\begin{align}\label{eq-gmleb-tweedie}
\estvec_i 
&= X_i + \Sigma_i\frac{\nabla f_{\npmle, \Sigma_i}(X_i)}{f_{\npmle, \Sigma_i}(X_i)}.
\end{align}
Tweedie's formula clarifies that under model~\eqref{eq-obs-model} the posterior means only depend on the prior~$\trueprior$ via the marginal likelihood~$f_{\trueprior, \Sigma_i}(X_i)$ and its gradient. \cite{jiang2009general} first leveraged this observation to relate the empirical Bayes regret to the problem of estimating the marginal density. In heteroscedastic problems, there are~$n$ different marginal densities,~$(f_{\trueprior, \Sigma_i})_{i=1}^n$, to estimate, and~{$n$} corresponding estimators~$(f_{\npmle, \Sigma_i})_{i=1}^n$. We show in Theorem~\ref{thm-density-estimation} and Corollary~\ref{cor-density-estimation} that the NPMLE achieves similar adaptive rates in the density estimation problem under an appropriate average Hellinger distance across all~$i=1, \ldots,n$ estimands~$(f_{\trueprior, \Sigma_i})_{i=1}^n$. 

Whereas most recent work has focused on properties of~$\npmle$ for density estimation and denoising, the NPMLE is potentially much more generally applicable as a plug-in estimate of the prior. To expand our understanding of its applicability, we present the first analysis of the deconvolution error for the NPMLE. Whereas density estimation captures the problem of describing the observations~$(X_i)_{i=1}^n$, deconvolution is the equally natural problem of interpreting the infinite-dimensional parameter~$\trueprior$. We study the accuracy of the NPMLE under a Wasserstein distance $W_2(\npmle, \trueprior)$. The Wasserstein distance is particularly useful for this problem since~$\npmle$ and~$\trueprior$ are typically mutually singular; in particular,~$\trueprior$ may be absolutely continuous whereas~$\npmle$ is always discrete. The Wasserstein distance will be discussed in detail in Section~\ref{sec-stats}. We show in Theorem~\ref{thm-deconvolution} that~$\npmle$ attains the minimax rate of deconvolution, which happens to be a very slow, logarithmic rate~$\frac{1}{\log n}$. Inspired by the richness of the density estimation and denoising results, we hint at some of the adaptation properties of the NPMLE under the Wasserstein loss; Theorem~\ref{thm-deconvolution-point-mass} shows that when~$\trueprior = \delta_\mu$ is a point mass distribution, the Wasserstein rate improves dramatically to~$n^{-1/4}$ up to logarithmic factors.

The rest of the paper is organized as follows: Section~\ref{sec-characterization} systematically addresses basic properties of the NPMLE, including existence and non-uniqueness. Section~\ref{sec-computation} gives a full account of the approximate computation of NPMLEs. Section~\ref{sec-stats} establishes finite-sample risk bounds on the accuracy of~$\npmle$ as an estimator of~$\trueprior$ for the purposes of density estimation, denoising and deconvolution. In Section~\ref{sec-app}, we apply the method to astronomy data to construct a fully data driven color-magnitude diagram of $1.4$ million stars and compare our method to extreme deconvolution where it has previously been applied \citep{anderson2018improving}. We also apply the method to chemical abundance data for a smaller subset of stars that has previously been analyzed by \citet{ratcliffe2020tracing}. {We then present two applications using the hierarchical linear model, one on education data and the other on microarray data of leukemia patients.} Section~\ref{sec-discussion} concludes with some discussion of future work. The proofs {and some simulation experiments} are in the Supplementary Material. All figures and analyses in this manuscript can be reproduced using code at the following repository: \href{https://github.com/jake-soloff/Multivariate-Heteroscedastic-NPMLE}{https://github.com/jake-soloff/Multivariate-Heteroscedastic-NPMLE}. \noindent{An open source software package for computing approximate NPMLEs and empirical Bayes estimates in Python is available at: \href{https://github.com/jake-soloff/npeb}{https://github.com/jake-soloff/npeb}.\\}\label{code-npeb}

\section{Structural properties}\label{sec-characterization} 

In this section, we establish some basic properties of solutions to the nonparametric maximum likelihood problem~\eqref{eq-NPMLE}, including existence, non-uniqueness, invariance under certain transformations, and bounds on the support. These results provide a foundation both for computing~$\npmle$ (Section~\ref{sec-computation}) and for understanding its statistical properties (Section~\ref{sec-stats}). Our first result extends the well-known characterization of~$\npmle$ for univariate, homoscedastic errors \citep[][Theorems~18-21]{lindsay1995mixture} to our more general setting. 

\begin{lemma}\label{lem-characterization} Problem~\eqref{eq-NPMLE} attains its maximum: there exists a discrete solution~$\npmle$ with at most~$n$ atoms, and the vector $\hat{L} \equiv (\hat{L}_1,\dots,\hat{L}_n)= (f_{\npmle, \Sigma_i}(X_i))_{i=1}^n$ of fitted likelihood values is unique. Moreover, $\npmle\in \Ps(\R^\p)$ solves~\eqref{eq-NPMLE} if and only if 
\[
D(\npmle, \vartheta) \le 0\text{ for all }\vartheta\in \R^\p\text{, where }D(\aprior, \vartheta) \coloneqq \frac{1}{n}\sum_{i=1}^n\frac{\varphi_{\Sigma_i}(X_i-\vartheta)}{f_{\aprior,\Sigma_i}(X_i)} -1.
\]
The support of any $\npmle$ is contained in the zero set $\mathcal{Z} \coloneqq \{\vartheta : D(\npmle, \vartheta) = 0\}$; the zero set $\mathcal{Z}$ is equal to the set of global maximizers of the $n$-component, heteroscedastic {\sl dual} mixture density
\[\widehat\psi_n(\vartheta) \coloneqq \sum_{i=1}^n \left(\frac{\hat{L}_i^{-1}}{\sum_{\iota=1}^n\hat{L}_\iota^{-1}}\right)\varphi_{\Sigma_i}(X_i-\vartheta).\]
\end{lemma}

We prove Lemma~\ref{lem-characterization}, along with all results in this section, in Section~\ref{sec-proofs-characterization} of the Supplementary Material. The first statement of the lemma guarantees the existence of a discrete solution, which we typically write as~$\npmle = \sum_{j=1}^{\khat}\hat{w}_j\delta_{\hat{a}_j}$ (here $\hat{w}_j\ge 0$, $\sum_j\hat{w}_j = 1$ and $\hat{a}_j\in \R^\p$), with $\khat\le n$ providing an upper bound on the complexity of at least one solution. This implies that~$\npmle$ may be taken to be the maximum likelihood solution to a $\khat$-component, heteroscedastic Gaussian mixture model where $\khat$ is selected in a data dependent manner. Since finite mixture models are nested by the number of components and $\khat\le n$, we may also say in general that~$\npmle$ is the maximum likelihood solution to an $n$-component, heteroscedastic Gaussian mixture model. 

The bound $\khat\le n$ is tight: for each $n\ge 1$, there are sequences of observations $(X_i)_{i=1}^n$ and covariances $(\Sigma_i)_{i=1}^n$ such that the smallest number of components $\khat$ of any solution~$\npmle$ to~\eqref{eq-NPMLE} is precisely $n$ \citep[see, e.g.,][p. 116]{lindsay1995mixture}. However, in practice, the number of components is typically much smaller than $n$. For instance, in the univariate, homoscedastic case, \citet{polyanskiy2020self} established a much stronger bound of $\khat = O_P(\log n)$ under certain conditions on the prior distribution~$\trueprior$. 

The last part of Lemma~\ref{lem-characterization} states that the atoms of~$\npmle$ occur at the global maximizers of the $n$-component Gaussian mixture $\widehat\psi_n$, which has component distributions of the form $\cn(X_i, \Sigma_i)$ for $i=1,\dots,n$ with weights inversely proportional to fitted likelihoods $\hat{L}$. Results on the modes of Gaussian mixtures~\citep[e.g.][]{ray2005topography, dytso2019capacity, amendola2020maximum} thus provide information about the support of the NPMLE; in particular, our next two results exploit this connection to yield novel results on the NPMLE.

In the univariate $\p=1$ and homoscedastic setting $\Sigma_i\equiv \sigma^2$, it is additionally known that~\eqref{eq-NPMLE} has a {\sl unique} solution $\npmle$ for all observations $X_1,\dots, X_n$ \citep{lindsay1993uniqueness}. This means that, for every data set $X_1,\dots, X_n$ and every variance level~$\sigma^2 > 0$, there is a unique probability measure~$\npmle\in \Ps(\R)$ such that $\hat{L}_i = f_{\npmle, \sigma^2}(X_i)$ for all $i$, where $\hat{L}$ is the unique vector of optimal likelihoods from Lemma~\ref{lem-characterization}. We observe, however, that uniqueness of the solution~$\npmle$ may not hold when $\p > 1$, even with isotropic covariances $\Sigma_i\equiv \sigma^2I_\p$. 

\begin{lemma}\label{rem-nonunique} Let $\p=2$, $n=3$ and $X_1 = (0, 1)$, $X_2 = (\frac{\sqrt{3}}{2}, -\frac{1}{2})$,  $X_3 = (-\frac{\sqrt{3}}{2}, -\frac{1}{2})$. Then~\eqref{eq-NPMLE} with data $(X_i)_{i=1}^3$, covariances $\Sigma_i\equiv \sigma^2I_2$ and $\sigma^2 = 3/(\log 256)$ has infinitely many solutions of the form \[\npmle = \alpha \delta_0 + (1-\alpha)\frac{1}{3}\sum_{i=1}^3\delta_{X_i/2}\]
where $\alpha\in[0,1]$.
\end{lemma}

Figure~\ref{fig-contours} illustrates the counterexample given in Lemma~\ref{rem-nonunique}. A key observation in the proof of Lemma~\ref{rem-nonunique} is that the dual mixture $\widehat\psi_n= f_{H, \sigma^2I_2}$ can be written explicitly as a homoscedastic mixture with uniform mixing distribution $H = \frac{1}{3}\sum_{i=1}^3\delta_{X_i}$ over the observations $(X_i)_{i=1}^3$. This set-up closely follows a construction, due to Duistermaat \citep[see][]{amendola2020maximum}, exhibiting an isotropic, homoscedastic Gaussian mixture with more modes than components. Duistermaat used the same component locations $X_i$ but took $\sigma^2 = 0.53$ to obtain an example of a three-component mixture of isotropic, homoscedastic Gaussians such that the mixture has four modes. By specifically choosing $\sigma^2 = \frac{3}{\log 256} \approx 0.54$, the height of the mixture $\widehat\psi_n = f_{H, \sigma^2I_2}$ is equal at all four modes, i.e. all four modes are global maximizers, and the modes are located at $\{X_1/2, X_2/2, X_3/2, 0\}$. By Lemma~\ref{lem-characterization} any NPMLE must be supported on these modes. Representing the fitted values~$\hat{L} = (f_{\npmle, \sigma^2I_2}(X_i))_{i=1}^3$ by a probability measure~$\npmle= \sum_{j=1}^3\hat{w}_j\delta_{X_j/2} + \hat{w}_4\delta_{0}$ supported on the global modes is equivalent to finding a set of weights $\hat{w}\in \R_+^4$ such that $\sum_{j=1}^4\hat{w}_j = 1$ and $\hat{w}$ solves the under-determined linear system $\hat{L} = A\hat{w}$,
where $A$ is a $3\times 4$ matrix given by
\[
A_{ij} = \begin{cases}\varphi_{\sigma^2I_2}(X_i - X_j/2) & j\le 3 \\
\varphi_{\sigma^2I_2}(X_i) & j=4.
\end{cases}
\]
Finally, we also note that although the fitted likelihoods $f_{\npmle, \sigma^2 I_2}(X_i)$ are unique, the posterior means~$\estvec_i$ in this example differ for the solutions~$\npmle$ given in Lemma~\ref{rem-nonunique}.

\begin{figure}[ht!]
\centering
\includegraphics[width=0.5\textwidth]{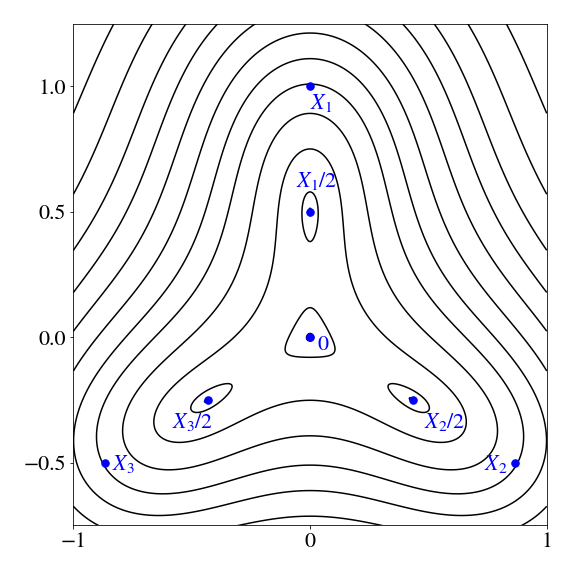}
\caption{Level sets of the dual mixture density $\widehat\psi_n = f_{H, \sigma^2I_2}$ where $n=3$ and $H = \frac{1}{3}\sum_{i=1}^3 \delta_{X_i}$ is uniform over the vertices of the larger equilateral triangle $\triangle X_1X_2X_3$. With $\sigma^2 = \frac{3}{\log 256}$, the dual mixture density~$\widehat\psi_n$ has four global modes.}\label{fig-contours} 
\end{figure}

Although the NPMLE searches over all probability measures~$\aprior\in \Ps(\R^\p)$ supported on~$\R^\p$, it is useful algorithmically to reduce the search space to probability measures supported on a compact subset of~$\R^\p$. By Lemma~\ref{lem-characterization}, to restrict the support of the NPMLE it suffices to bound the maximizers $\cz$ of the $n$-component Gaussian mixture $\widehat\psi_n$. \citet[Theorem~1]{ray2005topography} showed that all critical points of a Gaussian mixture \blue{can be constrained to a compact set known as the \emph{ridgeline manifold}. We record this in the following lemma.}

\blue{
\begin{lemma}\label{lem-ridgeline} 
The support of every solution~$\npmle$ to~\eqref{eq-NPMLE} is contained in the ridgeline manifold $\cm$, a compact subset of~$\R^\p$ defined as
\begin{equation}\label{eq-ridgeline}
\begin{aligned}
\cm 
&\coloneqq \left\{x^*(\alpha) : \alpha\in \Delta_{n-1}\right\}, \text{ where } \\ 
x^*(\alpha)&\coloneqq\left(\sum_{i=1}^n \alpha_i \Sigma_i^{-1}\right)^{-1}\sum_{i=1}^n \alpha_i \Sigma_i^{-1}X_i.
\end{aligned}
\end{equation} 
For any $\alpha\in \Delta_{n-1}$ there is an $\alpha'\in  \Delta_{n-1}$ with at most $d+1$ nonzeros such that $x^*(\alpha) = x^*(\alpha')$.
\end{lemma}
}

\blue{The ridgeline manifold~$\cm$ constrains the support of the NPMLE to a compact subset. This set is much larger than the set of critical points~$\cz$ in Lemma~\ref{lem-characterization}, but crucially its definition does not depend on the weights~$\left(\frac{\hat{L}_i^{-1}}{\sum_{\iota=1}^n\hat{L}_\iota^{-1}}\right)_{i=1}^n$ in the dual mixture~$\widehat\psi_n$.} In the univariate case~$\p=1$, the rigdeline manifold~$\cm = [X_{(1)}, X_{(n)}]$ is simply the range of the data, so the univariate NPMLE is constrained to be supported on this range. In the multivariate setting, we may further simplify $\cm$ depending on certain shape restrictions on the covariance matrices. 

\begin{prop}\label{prop-support} Depending on the values of $(\Sigma_i)$ we further bound the support of~$\npmle$ as follows:
\begin{enumerate}
\item (Homoscedastic) If $\Sigma_i= \Sigma$ for all $i$, or if $\Sigma_i = c_i\Sigma$ are proportional up to a sequence $(c_i)$ of positive scalars, the ridgeline manifold~$\cm$ is the convex hull of the data $\mathrm{conv}(\{X_1,\dots, X_n\})$. 
\item (Diagonal Covariances) If $\Sigma_i$ is a diagonal matrix for every $i$, the ridgeline manifold~$\cm$ is contained in the axis-aligned minimum bounding box of the data 
\[\prod_{j=1}^\p\left[\min_{i\in \{1, \dots, n\}} X_{ij}, \max_{i\in \{1, \dots, n\}} X_{ij}\right],\]
where $X_i = (X_{i1},\dots,X_{i\p})$ for all $i$.
\item (General Covariances) Let $\overline{k}\ge \underline{k}>0$ be chosen such that $\underline{k}I_\p\preceq\Sigma_i\preceq \overline{k}I_\p$ for all $i$, where $A\preceq B$ means $B-A$ is a symmetric positive semidefinite matrix. Choose $r > 0$ and $x_0\in \R^\p$ such that $\|X_i - x_0\|_2\le r$ for all $i$. Then the ridgeline manifold~$\cm$ is contained in the ball
\[
\mathbb{B}_{\kappa r}(x_0) \coloneqq \left\{y\in \R^\p : \|y-x_0\|_2\le \kappa r\right\}
\]
where $\kappa = \overline{k}/\underline{k}$.
\end{enumerate}
\end{prop}

The first part of Proposition~\ref{prop-support} in general gives the smallest possible convex body over which the support of $\npmle$ can be constrained independently of $\{\Sigma_i\}$. To see that the first part is tight, consider a fixed set of observations $(X_i)_{i=1}^n$ and isotropic covariance matrices $\Sigma = \sigma^2 I_\p$; as $\sigma$ is made arbitrarily small, the support of $\npmle$ approaches the set of observations~$(X_i)_{i=1}^n$ \citep{lindsay1995mixture}. Therefore, in general the convex hull is the smallest convex body containing the support in the homoscedastic setting and more generally the setting of proportional covariance matrices. By contrast, the convex hull of the data is in general too small to capture the support of~$\npmle$ in the heteroscedastic setting. Figure~\ref{fig-bounding-box} presents one example with diagonal covariances where the support of~$\npmle$ is pushed towards the corners of the minimum axis-aligned bounding box of the data. Thus, the above discussion and Figure~\ref{fig-bounding-box} indicate that both parts $(a)$ and $(b)$ of Proposition~\ref{prop-support} give the tightest possible convex support bounds in their respective special cases.

\begin{figure}[ht!]
\centering
\includegraphics[width=0.4\textwidth]{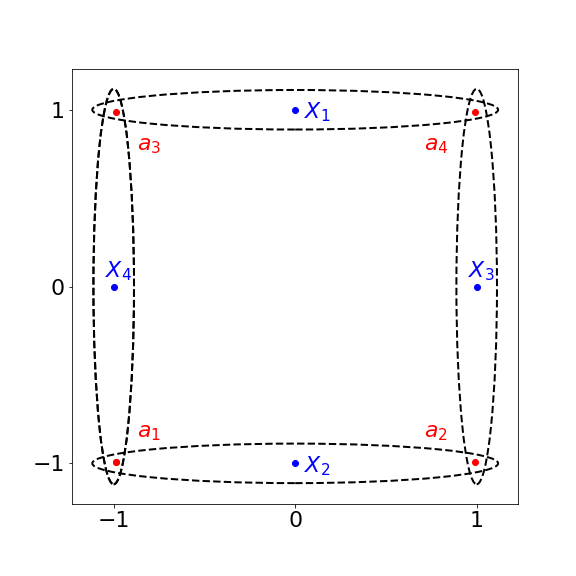}
\includegraphics[width=0.5\textwidth]{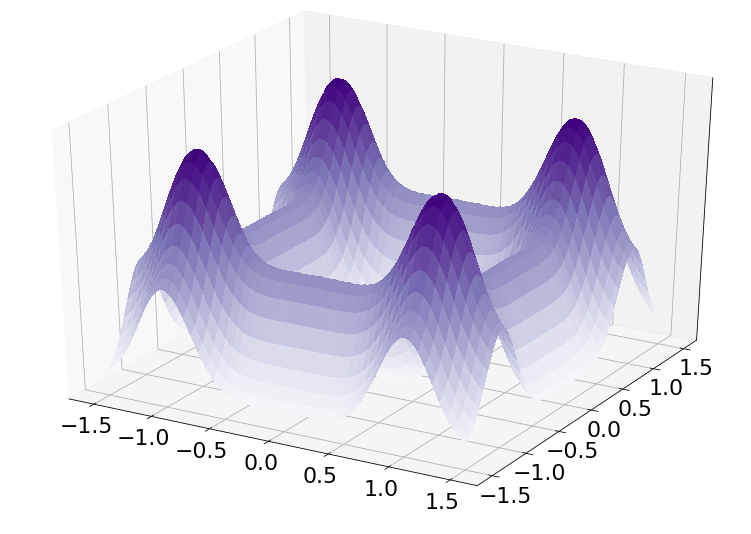}
\caption[An example of observations (blue points) with diagonal covariances (dashed ellipses) where the NPMLE is supported on atoms (red points)  well outside the convex hull of the data, and near the corners of the minimum axis-aligned bounding box. ]{Left: An example of observations (blue points) $X_1 = (0, 1)$, $X_2 = (0, -1)$, $X_3 = (1, 0)$, and $X_4  = (-1, 0)$ 
with diagonal covariances (dashed ellipses) $\Sigma_1=\Sigma_2 = \protect\begin{bmatrix}5&0\\0&.05\protect\end{bmatrix}$  and $\Sigma_3=\Sigma_4 = \protect\begin{bmatrix}.05&0\\0&5\protect\end{bmatrix},$ where the NPMLE is supported on atoms (red points) $a_1,\dots,a_4$ well outside the convex hull of the data, and near the corners of the minimum axis-aligned bounding box. Right: The mixture $\widehat\psi_n(\vartheta) = \frac{1}{4}\sum_{i=1}^4 \varphi_{\Sigma_i}(X_i - \vartheta)$ only has modes at the atoms $a_1,\dots,a_4$, so no NPMLE is supported within the convex hull of the data.}\label{fig-bounding-box} 
\end{figure}

We close this section with a brief discussion on how the NPMLE behaves under certain simple transformations of the data $(X_i, \Sigma_i)_{i=1}^n$. Given a map $T : \R^\p \to \R^\p$, let $T_\#\aprior\in \Ps(\R^\p)$ denote the pushforward of $\aprior\in\Ps(\R^\p)$ given by $T_\#\aprior(B) = \aprior(T^{-1}(B))$, for any Borel set $B\subseteq\R^\p$. In other words, if~$V\sim\aprior$, then $T_\#\aprior$ is the distribution of~$T(V)$.

\begin{lemma}\label{lem-invariance} Fix a data set $(X_i, \Sigma_i)_{i=1}^n$, a point $x_0\in\R^\p$ and a $\p\times\p$ orthogonal matrix $U_0$. Consider the transformed data set $(X_i', \Sigma_i')_{i=1}^n$ where $\Sigma_i' = U_0\Sigma_iU_0^\mathsf{T}$ and $X_i' = T(X_i)$ for $i=1,\dots, n$, with $T(x) = U_0x + x_0$. Then \[f_{T_\#\aprior, \Sigma_i'}(X_i') = f_{\aprior, \Sigma_i}(X_i)\] for all $i=1,\dots, n$ and all $\aprior\in \Ps(\R^\p)$.
\end{lemma}

Lemma~\ref{lem-invariance} is a straightforward consequence of the change of variables formula, but it has a number of useful corollaries. In particular, if $\npmle\in \Ps(\R^\p)$ is an NPMLE for the data set $(X_i, \Sigma_i)_{i=1}^n$, then $T_\#\npmle$ is an NPMLE for the modified data set $(X_i', \Sigma_i')_{i=1}^n$, and the fitted likelihood values are the same, i.e. 
\[
f_{T_\#\npmle, \Sigma_i'}(X_i') = f_{\npmle, \Sigma_i}(X_i),
\]
for all $i=1,\dots, n$. Thus, an NPMLE~$\npmle= \sum_{j=1}^{\khat} \hat{w}_j\delta_{\hat{a}_j}$ is equivariant under translations $T(y) = y+x_0$: if every observation is shifted by some fixed $x_0\in \R^\p$, then the modified NPMLE $T_\#\npmle= \sum_{j=1}^{\khat} \hat{w}_j\delta_{\hat{a}_j + x_0}$ simply shifts every atom by $x_0$. Similarly, the NPMLE is equivariant under orthogonal transformations, which explains why the fitted likelihood values are all equal in the rotationally symmetric toy data sets presented in Figure~\ref{fig-contours} and Figure~\ref{fig-bounding-box}.

\section{Computation}\label{sec-computation}

The NPMLE solves a convex optimization problem~\eqref{eq-NPMLE} that is {\sl infinite-dimensional} in the sense that the decision variable $\aprior$ ranges over all probability measures on $\R^\p$. Many numerical methods for approximately computing the NPMLE have been considered---including EM \citep{laird1978nonparametric}, vertex direction and exchange methods \citep{bohning1985numerical}, semi-infinite methods \citep{lesperance1992algorithm}, constrained-Newton methods \citep{wang2007fast}, 
and hybrid methods \citep{liu2007partially, bohning2003algorithm}---typically described for the special case of univariate and homoscedastic errors. In this section, we discuss our strategy for computing the NPMLE as well as the challenges of scaling the computation to large data sets. \blue{To make our implementation accessible to the research community, we have developed a Python package \texttt{npeb} with code to fit NPMLEs in a variety of empirical Bayes problems.}\label{code-npeb2}

We follow the approach of \cite{koenker2014convex}, who approximated the infinite-dimensional problem by constraining the support of $\aprior$ to a large finite set. For a non-empty, closed set $\ca \subseteq \R^\p$, define a support-constrained NPMLE as any solution
\begin{align}\label{eq-discretized-npmle-defn}
\npmle^{\ca}&\in \argmax_{\aprior\in \Ps(\ca)} \ell_n(\aprior),
\end{align}
where $\Ps(\ca)$ denotes the set of probability measures supported on \blue{a subset of} $\ca$. In particular, $\npmle = \npmle^{\R^\p}$ by definition, and by Proposition~\ref{prop-support} we may write $\npmle = \npmle^\cm$ for a compact subset $\cm$ defined explicitly in terms of the data.

~

\subsection{Computational guarantees for discretization}\label{sec-discretization-guarantees}

We now describe one strategy for choosing the discretization set~$\ca$. Fix~$\delta > 0$. Let~$\ch$ denote a covering of~$\cm$ by closed hypercubes of width~$\delta$, i.e.~\[\ch = \left\{x_j + [-\delta/2, \delta/2]^\p : j\in\{1,\dots, J\}\right\}\] for some set of points~$x_1, \dots,x_J\in\R^\p$ such that $\cm \subseteq\bigcup_{j=1}^J\left(x_j + [-\delta/2, \delta/2]^\p\right)$. \blue{To construct the covering $\ch$, we may take $\{x_j\}_{j=1}^J$ to be a grid of points spaced $\delta$ apart along each coordinate.} Now define the discretized support~$\ca$ to be the set of corners of hypercubes in~$\ch$; specifically, for each hypercube $x_j + [-\delta/2, \delta/2]^\p$ in $\ch$, the point~$x_j +\frac{\delta}{2}v\in\ca$ for every $v\in \{-1, 1\}^\p$. Because~$\cm$ is compact,~$\ca$ is a finite set which we denote by $\{a_j\}_{j=1}^m$. Constraining the NPMLE to this finite set of atoms $a_1,\dots,a_m$ yields a finite-dimensional convex optimization problem over the mixing proportions. That is, the solution to~\eqref{eq-discretized-npmle-defn} can be written as~$\npmle^{\ca} = \sum_{j=1}^m \tilde{w}_j\delta_{a_j}$, where 
\begin{align}\label{eq-discretized-npmle}
\tilde{w} \in \argmax_{w\in \Delta_{m-1}} \frac{1}{n}\sum_{i=1}^n\log \left(\sum_{j=1}^m L_{ij}w_j\right),
\end{align}
and $L_{ij} = \varphi_{\Sigma_i}(X_i - a_j)$ encodes an $n\times m$ kernel matrix. The EM algorithm \citep{dempster1977maximum} can be used to optimize directly over the mixing proportions~$\tilde{w}$. While this approach was advocated by~\citet{lashkari2008convex} and \citet{jiang2009general}, EM can be prohibitively slow~\citep{redner1984mixture, koenker2014convex}. A crucial observation made by \cite{koenker2014convex} is that~\eqref{eq-discretized-npmle} is a (finite-dimensional) convex optimization problem, enabling the use of a wide array of tools from modern convex optimization; they proposed solving the dual to~\eqref{eq-discretized-npmle} using an interior point solver, and \cite{gu2017rebayes} provided an R implementation to solve univariate problems. \cite{kim2020fast} proposed sequential quadratic programming to solve a variant of the primal problem directly, demonstrating superior scalability with the sample size~$n$. 

To justify the grid approximation, some consideration of the discretization error is warranted. Our next result shows that as $\delta\downarrow 0$, the log-likelihood $\ell_n(\npmle^{\ca})$ of the discretized NPMLE approaches that of the unconstrained, exact NPMLE; moreover, the bound on the gap depends on known quantities, so it can be used to guide a suitable choice of~$\delta$.

\begin{prop}\label{prop-approximation} Let $\cm\subset\R^\p$ denote any compact set such that every solution~\eqref{eq-NPMLE} is supported on~$\cm$. Suppose the diameter of the set~$\cm$ is at most~$D$, the minimum eigenvalue of each $\Sigma_i$ is at least $\underline{k}$, and fix $\delta \in \left(0, \sqrt{\frac{3}{4\p}}\underline{k}D^{-1}\right)$. Let~$\ch$ denote a cover of~$\cm$ by closed hypercubes of width~$\delta$, and let~$\ca$ denote the set of corners of hypercubes in~$\ch$. Every approximate NPMLE~$\npmle^{\ca}$ satisfies
\begin{align}\label{eq-log-likelihood-approximation}
\sup_{\aprior\in\Ps(\R^\p)}\ell_n(\aprior)
- \ell_n(\npmle^{\ca})
\le \p\underline{k}^{-2}\left(2D^2+\frac{1}{2}\right)\delta^2.
\end{align}
\end{prop}

We prove Proposition~\ref{prop-approximation} in Section~\ref{sec-proofs-characterization} of the Supplementary Material. Proposition~\ref{prop-approximation} shows that we can tractably approximate the NPMLE via a finite-dimensional, convex optimization problem. As we show in Section~\ref{sec-stats}, our theoretical results on the statistical properties of the NPMLE hold for any approximate solution~$\npmle^\ca$ which places nearly as high likelihood on the observations as the \blue{true prior~$\trueprior$. Proposition~\ref{prop-approximation} can be used to guarantee this, since~\eqref{eq-log-likelihood-approximation} yields a bound on~$\ell_n(\trueprior) - \ell_n(\npmle^\ca)$ without needing to know~$\trueprior$}. Hence for $\delta$ sufficiently small we can guarantee that the discretization error is negligible.

\cite{dicker2016high} showed in the univariate, homoscedastic case that a finely discretized NPMLE is statistically indistinguishable from the NPMLE for the purpose of density estimation. However, their analysis of the discretization error makes use of the modeling assumptions~\eqref{eq-obs-model} and is statistical in nature, so their theoretical results provide little guidance on how much error is incurred due to discretization for a fixed data set. Our result aligns more closely with and in fact essentially generalizes \citet[][Proposition~5]{jiang2009general}, which bounded the optimality gap for a particular algorithm, discretization scheme and fixed data set. The main difference between our result and \citet[][Proposition~5]{jiang2009general} is that the latter analyzed the EM algorithm for the mixing proportions~\eqref{eq-discretized-npmle}, whereas by using a black-box, second-order optimization method to solve for the mixing proportions~$\tilde{w}$, we can solve for the discretized NPMLE~$\npmle^\ca$ much more accurately.

\blue{\subsection{Exemplar+ method for moderate dimensions}\label{sec-exemplar}

The gridding discretization method does not scale well to moderate dimensions. In the approximation guarantee in Proposition~\ref{prop-approximation}, the number of hypercubes of width~$\delta$ typically scales like~$\delta^{-\p}$. An alternative---known as the \emph{exemplar NPMLE}---is to select the data points as the atoms $\ca = \{X_1,\ldots,X_n\}$ \citep{bohning1999computer, lashkari2008convex}. The exemplar NPMLE is the support constrained NPMLE~$\npmle^\ca$ using the data as the atoms. 

\begin{figure}
    \centering
    \includegraphics[width=.32\textwidth]{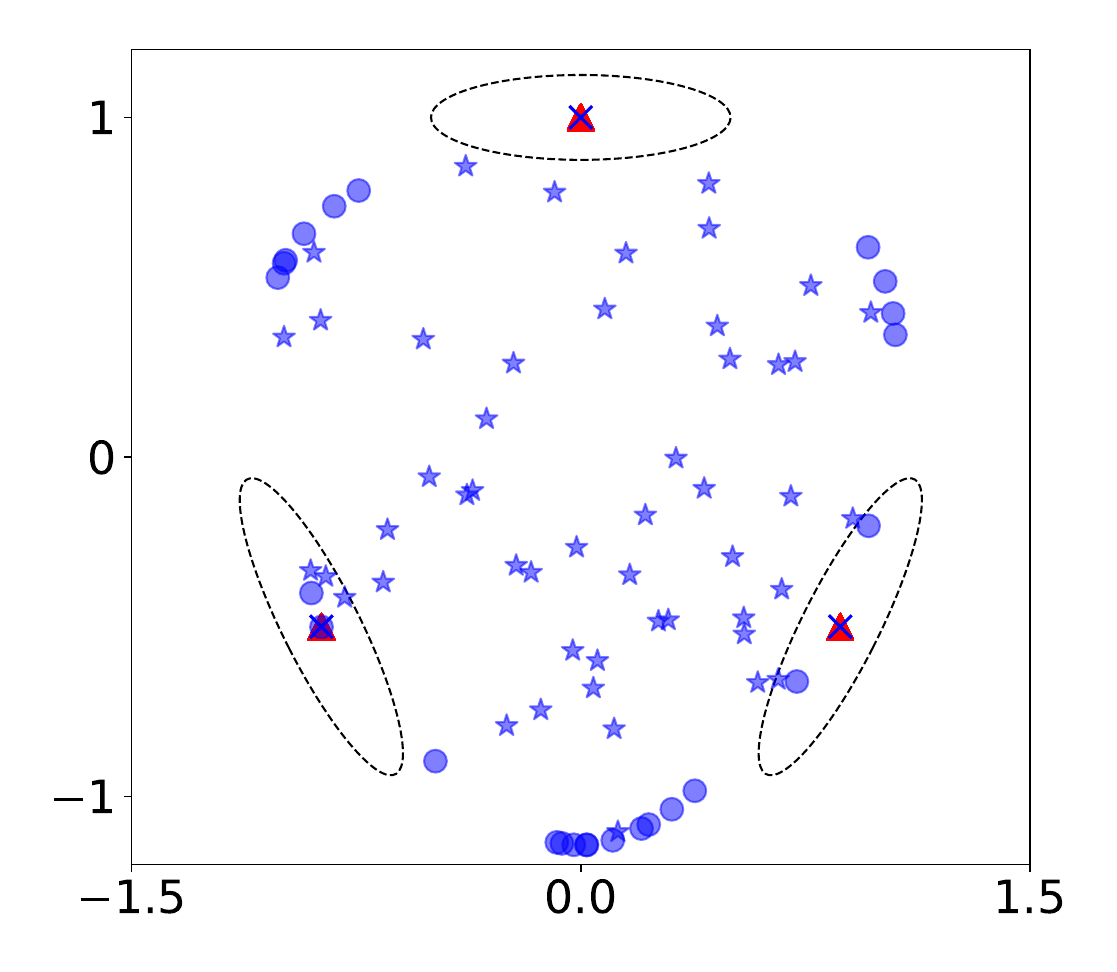}
    \includegraphics[width=.32\textwidth]{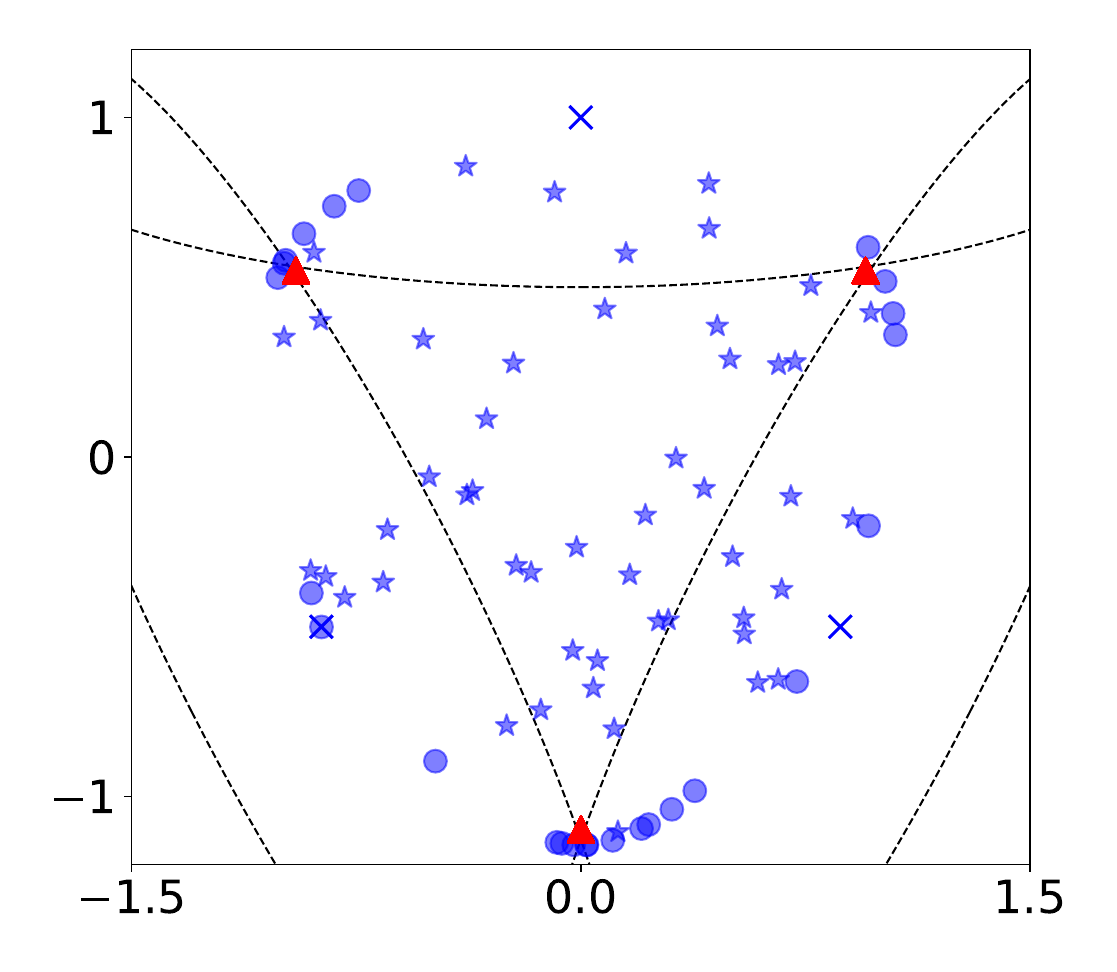}
    \includegraphics[width=.32\textwidth]{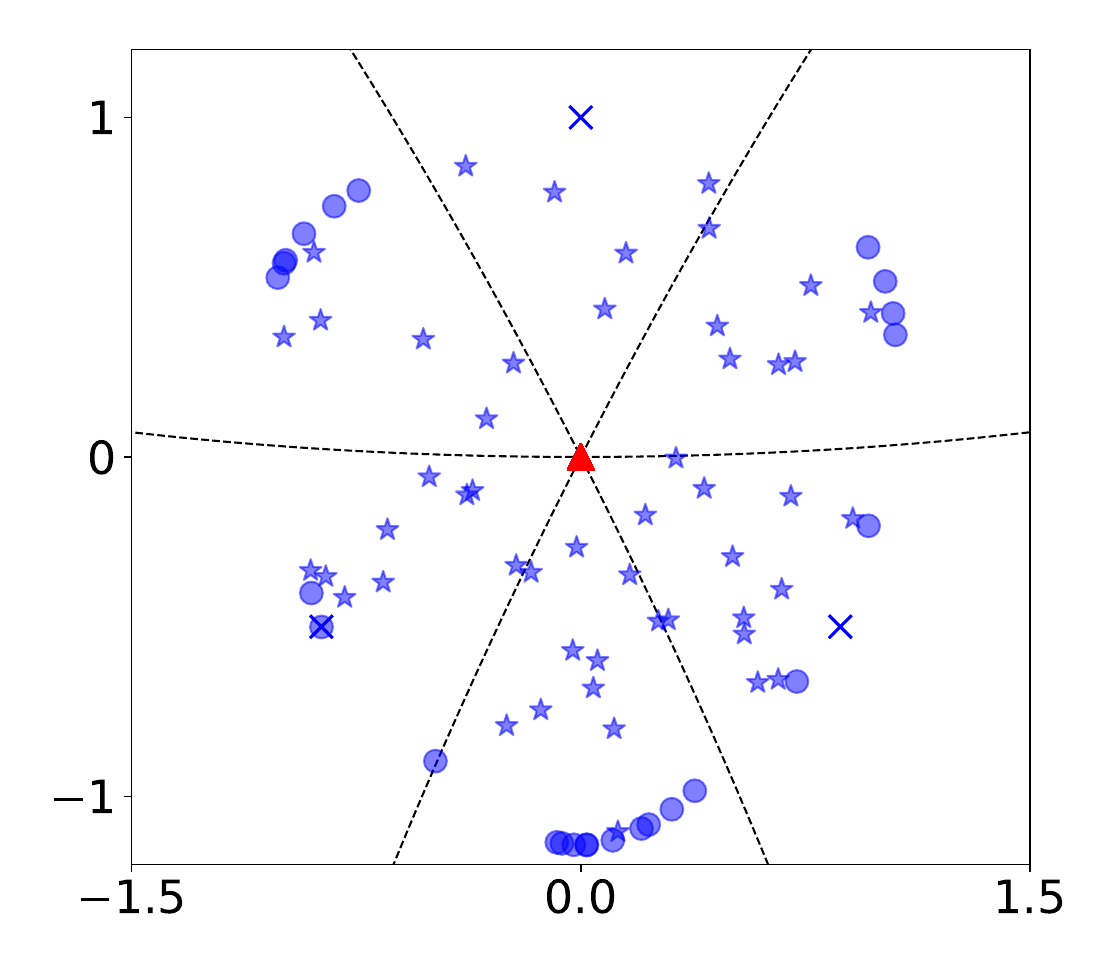}
    \includegraphics[width=.3\textwidth]{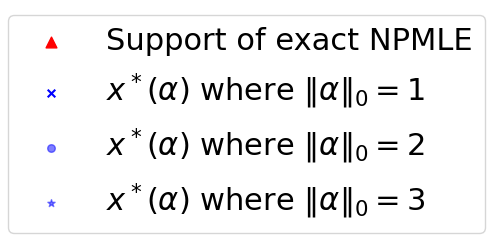}
    
    \caption{\blue{Left: An example of observations (blue \texttt{x}'s) $X_1 = (0, 1)$, $X_2 = (\frac{\sqrt{3}}{2}, -\frac{1}{2})$, and $X_3 = (-\frac{\sqrt{3}}{2}, -\frac{1}{2})$
with covariances (dashed ellipses) $\Sigma_1=\protect\begin{bmatrix}4&0\\0&1\protect\end{bmatrix}$, $\Sigma_2=\protect\begin{bmatrix}1.75&\frac{3\sqrt{3}}{4}\\\frac{3\sqrt{3}}{4}&3.25\protect\end{bmatrix}$  and $\Sigma_3=\protect\begin{bmatrix}1.75&-\frac{3\sqrt{3}}{4}\\-\frac{3\sqrt{3}}{4}&3.25\protect\end{bmatrix}$, where the NPMLE is supported on atoms (red $\blacktriangle$'s). Blue $\bullet$'s represent atoms~$x^*(\alpha)$ sampled from the ridgeline manifold~$\cm$ with~$\alpha$ being $2$-sparse, and blue $\star$'s represent the same with~$\alpha$ being $3$-sparse. Middle: same plot but each $\Sigma_i$ is scaled by~$4$. Right: same plot but each~$\Sigma_i$ is scaled by~$8$.}}
    \label{fig:exemplar-types}
\end{figure}

One issue with the exemplar approach is that, in general, the observations need not be close to the atoms of the NPMLE. For example, consider the heteroscedastic problem shown in Figure~\ref{fig-bounding-box}. Since all of the covariance matrices are diagonal, we know from Proposition~\ref{prop-support} that the support of the NPMLE is contained in the axis-aligned minimum bounding box of the data. In this case, the atoms of the NPMLE are near the corners of the box, far from the observations. More generally, it is helpful to view the exemplar NPMLE of~\citet{lashkari2008convex} as selecting atoms in the ridgeline manifold~$\cm$ with a special structure. Specifically, the exemplar method selects~$a_j = x^*(e_j)$, $j=1,\ldots,n$, where~$e_j\in \R^n_+$ is a standard basis vector. 

We introduce a new strategy for selecting atoms~$\ca$, called the \textit{exemplar+ NPMLE}, which aims to get a more representative sampling of points in the ridgeline manifold~$\cm$. By Proposition~\ref{prop-support}, all possible atoms for $\npmle$ can be generated by vectors of the form $x^*(\alpha)$ where $\alpha\in \Delta_{n-1}$ is a probability vector that is $(\p+1)$-sparse. The original exemplar method considers only~$\alpha$ that are~$1$-sparse. We propose to randomly sample, for each~$j\in \{1,\ldots,\p+1\}$, $N_j\ge 0$ weight vectors~$\alpha$ that are~$j$-sparse. See Algorithm~\ref{alg:exemplar-plus} for a precise description of our proposal.

\begin{algorithm}
\caption{Exemplar+ NPMLE}\label{alg:exemplar-plus}
\begin{algorithmic}
\blue{\Require $N_1,\ldots,N_{\p+1} \geq 0$
\State $\ca \gets \emptyset$
\For{$l = 1,\ldots,\p+1$}
    \For{$k=1,\ldots,N_l$}
    
    \State Sample a subset $S\subseteq [n]$ of size $l$, uniformly at random \Comment{(without replacement if $l = 1$)}
    \State Sample weights $\alpha\in \Delta_{n-1}$ such that $\alpha_i=0$ for $i\not\in S$, uniformly at random
    \State $\ca \gets \ca\cup\{x^*(\alpha)\}$ (see~\eqref{eq-ridgeline})
    \EndFor
    \EndFor}

\State \blue{Enumerate $\ca = \{a_1,\ldots,a_m\}$}
\State \blue{Compute $n\times m$ kernel matrix $L_{ij}\gets \varphi_{\Sigma_i}(X_i - a_j)$}
\State \blue{Solve~\eqref{eq-discretized-npmle} for weights~$\tilde{w}\in \Delta_{m-1}$}

\noindent\blue{\Return Approximate NPMLE $\npmle^\ca = \sum_{j=1}^m \tilde{w}_j\delta_{a_j}$}
\end{algorithmic}
\end{algorithm}

}

\blue{Algorithm~\ref{alg:exemplar-plus} is motivated by the observation that different sparsity levels~$\|\alpha\|_0 = j$ for $j\in \{1,\ldots,\p+1\}$ will give better approximations to the support of the NPMLE depending on the scale of the covariance matrices $(\Sigma_i)$. Figure~\ref{fig:exemplar-types} illustrates this point. When the covariances~$\Sigma_i$ are sufficiently small (relative to the distance between observations), the support of the NPMLE will be nearly the observations: this is where the original exemplar method will be most accurate. However, as the scales of the covariance matrices~$(\Sigma_i)$ grow, the atoms of the NPMLE occur at~$x^*(\alpha)$ where~$\|\alpha\|_0 > 1$.}

\blue{The toy examples in Figure~\ref{fig:exemplar-types} suggest that we should not set any of the~$N_j$'s equal to~$0$. We propose to set the counts equal to $N_1 = n$ and $N_2 = \cdots = N_{\p+1} = \lceil\frac{n}{\p}\rceil$. This ensures both that the exemplar+ always produces a higher likelihood than the original exemplar method, and that the computational cost is not much higher since we only require twice the number of atoms, regardless of the dimension~$\p$.}

\blue{As we shall see in Section~\ref{sec-stats}, for our statistical guarantees, the main condition on an approximate NPMLE~$\npmle$ is that its likelihood $\ell_n(\npmle)$ is larger than the likelihood $\ell_n(\trueprior)$ of the true prior. In Appendix~\ref{sec-additional-experiments} of the Supplementary Material, we show in simulations that exemplar+ consistently satisfies this requirement across a variety of choices of $\trueprior$. In fact, the gap between the two likelihoods grows considerably in moderate dimensions~(in our experiment we tested up to $\p=20$). This suggests that the exemplar+ method is especially well-suited for applications in moderate to high dimensions.}

\section{Statistical properties}\label{sec-stats}

The NPMLE~$\npmle$ applies as a plug-in estimator of the prior distribution~$\trueprior$ for many purposes. The traditional statistical setting is density estimation, where working in a Gaussian mixture model greatly simplifies the problem of estimating the marginal density of each observation~$X_i$. In particular, $f_{\npmle, \Sigma_i}$ is a natural, tuning-free estimate of the true marginal density $f_{\trueprior, \Sigma_i}$. Another problem setting---at the heart of empirical Bayes methodology---is to imitate the Bayesian inference we would conduct if we knew~$\trueprior$. Denoising, using~$(\estvec_i)_{i=1}^n$ as plug-in estimators of the true posterior means~$(\orvec_i)_{i=1}^n$, represents the most basic instantiation. Finally, often we wish to compare~$\npmle$ to the prior~$\trueprior$ directly. Since we are estimating the prior given observations from a convolution model~$X_i\simind f_{\trueprior, \Sigma_i}$, deconvolution refers to the problem of estimating~$\trueprior$.

In this section, we establish that the NPMLE is well-suited for all three disparate targets of estimation: the marginal densities~$(f_{\trueprior, \Sigma_i})_{i=1}^n$, the oracle posterior means~$(\orvec_i)_{i=1}^n$ and the prior~$\trueprior$. We allow for the possibility that~$\npmle$ is any approximate NPMLE, with the exact conditions being given in each theorem; \blue{hence, potential non-uniqueness (established in Lemma~\ref{rem-nonunique}) is not an issue in practice.}\label{remark-nonuniqueness-does-not-matter} Throughout this section, we use the standard notation~$x\lesssim_{p, q} y$ to mean $x\le C_{p, q}y$ for some positive constant $C_{p,q} > 0$ depending only on problem parameters~$p,q$.

\subsection{Density estimation: average Hellinger accuracy}\label{sec-density-estimation}

As the distribution of $X_i$ varies with $i$, we consider the density estimation quality of the NPMLE~\eqref{eq-NPMLE} in terms of the average squared Hellinger distance, i.e. for $G, H\in \Ps(\R^\p)$, 
\[
\bar{h}^2(f_{G, \bullet}, f_{H, \bullet}) 
\coloneqq \frac{1}{n}\sum_{i=1}^n h^2(f_{G,\Sigma_i}, f_{H, \Sigma_i}),
\]
where $h^2(f,g)= \frac{1}{2}\int\left(\sqrt{f}-\sqrt{g}\right)^2$ denotes the usual squared Hellinger distance between a pair of densities $f, g$. In the homoscedastic case where $\Sigma_i\equiv\Sigma$, our proposed loss function~$\bar{h}^2(f_{G, \bullet}, f_{H, \bullet}) = h^2(f_{G, \Sigma}, f_{H, \Sigma})$ agrees with the usual squared Hellinger distance. Our first result bounds the average squared Hellinger accuracy $\bar{h}^2(f_{\npmle, \bullet}, f_{\trueprior, \bullet})$ of the NPMLE. In order to accommodate general heteroscedastic $\Sigma_i$, we state our results in terms of uniform upper and lower bounds on the spectra of all of the matrices, i.e. $\underline{k}I_\p\preceq \Sigma_i \preceq \overline{k}I_\p$ for all $i$. To state the result, some additional notation is needed. We fix a positive scalar $M\ge \sqrt{10\overline{k}\log n}$ and a non-empty compact set $S\subset \R^\p$. Define the rate function controlling the squared Hellinger distance
\begin{align}\label{eq-eps}
\eps_n^2(M, S, \trueprior)
\coloneqq \mathrm{Vol}(S^{\overline{k}^{1/2}})\frac{M^\p}{n}(\log n)^{\p/2+1} + \inf_{q\in [(\p+1)/(2\log n),\infty)}\left(\frac{2\mu_q}{M}\right)^q\log n,
\end{align}
where $\mu_q$ denotes the $q^\mathrm{th}$-moment of $\mathfrak{d}_S(\vartheta)\coloneqq \inf_{s\in S}\|\vartheta - s\|_2$ under $\vartheta\sim \trueprior$, and $S^a \coloneqq \{y :  \mathfrak{d}_S(y)\le a\}$ denotes the $a$-enlargement of the set $S$. Note that we have suppressed the dependence of $\eps_n^2$ on the upper bound~$\overline{k}$.

The following result states that~$\eps_n^2(M, S, \trueprior)$ bounds the rate in average Hellinger accuracy both with high probability and in expectation. The scalar~$M\ge \sqrt{10\overline{k}\log n}$ and compact set~$S\ne \emptyset$ are free parameters. Note that the first term on the right-hand side of~\eqref{eq-eps} is increasing in $M$ and $S$, whereas the second is decreasing in each. In principle, then, we may tune the values of $M$ and $S$ to optimize the rate function~$\eps_n^2(M, S, \trueprior)$. Later in this section, we discuss a number of special cases where a more explicit rate can be obtained.

\begin{theorem}\label{thm-density-estimation}
Suppose $X_i\simind f_{\trueprior, \Sigma_i}$ where $\underline{k}I_\p\preceq\Sigma_i\preceq \overline{k}I_\p$ for all $i$. Any (approximate) solution $\npmle\in\Ps(\R^\p)$ of \eqref{eq-NPMLE} \blue{such that}
\begin{align}\label{eq-hellinger-accuracy-approximate-npmle}
\blue{\ell_n(\trueprior) - \ell_n(\npmle) \lesssim_{\p,\overline{k},\underline{k}} \eps_n^2(M, S, \trueprior)}
\end{align}
satisfies
\begin{align}\label{eq-hellinger-accuracy-whp}
\P\bigg(\bar{h}^2(f_{\npmle, \bullet}, f_{\trueprior, \bullet})  \gtrsim_{\p, \overline{k}, \underline{k}}t^2\eps_n^2(M, S, \trueprior)\bigg) \le 2n^{-t^2},
\end{align}
for all $t\ge 1$, provided $n > \max(e\underline{k}^{-\p/2}, (2\pi)^{\p/2})$. Moreover, 
\begin{align}\label{eq-hellinger-accuracy-expectation}
\E\bigg[\bar{h}^2(f_{\npmle, \bullet}, f_{\trueprior, \bullet})\bigg] \lesssim_{\p, \overline{k}, \underline{k}}\eps_n^2(M, S, \trueprior). 
\end{align}
\end{theorem}

We prove Theorem~\ref{thm-density-estimation} in Section~\ref{sec-proofs-density-estimation} of the Supplementary Material. Our proof extends Theorem~2.1 of \citet{saha2020nonparametric} on the multivariate, homoscedastic case~$\Sigma_i\equiv I_\p$ and Theorem~4 of \citet{jiang2020general} on the univariate, heteroscedastic case~$\p=1$, which in turn build upon Theorem~1 of~\citet{zhang2009generalized} on the univariate, homoscedastic case. The general theory on rates of convergence for maximum likelihood estimators~\citep{wong1995probability, geer2000empirical} can in principle be used to bound $\bar{h}^2(f_{\npmle, \bullet}, f_{\trueprior, \bullet})$. Our proof technique deviates from the general theory by directly bounding the likelihood~$f_{\npmle, \Sigma_i}(x)$ for $x$ outside some pre-specified domain (controlled by the choice of set~$S$), and then covering the set of densities $\{f_{\aprior, \bullet} : \aprior\in \Ps(\R^\p) \}$ within the domain in the~$L_\infty$ metric. 

Theorem~\ref{thm-density-estimation} provides a sharp bound in many special cases of~$\trueprior$. For a given~$\trueprior$ we need to optimize over the choices of $M \ge  \sqrt{10\overline{k}\log n}$ and the nonempty compact set $S \subset \R^\p$ to obtain the smallest value of the rate function~$\eps_n^2(M, S, \trueprior)$. Our next result performs this calculation for various assumptions on the prior~$\trueprior$.

\begin{corollary}\label{cor-density-estimation} Suppose $X_i\simind f_{\trueprior, \Sigma_i}$ where $\underline{k}I_\p\preceq\Sigma_i\preceq \overline{k}I_\p$ for all $i$. Suppose $\npmle \in \Ps(\R^\p)$ is any approximate NPMLE such that 
\begin{align}\label{eq-approximate-npmle}
\blue{\ell_n(\trueprior) - \ell_n(\npmle)}
\lesssim_{\p, \overline{k}, \underline{k}} \frac{(\log n)^{\p+1}}{n}.
\end{align}
\blue{Let $\mathbb{B}_r(x) \coloneqq \{y : \|x-y\|_2\le r\}$ denote the $\p$-dimensional closed ball of radius $r$ centered at $x$.}
\begin{enumerate}
\item (Discrete support) If $\trueprior = \sum_{j=1}^{k^*} w^*_j\delta_{a^*_j}$, then 
\[
\E \bar{h}^2(f_{\npmle, \bullet}, f_{\trueprior, \bullet})
\lesssim_{\p, \overline{k}, \underline{k}}  \frac{k^*}{n} (\log n)^{\p+1}.
\]
\item  (Compact support) If $\trueprior$ has compact support $S^*$, then 
\[
\E \bar{h}^2(f_{\npmle, \bullet}, f_{\trueprior, \bullet})
\lesssim_{\p,\overline{k}, \underline{k}} \frac{\mathrm{Vol}\left(S^* + \mathbb{B}_{\overline{k}^{1/2}}(0)\right)}{n} (\log n)^{\p+1}.
\]

\item (Simultaneous moment control) Suppose that there is a compact $S^*\subset\R^\p$ and $\alpha\in (0, 2]$, $K\ge 1$ such that $\mu_q\coloneqq \E_{\vartheta\sim \trueprior}\left[\mathfrak{d}^q(\vartheta, S)\right]^{1/q}\le Kq^{1/\alpha}$ for all $q\ge 1$ (recall $\mathfrak{d}_S(\vartheta)\coloneqq \inf_{s\in S}\|\vartheta - s\|_2$ as above). Then 
\[
\E \bar{h}^2(f_{\npmle, \bullet}, f_{\trueprior, \bullet})
\lesssim_{\alpha, K, \p, \overline{k}, \underline{k}}  \frac{\mathrm{Vol}\left(S^* + \mathbb{B}_{\overline{k}^{1/2}}(0)\right)}{n} (\log n)^{\frac{2+\alpha}{2\alpha}\p + 1}.
\]
\item (Finite $q^\mathrm{th}$ moment) Suppose that there is a compact $S^*\subset\R^\p$ and $\mu, q > 0$ such that $\mu_q\le \mu$. Then 
\[
\E \bar{h}^2(f_{\npmle, \bullet}, f_{\trueprior, \bullet})
\lesssim_{\mu, q, \p,\overline{k}, \underline{k}}  \left(\frac{\mathrm{Vol}\left(S^* + \mathbb{B}_{\overline{k}^{1/2}}(0)\right)}{n}\right)^{\frac{q}{q+\p}} (\log n)^{\frac{q}{2q+2\p}\p+1}.
\]
\end{enumerate}
\end{corollary}

Given the general result in Theorem~\ref{thm-density-estimation}, Corollary~\ref{cor-density-estimation} follows directly from the calculations of \citet{saha2020nonparametric} in Corollary~2.2 and Theorem~2.3. Corollary~\ref{cor-density-estimation} captures an important adaptation property of the NPMLE. The cases $(a)-(d)$ described in the result are nested in the sense that $(a)$ implies $(b)$, $(b)$ implies $(c)$, and $(c)$ implies $(d)$; consequently the rates get progressively worse as our assumptions weaken. This means that the NPMLE, despite searching over all probability measures~$\Ps(\R^\p)$, obtains better rates when structure is present in the prior~$\trueprior$. 

Most strikingly, when~$\trueprior$ has discrete support with~$k^*$ support points, the rate in $(a)$ is~$\frac{k^*}{n}$ up to logarithmic factors {\sl without assuming any knowledge of~$k^*$}. This rate matches the minimax rate over all discrete distributions with at most~$k^*$ support points \citep{saha2020nonparametric}, meaning we could not expect to do much better even if~$k^*$ were known. In the extreme case where~$k^*=1$, the observations actually come from a simple Gaussian, i.e. $f_{\trueprior, \Sigma_i}(x) = \varphi_{\Sigma_i}(x-a^*_1)$ with common mean~$a^*_1\in\R^\p$, so our result says we don't lose much in the rate when we model the density with a mixture even when it turns out to be a simple Gaussian. Similarly, in $(b)$, the rate adapts to the size of the support~$S^*$ without prior knowledge of this support or even a bound on its size. Up through simultaneous moment control $(c)$, the dimension~$\p$ only affects the rate as a function of~$n$ through the logarithmic factor. Hence, the NPMLE avoids the usual curse of dimensionality to some extent, while still achieving consistency in the heavier tailed setting~$(d)$. The logarithmic factors in our bounds might be reduced slightly but cannot be eliminated as they are present in the minimax lower bounds \citep{kim2022minimax}.

\subsubsection{Implications for the Discretization Rate}\label{sec-discretization-for-statistical-purposes}

Theorem~\ref{thm-density-estimation} establishes that up to a multiplicative constant (depending only on the dimension~$\p$ and bounds $\underline{k}, \overline{k}$ on the eigenvalues of the covariance matrices) the quantity~$\eps_n^2(M, S, \trueprior)$ controls the average Hellinger accuracy~$\E[\bar{h}^2(f_{\npmle, \bullet}, f_{\trueprior, \bullet})]$ of the NPMLE. This also holds for approximate solutions to the optimization problem~\eqref{eq-NPMLE} that, in accordance with~\eqref{eq-hellinger-accuracy-approximate-npmle}, place nearly as much likelihood on the data as does \blue{the true prior~$\trueprior$}. It is natural to compare the requirement~\eqref{eq-hellinger-accuracy-approximate-npmle} with our computational guarantee on the discretization error~\eqref{eq-log-likelihood-approximation} from Proposition~\ref{prop-approximation}. The free parameter which controls the discretization error is the resolution~$\delta > 0$, which represents the width of the hypercubes we use to cover the ridgeline manifold~$\cm$ or any of its outer-approximations from Proposition~\ref{prop-support}. Thus, in order to satisfy the main requirement of Theorem~\ref{thm-density-estimation}, we need to take~$\delta$ such that
\[
\eps_n^2(M, S, \trueprior)
\gtrsim_{\p,\overline{k},\underline{k}} \p\underline{k}^{-2}\left(2D^2+\frac{1}{2}\right)\delta^2.
\]
Observe from the definition of~$\eps_n^2$ that $\eps_n^2(M, S, \trueprior) \gtrsim_{\p, \overline{k}, \underline{k}} \frac{(\log n)^{\p+1}}{n}$ for all $M \ge  \sqrt{10\overline{k}\log n}$ and all compact~$S$. Absorbing additional terms depending on $\p$,$\overline{k}$, and $\underline{k}$ and assuming for simplicity that $D > \frac{1}{2}$, choosing $\delta$ such that \begin{align}\label{eq-disc-rate-preliminary}
D^2\delta^2\lesssim_{\p, \overline{k}, \underline{k}} \frac{(\log n)^{\p+1}}{n}\end{align} 
suffices for the discretized NPMLE to be statistically indistinguishable from a global maximizer. 

The inequality~\eqref{eq-disc-rate-preliminary} gives a preliminary bound on the rate at which the discretization level~$\delta$ should decrease with~$n$. Still, recall from Proposition~\ref{prop-approximation} that $D$ denotes the diameter of the ridgeline manifold~$\cm$, so~$D$ does depend on~$n$. To sketch the dependence, let us consider a representative example where~$\trueprior$ has sub-Gaussian tails and all of the $\Sigma_i$'s are diagonal. In this case, by Proposition~\ref{prop-support} part $(b)$, the ridgeline manifold~$\cm$  is contained in the axis-aligned minimum bounding box of the data
\[\prod_{j=1}^\p\left[\min_{i\in \{1, \dots, n\}} X_{ij}, \max_{i\in \{1, \dots, n\}} X_{ij}\right].\]
Due to the tail condition, the length of each side of this hyper-rectangle grows like~$\sqrt{\log n}$ with high probability up to multiplicative factors depending on~$\overline{k}$: hence, the diameter~$D$ also scales like~$\sqrt{\log n}$ with high probability up to multiplicative factors depending on~$\overline{k}$ and~$\p$. We have thus shown that it suffices to discretize at a resolution of~$\delta\asymp \sqrt\frac{(\log n)^\p}{n}$. The number of points in our covering~$\ca$ is of order~$m\asymp \left(\frac{n}{(\log n)^\p}\right)^{\p/2}$. In the univariate case~$\p=1$, this slightly improves the finding of Theorem~2 of \citet{dicker2016high}, who showed that an~$m=\sqrt{n}$-discretization of the range of the data~$[X_{(1)}, X_{(n)}]$ suffices for the same rate in Hellinger distance. Their bound on the large-deviation probability is also logarithmic, i.e. $O\left(\frac{1}{\log n}\right)$ whereas our equation~\eqref{eq-hellinger-accuracy-whp} is polynomial in~$n$. Our analysis also clarifies that the sense in which we need approximate NPMLE~\eqref{eq-hellinger-accuracy-approximate-npmle} is through the likelihood of the observations, relative to the \blue{true prior~$\trueprior$}, which could be useful for comparing alternative approaches to approximating the NPMLE \blue{using simulations where we know~$\trueprior$}. 

\subsection{Denoising: an oracle inequality}

In this section we turn to the problem of estimating the oracle posterior means $(\orvec_i)_{i=1}^n$; see~\eqref{eq-oracle}. We evaluate the performance of~$(\estvec_i)_{i=1}^n$ (see~\eqref{eq-posterior-mean}) as an estimator for~$(\orvec_i)_{i=1}^n$ using the mean squared error risk measure:
\[\frac{1}{n}\sum_{i=1}^n\E\|\estvec_i - \orvec_i\|_2^2.\] 
Since~$\orvec_i$ is the optimal estimator of~$\truevec_i$ given model~\eqref{eq-obs-model}, the above mean squared error quantifies the price of misspecifying~$\trueprior$ with the data-driven estimator~$\npmle$. Hence, this loss is also known as the per-instance {\sl empirical Bayes regret}.

Our next result states that the rate function~$\eps_n^2(M, S, \trueprior)$ governing the Hellinger accuracy (see~\eqref{eq-eps}) also upper bounds the regret, up to additional logarithmic factors. We provide the same special cases of the rate as those stated in Corollary~\ref{cor-density-estimation}.

\begin{theorem}\label{thm-denoising} Suppose $X_i\simind f_{\trueprior, \Sigma_i}$ where $\underline{k}I_\p\preceq\Sigma_i\preceq \overline{k}I_\p$ for all $i$. Let $\npmle$ denote any approximate NPMLE satisfying~\eqref{eq-approximate-npmle} and
\begin{align}\label{eq-approximate-npmle-2}
\blue{\ell_n\left(n^{-1} \delta_{X_i} + (1-n^{-1})\npmle\right) \le \ell_n(\npmle), \qquad \text{for all } i=1,\ldots, n.}
\end{align}
Fix some $M \ge \sqrt{10\bar{k}\log n}$ and a nonempty, compact set $S\subset\R^\p$. Define $\eps_n^2(M, S, \trueprior)$ as in~\eqref{eq-eps}. For all $n\ge 5\underline{k}^{-\p/2}\lor (2\pi)^{\p/2}$,
\begin{align}
\frac{1}{n}\sum_{i=1}^n\E\|\estvec_i - \orvec_i\|_2^2
\lesssim_{\p, \overline{k}, \underline{k}}\eps_n^2(M, S, \trueprior)(\log n)^{(\p/2-1)\lor 3}.
\end{align}
In particular, consider the following special cases for $\trueprior$:
\begin{enumerate}
\item (Discrete support) If $\trueprior = \sum_{j=1}^{k^*} w^*_j\delta_{a^*_j}$, then 
\[
\frac{1}{n}\sum_{i=1}^n\E\|\estvec_i - \orvec_i\|_2^2
\lesssim_{\p, \overline{k}, \underline{k}}  \frac{k^*}{n} (\log n)^{\p+((\p/2)\lor 4)}.
\]
\item  (Compact support) If $\trueprior$ has compact support $S^*$, then 
\[
\frac{1}{n}\sum_{i=1}^n\E\|\estvec_i - \orvec_i\|_2^2
\lesssim_{\p, \overline{k}, \underline{k}} \frac{\mathrm{Vol}\left(S^* + \mathbb{B}_{\overline{k}^{1/2}}(0)\right)}{n} (\log n)^{\p+((\p/2)\lor 4)}.
\]
\item (Simultaneous moment control) Suppose that there is a compact $S^*\subset\R^\p$ and $\alpha\in (0, 2]$, $K\ge 1$ such that $\mu_q\coloneqq \E_{\vartheta\sim \trueprior}\left[\mathfrak{d}^q(\vartheta, S^*)\right]^{1/q}\le Kq^{1/\alpha}$ for all $q\ge 1$. Then 
\[
\frac{1}{n}\sum_{i=1}^n\E\|\estvec_i - \orvec_i\|_2^2
\lesssim_{\alpha, K, \p, \overline{k}, \underline{k}}  \frac{\mathrm{Vol}\left(S^* + \mathbb{B}_{\overline{k}^{1/2}}(0)\right)}{n} (\log n)^{\frac{2\alpha\p}{2+\alpha} + ((\p/2)\lor 4)}.
\]
\item (Finite $q^\mathrm{th}$ moment) Suppose that there exists a compact $S^*\subset\R^\p$ and $\mu, q > 0$ such that $\mu_q\le \mu$. Then 
\[
\frac{1}{n}\sum_{i=1}^n\E\|\estvec_i - \orvec_i\|_2^2
\lesssim_{\mu, q, \p, \overline{k}, \underline{k}}  \left(\frac{\mathrm{Vol}\left(S^* + \mathbb{B}_{\overline{k}^{1/2}}(0)\right)}{n}\right)^{\frac{q}{q+\p}} (\log n)^{\frac{q\p}{2q+2\p} + ((\p/2)\lor 4)}.
\]
\end{enumerate}
\end{theorem}

Theorem~\ref{thm-denoising} shows that the denoising problem shares the adaptation features as the density estimation problem. Since we have assumed~$\underline{k}I_\p\preceq\Sigma_i\preceq \overline{k}I_\p$ for all $i=1,\dots, n$, the same set of results also hold for the scaled regret $\frac{1}{n}\sum_{i=1}^n\E(\estvec_i - \orvec_i)^\mathsf{T}\Sigma_i^{-1}(\estvec_i - \orvec_i).$ 

\blue{Compared with the Hellinger accuracy results in the previous section, the denoising guarantees in Theorem~\ref{thm-denoising} have an additional condition \eqref{eq-approximate-npmle-2} that the approximate NPMLE~$\npmle$ places higher likelihood on the data than the slightly perturbed priors $n^{-1} \delta_{X_i} + (1-n^{-1})\npmle$. This condition is used to ensure that the likelihood of each individual observation $f_{\npmle, \Sigma_i}(X_i)$ is not too small. This condition can be directly verified for any approximate NPMLE, and in particular it necessarily holds for any discretized NPMLE $\npmle^\ca$ such that the observations~$\{X_i\}$ are a subset of the atoms~$\ca$, as is the case with the exemplar+ method (as long as~$N_1= n$) in Section~\ref{sec-exemplar}.}

\begin{remark} (On the proof of Theorem~\ref{thm-denoising} in Section~\ref{sec-proofs-denoising} of the Supplementary Material) Our proof extends Theorem~3.1 of \citet{saha2020nonparametric} on the multivariate, homoscedastic case~$\Sigma_i\equiv I_\p$ and Theorem~1 of \citet{jiang2020general} on the univariate, heteroscedastic case~$\p=1$, which in turn build upon Theorem~5 of~\citet{jiang2009general} on the univariate, homoscedastic case. \citet{jiang2009general} and \citet{jiang2020general} used a related notion of regret
\[
\sqrt{\frac{1}{n}\sum_{i=1}^n\E\|\estvec_i - \truevec_i\|_2^2} - \sqrt{\frac{1}{n}\sum_{i=1}^n\E\|\orvec_i - \truevec_i\|_2^2}.
\]

Tweedie's formula relates the oracle~\eqref{eq-oracle-tweedie} and empirical Bayes~\eqref{eq-gmleb-tweedie} posterior means to the corresponding marginal likelihoods, so the density estimation results of the previous section turn out to be useful for proving Theorem~\ref{thm-denoising} as well. In particular, we consider Bayes rules for priors in a covering of the Hellinger ball 
\[
\left\{\aprior\in \Ps(\R^\p) : \bar{h}^2(f_{\aprior, \bullet}, f_{\trueprior, \bullet})  \lesssim_{\p,\overline{k}, \underline{k}}t^2\eps_n^2(M, S, \trueprior)\right\},
\]
which, by Theorem~\ref{thm-density-estimation}, contains~$\npmle$ with high probability. For a fixed prior~$\aprior$, the denominator in the correction factor of Tweedie's formula
\[
X_i + \Sigma_i\frac{\nabla f_{\aprior, \Sigma_i}(X_i)}{f_{\aprior, \Sigma_i}(X_i)},
\]
namely~$f_{\aprior, \Sigma_i}(X_i)$, can be small. To avoid dividing by near-zero quantities, we regularize the above Bayes rule by replacing the denominator with~$\max\{f_{\aprior, \Sigma_i}(X_i), \rho\}$ for a small positive $\rho$. To handle heteroscedastic errors, we show that Tweedie's formula, even its regularized form, is equivariant under scale transformations.
\end{remark}

\subsection{Deconvolution: estimating the prior}

We turn to the fundamental question of how well~$\npmle$ estimates~$\trueprior$. This is known as the deconvolution problem and has received much attention in the statistical literature \citep{meister2009deconvolution,fan1991optimal,zhang1990fourier,carroll1988optimal}. Indeed, the original consistency results \citep{kiefer1956consistency, pfanzagl1988consistency} for the NPMLE focused on weak convergence of~$\npmle$ to~$\trueprior$ as~$n\to\infty$. While most prior work on deconvolution has focused on deconvolution with homoscedastic error distributions, \cite{delaigle2008density} \blue{\citep[also see][]{staudenmayer2008density}} allowed for heteroscedastic errors but relied on kernel estimators which contain additional smoothing parameters. By contrast, the NPMLE provides a tuning-free estimate of the mixing distribution~$\trueprior$, yet to our knowledge, non-asymptotic bounds on the rate of convergence for~$\npmle$ in the deconvolution problem are not known. 

In practice, the true prior~$\trueprior$ may not be discrete even though~$\npmle$ always has a discrete solution, and even if both distributions are discrete, their supports will typically differ. Our loss function must allow for comparisons of probability measures with potentially disjoint supports. \citet{nguyen2013convergence} established that a natural loss for this problem is the Wasserstein distance from the theory of optimal transport
\begin{align*}
W_2^2(G, H) \coloneqq \min_{(U,V)\in \Pi_{G, H}} \mathbb{E}\|U-V\|^2_2, 
\end{align*}
where $G, H\in \Ps(\R^\p)$ are two probability measures and $\Pi_{G,H}$ denotes the set of couplings of $G$ and $H$, i.e. joint distributions over $(U, V)\in \R^{2\p}$ such that $U\sim G$ and $V\sim H$. Indeed, even the likelihood criterion is intimately related to the Wasserstein distance: in the homoscedastic case $\Sigma_i\equiv \sigma^2 I_{\p}$, it is known that the NPMLE~\eqref{eq-NPMLE} equivalently solves an entropic-regularized optimal transport problem \citep{rigollet2018entropic}.

\citet{nguyen2013convergence} connected the deconvolution error $W_2^2(G, H)$ to the density estimation error between the mixtures, i.e. $h^2(f_{G,I_\p}, f_{H, I_\p})$ in a homoscedastic Gaussian deconvolution setting. By leveraging similar techniques as well as the support bounds of Proposition~\ref{prop-support}, we arrive at the following upper bound on the deconvolution error.

\begin{theorem}\label{thm-deconvolution} Suppose $X_i\simind f_{\trueprior, \Sigma_i}$ where $\underline{k}I_\p\preceq\Sigma_i\preceq \overline{k}I_\p$ and $\Sigma_i$ is a diagonal matrix for each $i$. Suppose further that $\trueprior([-L, L]^\p)=1$ for some $L\ge 0$. Let $\npmle$ denote any approximate NPMLE supported on the minimum axis-aligned bounding box of the data satisfying~\eqref{eq-approximate-npmle}. Then there is a function $n(d,\overline{k}, \underline{k}, L)$ such that, for all sample sizes $n$ with $n \ge n(d,\overline{k}, \underline{k}, L)$,
\begin{align*}
W_2^2(\trueprior, \npmle)
&\lesssim_{\p,\overline{k}}\frac{1}{\log n},
\end{align*}
with probability at least $1-\frac{4d}{n^8}$.
\end{theorem}

Theorem~\ref{thm-deconvolution} (proved in Section~\ref{sec-proofs-deconvolution} of the Supplementary Material) upper bounds the rate of convergence under the Wasserstein distance by the extremely slow logarithmic rate~$\frac{1}{\log n}$. It is well known that the smoothness of the Gaussian errors makes the deconvolution more difficult; in fact, the logarithmic rate is minimax optimal \citep{dedecker2013minimax}.

\begin{remark} (On Theorem~\ref{thm-deconvolution}) To our knowledge, Theorem~\ref{thm-deconvolution} is novel, and the rate of convergence for the NPMLE under a Wasserstein distance has not been studied previously. The structure of the proof follows the proof of Theorem~2 of \citet{nguyen2013convergence}. To deal with the fact that~$\npmle$ and~$\trueprior$ are typically singular, we convolve each with a distribution with full support but low variance. Compared to our results on the density estimation and denoising problems, Theorem~\ref{thm-deconvolution} makes additional assumptions on the problem structure, specifically that the covariance matrices are diagonal and that~$\trueprior$ is compactly supported. Many practical applications satisfy the diagonal covariances restriction, including both of our applications in Section~\ref{sec-app}.
\end{remark}

A common feature to our results on density estimation and denoising have been that the NPMLE adapts to the complexity of~$\trueprior$. It is reasonable to conjecture, then, that in the deconvolution problem,~$\npmle$ will also enjoy some adaptation properties under the Wasserstein distance. We close this section with a sharper result on the Wasserstein rate in the special case where the observations are drawn from Gaussian distributions with common mean~$\mu\in \R^\p$.

\begin{theorem}\label{thm-deconvolution-point-mass} Suppose $X_i\simind \cn(\mu, \Sigma_i)$, i.e. $X_i\simind f_{\trueprior, \Sigma_i}$ where $\trueprior = \delta_\mu$ and $\underline{k}I_\p\preceq \Sigma_i \preceq\overline{k}I_\p$ for all $i =1,\dots, n$. Let $\npmle$ denote any approximate NPMLE satisfying~\eqref{eq-approximate-npmle} and supported on $\mathbb{B}_{\kappa r}(\bar{X})$
where $\kappa = \overline{k}/\underline{k}$, $r = \max_i\|X_i-\bar{X}\|_2$, and~$\bar{X} = \frac{1}{n}\sum_{i=1}^nX_i$. Then
\[
W_2(\npmle, \trueprior) \lesssim_{\p, \overline{k}, \underline{k}}t^{3/2}\frac{(\log n)^{(\p+3)/4}}{n^{1/4}}
\]
with probability at least $1-3n^{-t^2}$ for all $t\ge 1$. 
\end{theorem}

If the approximate NPMLE~$\npmle$ of Theorem~\ref{thm-deconvolution-point-mass} is selected according to the strategy described in Section~\ref{sec-computation}, then by Proposition~\ref{prop-support} part $(c)$ its support will be contained within the ball~$\mathbb{B}_{\kappa r}(\bar{X})$. This additional assumption on the support of the approximate NPMLE is needed to have some control over the moments of~$\npmle$.

Up to logarithmic factors, the $n^{1/4}$-rate in Theorem~\ref{thm-deconvolution-point-mass} agrees with Corollary~4.1 of~\cite{ho2016strong} for the MLE of an overfitted mixture. Specifically, their result compared the MLE of $k$-component finite Gaussian mixture to a true mixing distribution~$\trueprior$ with $k^*<k$ components. \cite{wu2020optimal} and \cite{doss2020optimal} also derived the~$n^{1/4}$-rate for a different estimator under a different Wasserstein metric. All of these previous results were restricted to the homoscedastic setting. In our setting, $k^* = 1$ and $k=\khat$ is the data-dependent order of the NPMLE. The best known bound on~$\khat$ is logarithmic in $n$ \citep{polyanskiy2020self}, whereas \cite{ho2016strong} required $k$ to be fixed as~$n\to \infty$. When~$k^*$ is known, a faster $n^{1/2}$-rate is possible~\citep{heinrich2018strong} and is achieved by the MLE in a well-specified finite mixture model, i.e. setting~$k=k^*$~\citep{ho2016strong}. 

While the slower $n^{1/4}$-rate appears to be the price of flexibility of the NPMLE, Theorem~\ref{thm-deconvolution-point-mass} establishes that the NPMLE indeed adapts to structure in~$\trueprior$. Our analysis is greatly simplified by the assumption $\trueprior=\delta_{\mu}$, since there is only one coupling between~$\npmle$ and~$\trueprior$. We leave for future work the important question of the extent to which~$\npmle$ adapts to more general distributions~$\trueprior$.

\section{Applications}\label{sec-app} 

\blue{In this section, we first expand upon our discussion of denoising the color-magnitude diagram (CMD) from Section~\ref{sec-intro}, and then present three additional applications. The second application illustrates the performance of our method on an astronomy problem with $\p=19$ dimensions; the remaining two applications use hierarchical linear regression on education and biomedical data.}

\subsection{Color-magnitude diagram}\label{sec-cmd}

Our modeling strategy is closely related to the work of~\citet{anderson2018improving}. To compare our method to extreme deconvolution \citep{bovy2011extreme}, we use the same stellar sample, relaxing only their assumption that the prior $\trueprior$ is a mixture of Gaussians; by contrast, we allow $\trueprior$ to be an arbitrary probability measure. Specifically, we assume that after a suitable transformation of the color and magnitude measurements, the pair, denoted~$X_i\in \R^2$, come from a two-dimensional Gaussian mixture $f_{\trueprior, \Sigma_i}$ with known covariance~$\Sigma_i$. 

\begin{figure}[ht!]
\centering
\includegraphics[width=.9\textwidth]{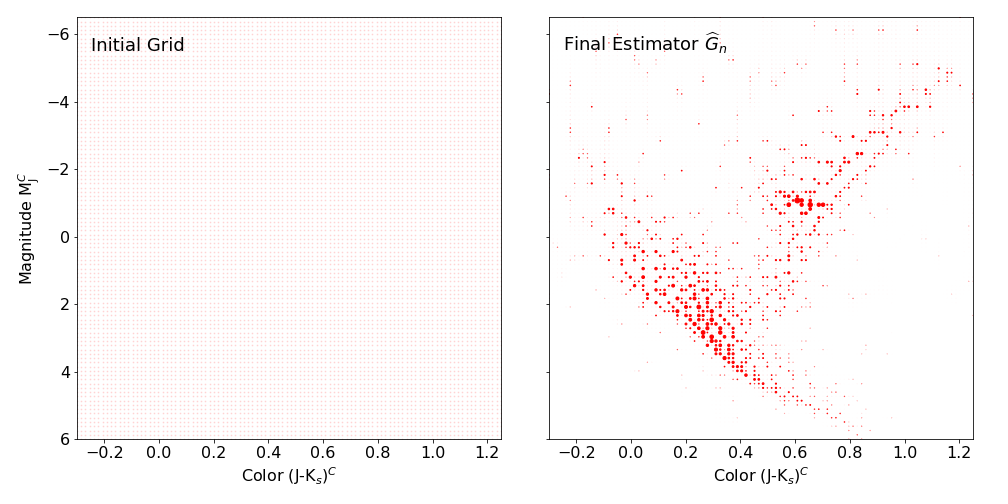}
\caption{Initial grid (left) of $m=10^4$ support points and estimated prior $\npmle$ (right) where the area of each atom is proportional to its weight.}\label{fig-cmd-npmle}
\end{figure}

Figure~\ref{fig-cmd} in Section~\ref{sec-intro} shows the plot of the observed data $X_i$ (left) and estimated posterior means~$\estvec_i$ (right), the latter constituting the denoised CMD. Contrasting our CMD with theirs \citep[][Figure 7]{anderson2018improving}, which we do not depict here, it appears that ours performs more shrinkage overall. Our CMD has rather sharp tails in the bottom of the plot (i.e. the main sequence) and the top right (i.e. the tip of the red-giant branch) as well as a definitive cluster in the center-right (i.e. the red clump).

There are also important differences between the NPMLE and extreme deconvolution in the estimated prior $\npmle$. Figure~\ref{fig-cmd-npmle} shows the initial and final iterates in the computation of the NPMLE. We are ultimately using a discrete distribution to model the prior, and since all of the covariance matrices $\Sigma_i$ are diagonal, by Proposition~\ref{prop-support} we have restricted the support points to lie in the minimum axis-aligned bounding box of the data. By contrast, extreme deconvolution models the prior as itself a Gaussian mixture, so the estimated prior \citep[][Figure 4]{anderson2018improving} actually is supported on all of $\R^2$. 

\subsection{Chemical abundance ratios}\label{sec-apogee}

Our second data set is taken from the Apache Point Observatory
Galactic Evolution Experiment survey (APOGEE); see \citet{majewski2017apache}, \citet{abolfathi2018fourteenth}. We examine chemical abundance ratios
for the red clump (RC) stars given in the DR14 APOGEE red clump
catalog; see \citet{ratcliffe2020tracing} where this data set has been
studied. Following the pre-processing in \citet{ratcliffe2020tracing} to
remove the outliers with anomalous abundance measurements, the data
set contains $n = 27,238$ observations of $\p=19$ relative chemical abundances. 

\blue{
We restrict our attention to two abundances among the $19$~dimensions, namely, [Si/Fe]-[Mg/Fe]. To explore the dependence of our estimator on the dimensionality of the data, we consider two distinct empirical Bayes analyses. The first fits the NPMLE on all $\p=19$ dimensions, estimating the posterior mean and then plotting only the components of our estimates corresponding to the two abundances of interest [Si/Fe]-[Mg/Fe]. The second analysis discards the $17$ other coordinates and then computes the $2$-dimensional NPMLE on those same abundances.

In Figure~\ref{fig-apogee-npmle} we plot the observed data (left) and estimated posterior
means using Gaussian denoising under the estimated prior $\npmle$. The estimates based on the full $\p=19$ dimensional data set are in the center panel, and the estimates based on only the $2$ dimensions are on the right. In two dimensions, the denoised data reveals a very interesting structure --- it shows that the variables
[Si/Fe] and [Mg/Fe] are strongly correlated. The observations for the upper right cluster of stars could be lying on one dimensional manifold; something that is not at all visible when plotting the original data. On the other hand, the NPMLE based on the full $\p=19$ dimensional data set performs much less aggressive shrinkage. We believe that the dimension-dependent logarithmic factor $(\log n)^{\p/2+1}$ in the Definition~\eqref{eq-eps} of $\epsilon_n^2$ plays a nontrivial role in this relatively small data set. For example, setting $\p=2$ and $n=27,238$, the log factor is approximately $104$, i.e., much smaller than~$n$; on the other hand, changing $\p$ to $19$ in the same formula gives a logarithmic factor of over $39$ billion. The dimension-dependent logarithmic factor is not simply an artifact of our analysis: \citet{kim2022minimax} prove a minimax lower bound for Hellinger accuracy which contains a similar, dimension dependent logarithmic factor.}

\begin{figure}[ht!]
\centering
\includegraphics[width=.9\textwidth]{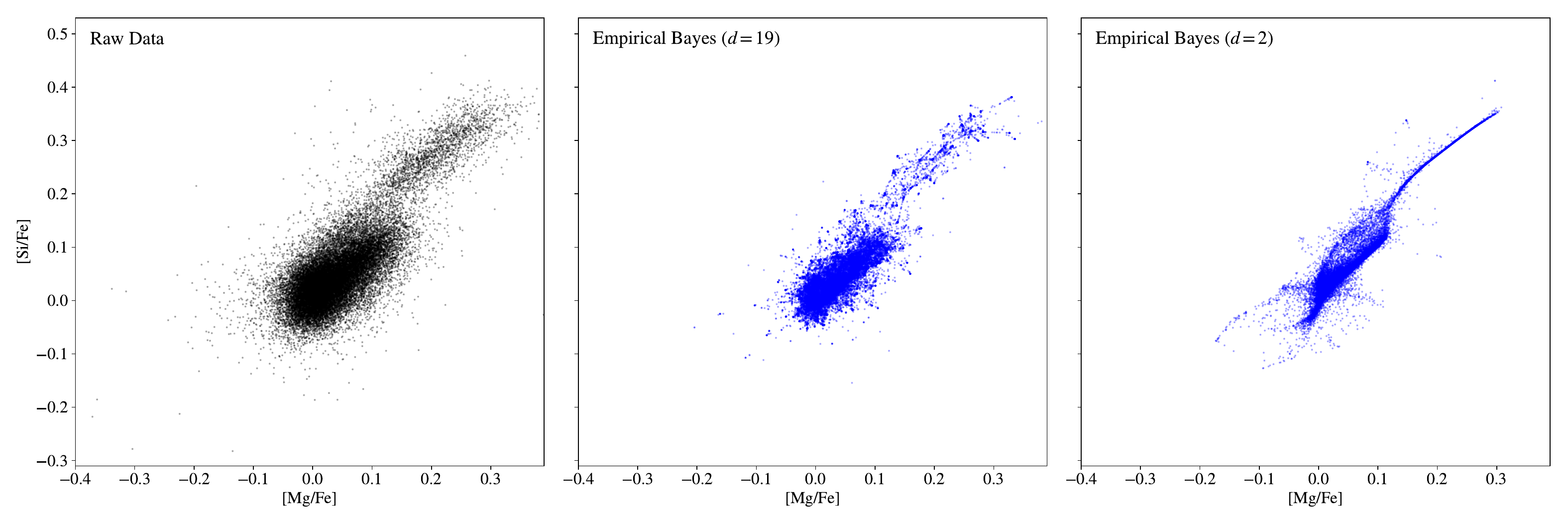}
\caption{Observed data (left) and estimated posterior means based on all $d=19$ dimensions (center) and only $d=2$ dimensions (right).}\label{fig-apogee-npmle}
\end{figure}

\blue{
\subsection{Math scores in U.S. public schools}\label{sec-math-scores}

Our third data set comes from the Education Longitudinal Study of 2002, collected by the National Center for Education Statistics of the U.S. Department of Education. We apply our approach to hierarchical linear models described in Section~\ref{sec-hierarchical-linear-model} and compare to a Bayesian analysis of \citet{hoff2009first}. Specifically, we look at a survey including math test scores and normalized socioeconomic status (SES) of 10th grade students. The survey has information across $n=100$ different large public high schools on a total of $\sum_{i=1}^n N_i = 1,993$ children. We consider a $d=2$ dimensional hierarchical regression setting where the response $y_{ij}$ represents the math score of student $j$ in school $i$, and $X_{ij} = [1, \texttt{SES}_{ij}]$ contains the corresponding SES score and an intercept term. The number of students~$N_i$ surveyed in each school varies significantly, ranging from $4$ to $32$ with a median of $20$ students. The goal is to estimate separate regression coefficients~$\beta^*_i$ for each school~$i=1,\ldots,n$, but if we simply compute OLS~$b_i = (X_i^\mathsf{T}X_i)^{-1}X_i^\mathsf{T}y_i$ separately for each school, then the estimates for schools with small~$N_i$ will be especially poor.

In Figure~\ref{fig-mathscore}, we compare the empirical Bayes regression lines to the OLS estimates. We see that the empirical Bayes method shrinks the bulk of slopes quite aggressively, but they are not shrunk towards a single point. The standard fully Bayesian analysis \citep[which can be seen in Figure~11.3 of][]{hoff2009first} models the prior~$\trueprior$ as a normal distribution, placing further priors on its parameters. An important difference is that all of the regression coefficients are shrunk towards a single point under a normal prior. By contrast, for our method, the multiple locations to shrink towards and their relative weights are determined by the NPMLE~$\npmle$, as depicted in the bottom right panel of Figure~\ref{fig-mathscore}. 

\begin{figure}[p!]
\centering
\includegraphics[width=.9\textwidth]{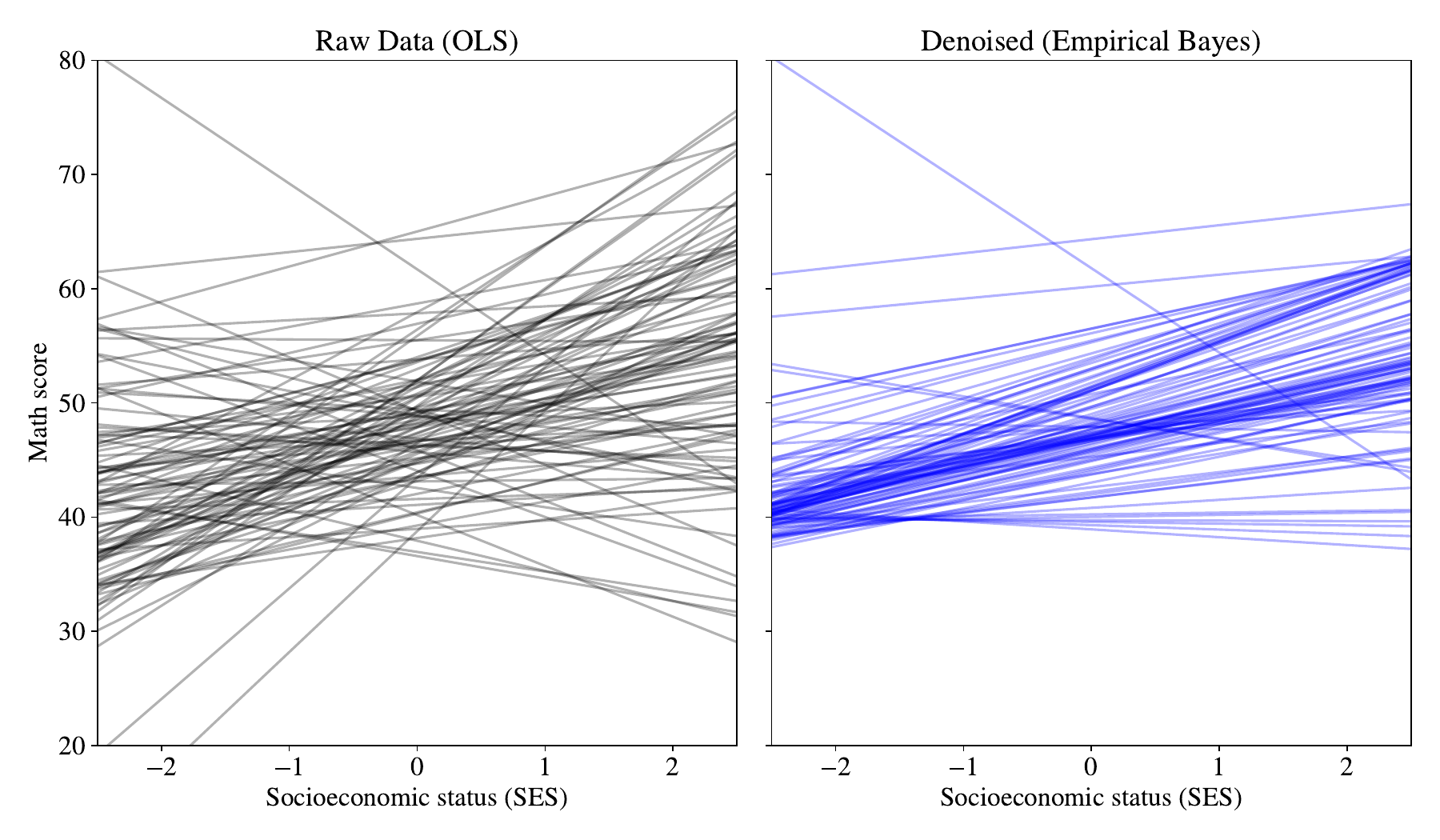}
\caption{\blue{Relationship between socioeconomic status (SES) and math score across different high schools. Above: The left panel shows regression lines for each school based on the OLS estimates~$b_i$ fitted separately. The right panel shows regression lines for each school based on the estimated posterior means~$\hat\beta_i$. \\
Below: The left panel contains the same information as above, but directly plotting the values of~$b_i$ and $\hat{\beta}_i$ rather than the regression lines. The right panel contains the discrete NPMLE~$\npmle$, where the area of each dot is proportional to the corresponding weight.}}\label{fig-mathscore}
\includegraphics[width=.9\textwidth]{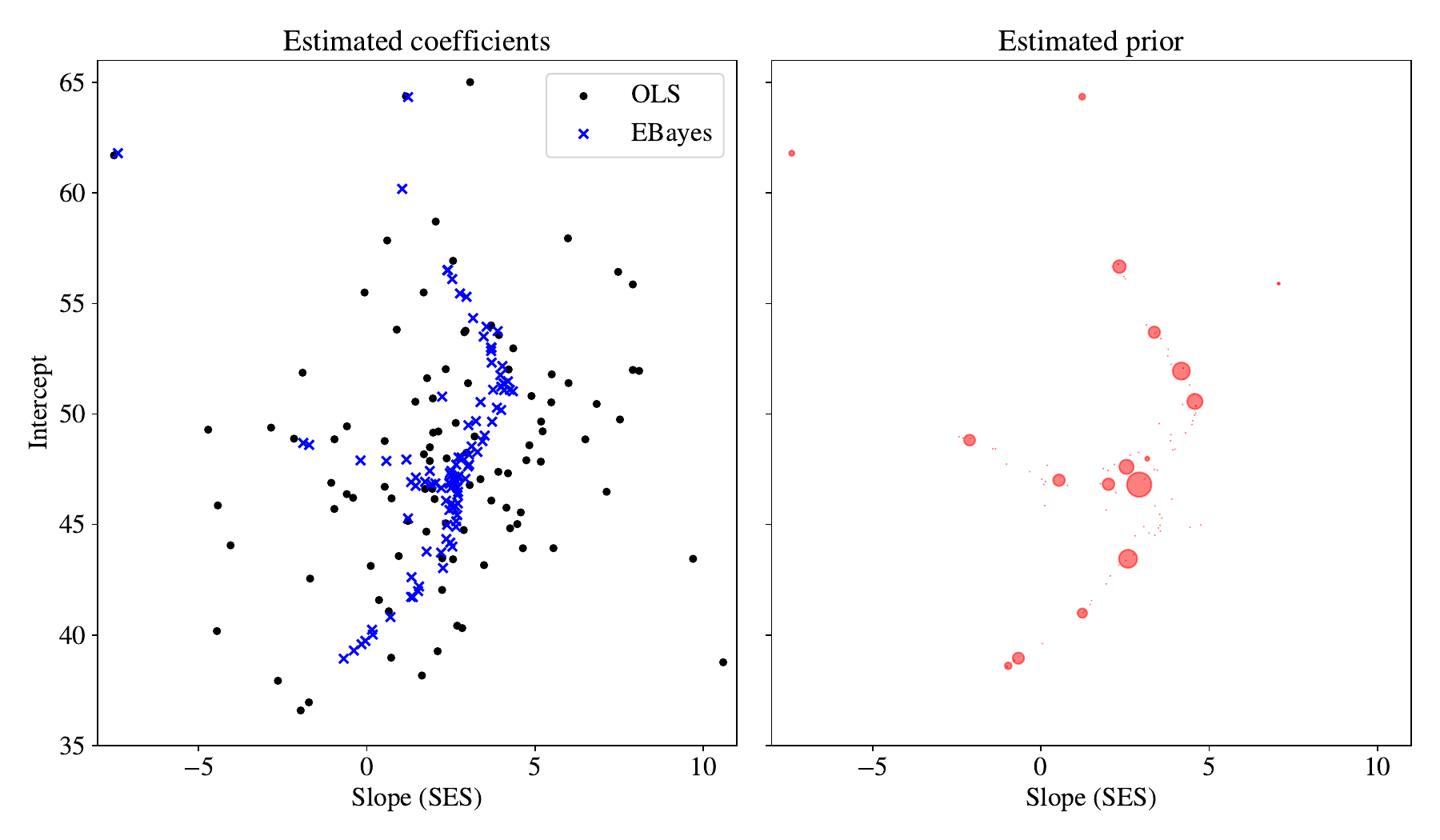}
\end{figure}

\begin{figure}[pht!]
\centering
\includegraphics[width=\textwidth]{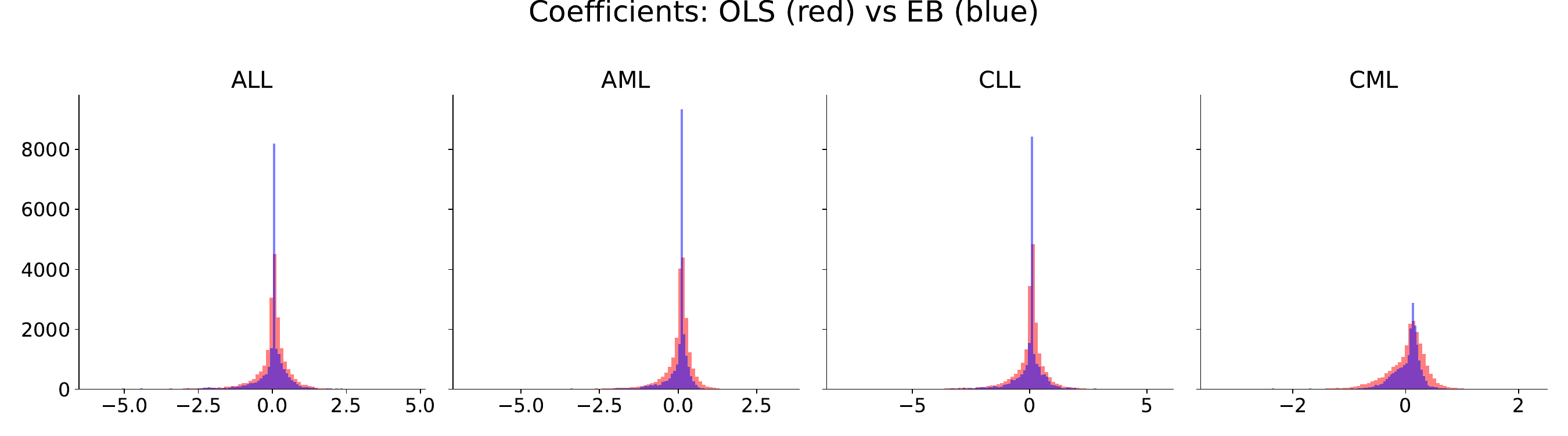}
\blue{(a) Histogram of coefficients estimated with OLS~$b_i$ (red) and empirical Bayes~$\hat\beta_i$ (blue).} \\
~

\includegraphics[width=\textwidth]{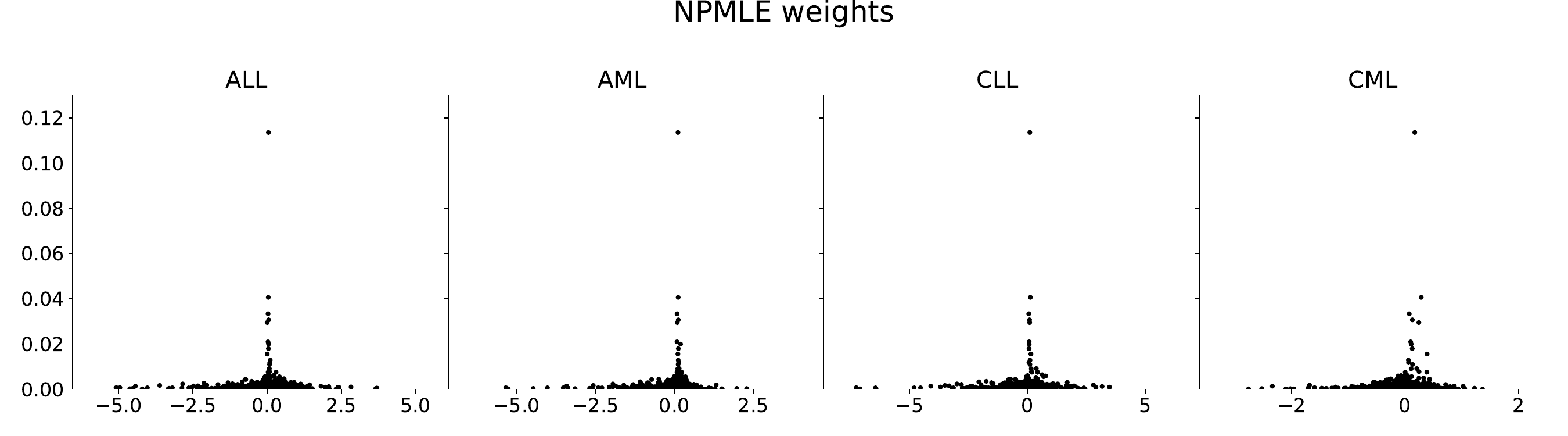}
\blue{(b) The NPMLE weights.} \\
~

\includegraphics[width=.45\textwidth]{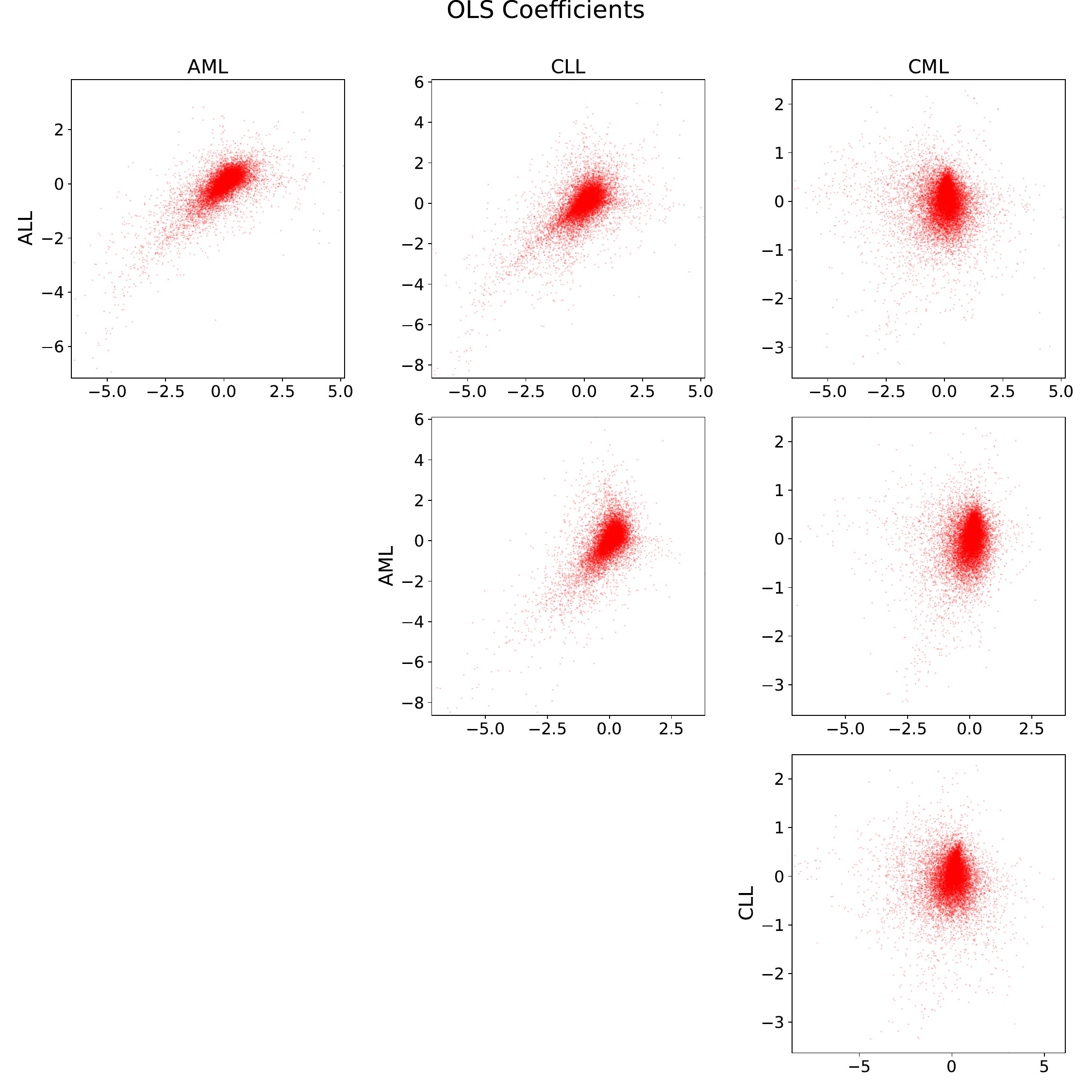}
\includegraphics[width=.45\textwidth]{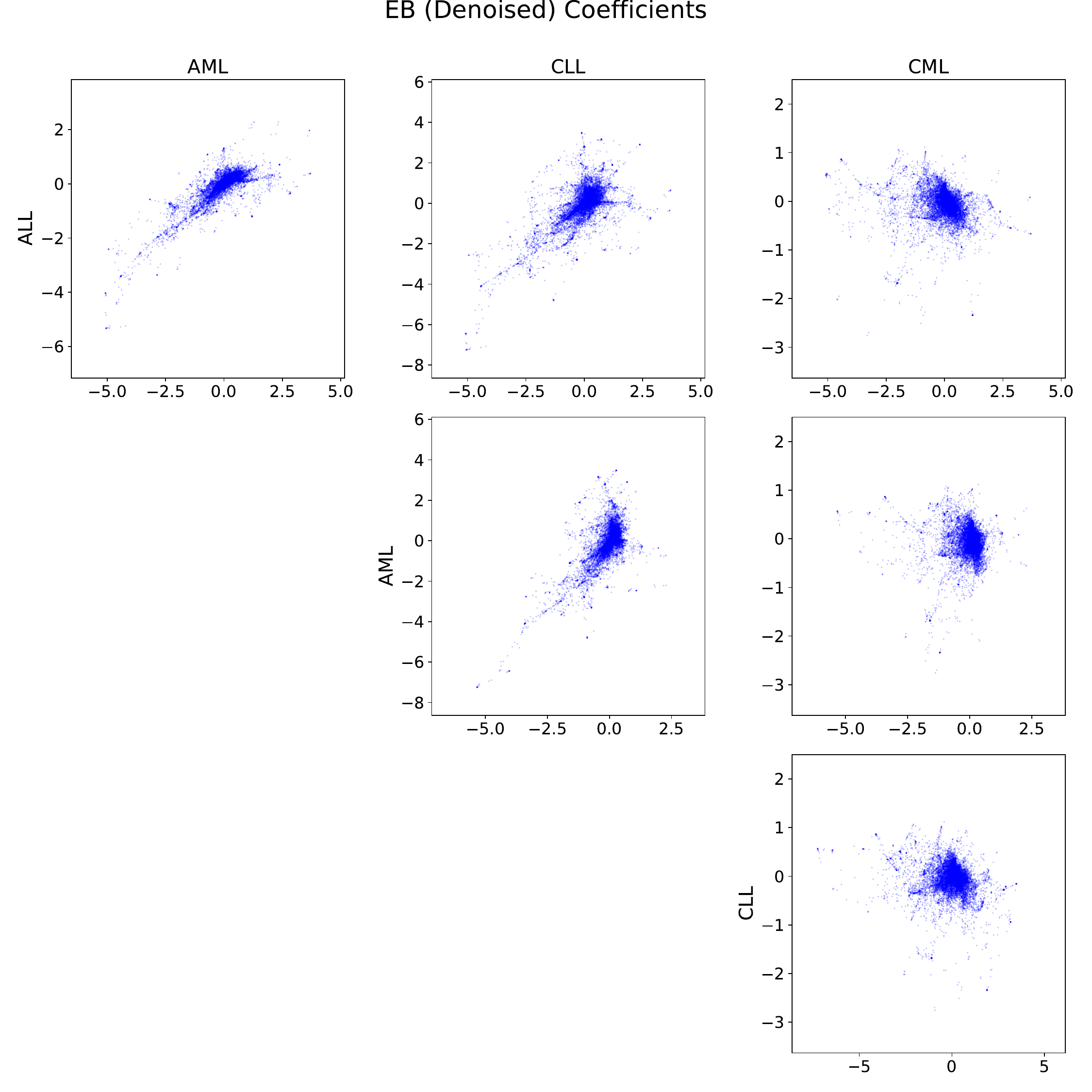}

\blue{(c) Pairwise scatterplots of coefficient estimates for OLS~$b_i$ (red) and empirical Bayes~$\hat\beta_i$ (blue).}  \\ 
~

\caption{Relationship between cancer status and gene expression level for $n = 20,172$ genes. }\label{fig-leukemia}
\end{figure}

\subsection{Microarray data}\label{sec-leukemia}

Our fourth data set \citep{leukemiasEset} is a subset of the Microarray Innovations in Leukemia (MILE) study \citep{kohlmann2008international,haferlach2010clinical}. The data contains $60$ bone marrow samples of patients with one of four types of leukemia  (abbreviated \texttt{ALL, AML, CLL, CML}) as well as non-leukemia patients. We are interested in estimating the contrast of the difference in expression levels for each leukemia type compared to the non-leukemia baseline. We observe $n = 20,172$ genes and $N_i\equiv N := 60$ samples per gene. We use a one-hot encoding for each type as well as an intercept term, so $\p=5$: specifically, for each patient $j\in \{1,\ldots,N\}$ and each leukemia type $\texttt{TYPE}\in \{\texttt{ALL, AML, CLL, CML}\}$, set $X_{j,\texttt{TYPE}} = 1$ if patient $j$ has that cancer type, and $X_{j,\texttt{TYPE}} = 0$ otherwise. We model the expression of gene~$i\in \{1,\ldots,n\}$ in patient~$j\in\{1,\ldots,N\}$ as 
\[
y_{ij} = \beta^*_{i,0} + X_{j,\texttt{ALL}}\beta^*_{i,\texttt{ALL}} + X_{j,\texttt{AML}}\beta^*_{i,\texttt{AML}} + X_{j,\texttt{CLL}}\beta^*_{i,\texttt{CLL}} + X_{j,\texttt{CML}}\beta^*_{i,\texttt{CML}}+\varepsilon_{ij},
\]
where $\varepsilon_{ij}\simiid\cn(0, \sigma^2)$. For each cancer \texttt{TYPE}, $\beta^*_{i,\texttt{TYPE}}$ represents the effect of that cancer type on the expression level of gene~$i$, relative to the non-leukemia baseline. In the hierarchical linear model, we treat the $\p$-dimensional vectors of regression coefficients~$\beta^*_i = (\beta^*_{i,0},\beta^*_{i,\texttt{ALL}},\beta^*_{i,\texttt{AML}},\beta^*_{i,\texttt{CLL}},\beta^*_{i,\texttt{CML}})$ as i.i.d. from some prior~$\trueprior\in \mathcal{P}(\R^\p)$. As in Section~\ref{sec-hierarchical-linear-model}, we first compute the OLS estimates~$b_i\in \R^p$ separately for each gene~$i$. Next, we compute the pooled estimate~$\hat\sigma^2$ of the observation variance~$\sigma^2$. Finally, we compute the exemplar+ NPMLE~$\npmle$ and the estimated posterior means~$\hat\beta_i$.

In Figure~\ref{fig-leukemia}, we compare the empirical Bayes estimates to the OLS estimates. From both the marginal histograms in \ref{fig-leukemia}(a) and the plots of the weights in \ref{fig-leukemia}(b), it is clear that the NPMLE promotes sparsity by shrinking coefficients towards the spikes near zero. Three of the four types have a spike very close to zero, whereas the fourth type (CML) places its spike on small but positive effects. In genomics applications where we are interested in multiple effects, we can consider multiple kinds of sparsity. The first kind is a strong form of sparsity, where $\beta^*_i = 0$ for most~$i$. For the second, $\beta^*_{i,\texttt{TYPE}} = 0$ for most~$i\in \{1,\ldots,n\}$ and most types, but for any given gene~$i$, $\beta^*_{i,\texttt{TYPE}} = 0$ need not imply $\beta^*_{i,\texttt{TYPE}'} = 0$ for some other type $\texttt{TYPE}'$. The NPMLE permits both forms of sparsity, so we do not need to know a priori which form is more appropriate.
}

\section{Concluding remarks}\label{sec-discussion}

In this paper we study the NPMLE~$\npmle$ as an estimator of a prior distribution~$\trueprior$ in the presence of multivariate, heteroscedastic measurement errors. We resolve a number of basic questions on the existence, uniqueness, and support of the NPMLE, where in several cases the answers differ significantly from the traditional univariate, homoscedastic setting. Our analysis identifies a dual mixture density~$\widehat\psi_n$ with Gaussian~$\cn(X_i,\Sigma_i)$ components at each observation, whose modes contain the atoms of the NPMLE. Our characterization implies that the NPMLE is supported in the ridgeline manifold $\cm$, which is a compact subset of~$\R^\p$ defined in terms of the observations $(X_i)_{i=1}^n$ and corresponding covariance matrices $(\Sigma_i)_{i=1}^n$. This support reduction allows us to approximate the NPMLE by a finite-dimensional convex optimization over the mixing proportions, and we develop a novel approach to bounding the discretization error, justifying the gridding scheme proposed by~\cite{koenker2014convex}. Our real data applications show that this approach is viable for practical astronomy problems. Our theoretical results in Section~\ref{sec-stats} provide strong justification for using the NPMLE in a variety of contexts---estimating the prior, marginal densities, and oracle posterior means. 

We conclude by outlining some possible future research directions. Computation remains an important barrier for large-scale applications. Specifically, for problems with a large number of samples, e.g. $n\gg 10^6$, some additional forms of approximation are warranted, such as stochastic optimization or binning via coresets (see also \cite{ritchie2019scalable} on approaches for scaling Extreme Deconvolution to large data sets). \blue{Another open problem is to come up with more principled approach to the exemplar+ method, specifically for choosing the counts~$(N_j)$ and for sampling the weights~$\alpha$'s.} 

Next, while our framework allows the prior~$\trueprior$ to be arbitrary, the underlying assumption---that the means~$(\truevec_i)$   are {\sl identically distributed}---can sometimes be difficult to justify for heteroscedastic observations. The i.i.d. assumption reflects the belief that the observation covariance~$\Sigma_i$ is uninformative for the corresponding mean~$\truevec_i$. This assumption led to reasonable results in our applications but may be problematic in other settings. In the univariate, heteroscedastic case, \citet{weinstein2018group} proposed grouping observations with similar variances and applying a spherically symmetric estimator separately within each group. Their approach is capable of capturing dependence between $\truevec_i$ and $\sigma_i^2$, at the expense of not sharing information across groups. \blue{Moving beyond binning the variances~$\sigma_i^2$, \cite{chen2023empirical} models the conditional distribution~$\truevec_i\mid \sigma_i^2$ as a flexible location-scale family.} To our knowledge, these approaches have not been extended to multivariate settings. Thus, in multivariate settings there remains the important problem of how to model the relationship between $\theta_i$ and $\Sigma_i$.

Finally, there remain a number of open statistical questions for future work. Our analysis of the denoising problem focuses on estimating the posterior mean based on the unknown prior~$\trueprior$, but there are numerous inferential goals one could target with an approximate prior. The analyst might summarize the empirical posteriors using a different functional, such as the posterior median or the posterior mean of some transformed parameter. This question warrants a more general analysis evaluating the quality of the empirical posterior distributions for the true, unknown posteriors.

\section*{Acknowledgements}

We thank Bridget L. Ratcliffe for providing both astronomy data sets and Nikos Ignatiadis for suggesting the microarray data set. J.A.S. would like to thank Jacob Steinhardt and Serena Wang for their valuable feedback on an early draft. We are grateful to Kevin Chen for correcting a step in the proof of Theorem~\ref{thm-denoising}, to Yihong Wu for correcting the statement of Proposition~\ref{prop-support}.

\section*{Funding}

J.A.S. was supported by NSF Grant DMS-2023505 and by the Office of Naval Research under the Vannevar Bush Fellowship program under grant number N00014-21-1-2941. A.G. was supported by NSF CAREER Grant DMS-16-54589. B.S. was supported by NSF Grant DMS-2015376.

\appendix
\section{Additional experiments}\label{sec-additional-experiments}

\blue{
We ran a series of simulation experiments to assess the plausibility of Assumption~\eqref{eq-hellinger-accuracy-approximate-npmle} of Theorem~\ref{thm-density-estimation}, which requires the likelihood~$\ell_n(\npmle)$ of the approximate NPMLE to be nearly as large as the likelihood~$\ell_n(\trueprior)$. We compute the exemplar+ NPMLE for simulated homoscedastic data sets, where the covariances are $\Sigma_i = s I_d$ for all $i\in \{1,\ldots, n\}$. In Figure~\ref{fig-npmle-moderate-dims-example}, we plot the difference $\ell_n(\widehat{G}_n) - \ell_n(G^*)$ as a function of the noise level~$s$, for various choices of~$\trueprior$ and dimensions~$d\in\{2, 5, 20\}$. The gap $\ell_n(\widehat{G}_n) - \ell_n(G^*)$ was positive across all experimental conditions: in fact, the gap between the two likelihoods grows considerably in high dimensions. We thus believe that the exemplar+ NPMLE is suitable for our applications.

\begin{figure}[ht!]
\centering
\includegraphics[width=.3\textwidth]{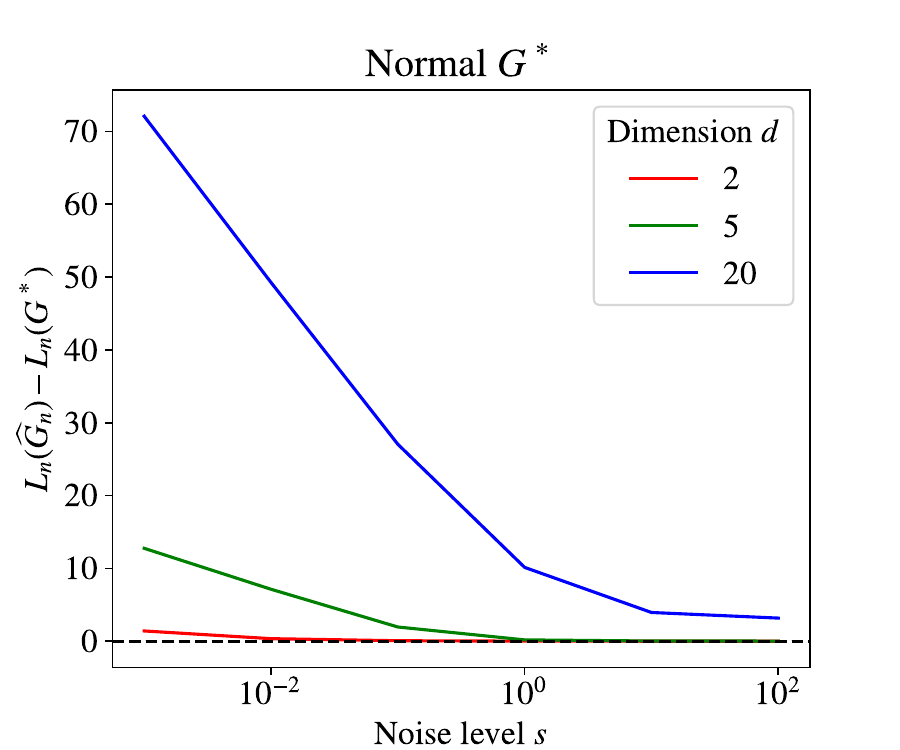}
\includegraphics[width=.3\textwidth]{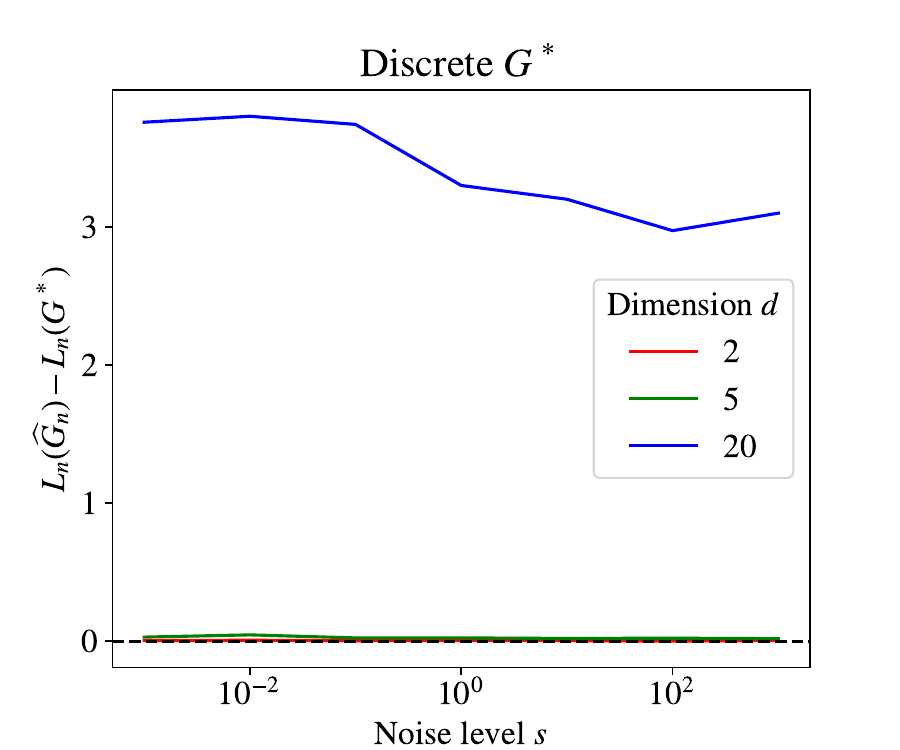}
\includegraphics[width=.3\textwidth]{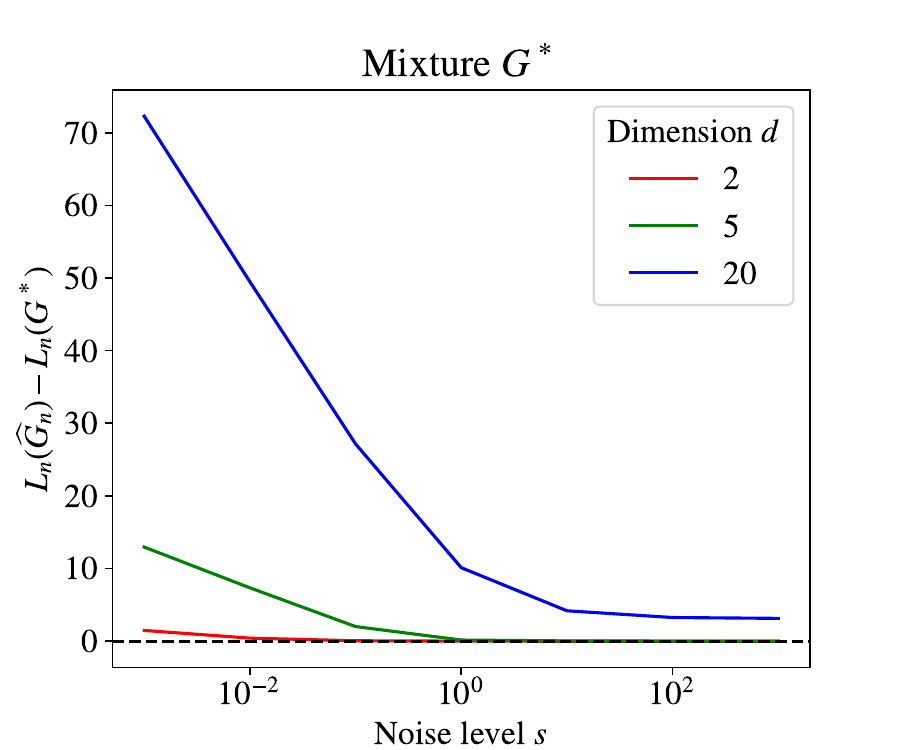}
\caption{\blue{The difference in likelihoods $\ell_n(\widehat{G}_n) - \ell_n(G^*)$ of the approximate NPMLE~$\widehat{G}_n$ and the true prior~$G^*$ as a function of the noise level~$s$ (where the observation covariance is~$\Sigma = sI_d$), for different dimensions~$d\in \{2, 5, 20\}$. Each panel shows a different choice of~$G^*$: from left to right, normal~($G^* = \mathcal{N}(0,I_d)$), discrete~($G^* = \frac{1}{2}\delta_{0} + \frac{1}{2}\delta_{{\bf 1}_{d}/\sqrt{d}}$) and a mixture of normals ($G^* = \frac{1}{2}\mathcal{N}(0,I_d) + \frac{1}{2}\mathcal{N}({\bf 1}_{d}/\sqrt{d},I_d)$). 
}}\label{fig-npmle-moderate-dims-example} 
\end{figure}

}

\section{Proofs of Results in Sections~\ref{sec-characterization} and~\ref{sec-computation}}\label{sec-proofs-characterization}

\subsection{Proof of Lemma~\ref{lem-characterization}}

The following uses similar techniques to those of \citet[][Section~5.2]{lindsay1995mixture}, which contains a subset of our result in the homoscedastic case.

\begin{proof}[Proof of Lemma~\ref{lem-characterization}] By convexity, the first-order optimality condition for $\npmle$ is 
\begin{align*}\label{eq-first-order-cvx}
D(\npmle, \aprior)\le 0 \text{ for all } \aprior\in \Ps(\R^\p)
\end{align*}
where 
\begin{align*}
D(\npmle, \aprior)
&\coloneqq \lim_{\alpha\downarrow 0}\frac{\frac{1}{n}\sum_{i=1}^n [\log f_{(1-\alpha)\npmle + \alpha \aprior, \Sigma_i}(X_i) - \log f_{\npmle, \Sigma_i}(X_i)]}{\alpha} \\
&=\frac{1}{n}\sum_{i=1}^n \frac{1}{f_{\npmle, \Sigma_i}(X_i)}\left(f_{\aprior, \Sigma_i}(X_i)- f_{\npmle, \Sigma_i}(X_i)\right)
=\frac{1}{n}\sum_{i=1}^n \frac{f_{\aprior, \Sigma_i}(X_i)}{f_{\npmle, \Sigma_i}(X_i)} - 1
\end{align*}
When $\aprior = \delta_{\vartheta}$ is a point mass we write $D(\npmle, \vartheta)$ instead of $D(\npmle, \aprior)$. It suffices to check $D(\npmle, \vartheta)\le 0$ for all $\vartheta\in\R^\p$ because $D(\npmle, \aprior) = \int D(\npmle,\vartheta)\diff G[\vartheta]$.

For the first part of the Lemma, define $\cc \coloneqq \left\{(f_{\aprior, \Sigma_i}(X_i))_{i=1}^n : \aprior\in \Ps(\R^\p)\right\}\cup\{0\}$. Observe that
\[
\cc = \conv\left(\cl\right), \text{ where } \cl\coloneqq \left\{(\varphi_{\Sigma_i}(X_i - \vartheta))_{i=1}^n : \vartheta\in \R^\p\right\} \cup \{0\}.
\]
Since $\vartheta \mapsto (\varphi_{\Sigma_i}(X_i - \vartheta))_{i=1}^n$ is continuous and $\lim_{\|\vartheta\|_2\to \infty}(\varphi_{\Sigma_i}(X_i - \vartheta))_{i=1}^n = 0$, the set $\cl$ is closed, and by boundedness of the Gaussian likelihood, $\cl$ is compact. Hence $\cc\subset \R^n$ is convex and compact, and $f(L) = \frac{1}{n}\sum_{i=1}^n \log L_i$ is strictly concave \blue{and coordinate-wise monotone} over $\cc$. Thus, $f$ attains its maximum at a unique (non-zero) boundary point~$\hat{L}\in \partial\cc$. Observe $\cc = \conv\left(\left\{(\varphi_{\Sigma_i}(X_i - \vartheta))_{i=1}^n : \vartheta\in \R^\p\right\}\cup \{0\}\right)$: by Carath\'eodory's theorem, any boundary point $\hat{L}\in \partial\cc$ can be written as $\hat{L}_i = \sum_{j=1}^{\khat}\hat{w}_j\varphi_{\Sigma_i}(X_i - \hat{a}_j)$ for some $\khat\le n$. 

Suppose $B\subset \mathrm{supp}(\npmle)$ is contained in the support of the NPMLE. Given $\npmle(B) > 0$, define a new probability measure $\npmle^B$ via $\npmle^B(A) \coloneqq \frac{\npmle(A\cap B)}{\npmle(B)}$. Since $\npmle = \alpha_0\npmle^B + (1-\alpha_0)\npmle^{B^\mathsf{c}}$ for $\alpha_0 = \npmle(B)$, the mixture
\[
\aprior_\alpha 
= (1-\alpha)\npmle  + \alpha \npmle^B
\]
remains a valid probability measure for $\alpha \ge -\frac{\alpha_0}{1-\alpha_0}$. Since $\alpha = 0$ maximizes the log-likelihood of $\aprior_\alpha$ over a range $\alpha\in [-\frac{\alpha_0}{1-\alpha_0}, 1]$ including both negative and positive values, the derivative of the log-likelihood is zero at $\alpha=0$, i.e. \[0 = D(\npmle, \npmle^B) = \int D(\npmle,\vartheta)\diff \npmle^B[\vartheta],\] so $\npmle^B(\mathcal{Z}) = 1$ for all $B\subset \mathrm{supp}(\npmle)$ such that $\npmle(B) > 0$. This implies $\npmle(B\cap\mathcal{Z}) = \npmle(B)$ for all measurable $B$, from which we may conclude $\mathcal{Z}\supseteq \mathrm{supp}(\npmle)$. Finally, observe that 
\[D(\npmle, \vartheta)=\frac{1}{n}\sum_{i=1}^n \hat{L}_i^{-1}\varphi_{\Sigma_i}(X_i - \vartheta) - 1 =  \left(\frac{1}{n}\sum_{i=1}^n\hat{L}_i^{-1}\right)\widehat\psi_n(\vartheta)-1,\] 
so $D(\npmle, \vartheta)\le 0$ is equivalent to $\widehat\psi_n(\vartheta)\le  \left(\frac{1}{n}\sum_{i=1}^n\hat{L}_i^{-1}\right)^{-1}$. This proves the last statement of the Lemma, that $\mathcal{Z}$ is equal to the set of global maximizers of $\widehat\psi_n$.
\end{proof}

\blue{
\begin{remark}\label{rmk-characterization} Inspecting the proof, we can extend the characterization of Lemma~\ref{lem-characterization} to a much broader class of heteroscedastic likelihoods. Suppose $\zeta_1,\ldots,\zeta_n$ are known, strictly positive, bounded and continuous densities on $\R^d$, and let $f_{G, i}(x) = \int \zeta_i(x-\vartheta)\text{d}G[\vartheta]$. Then the structural results of Lemma~\ref{lem-characterization} apply to the NPMLE 
\[\npmle
 = \argmin_{G\in \cp(\R^d)} \frac{1}{n}\sum_{i=1}^n \log f_{G,i}(X_i)\]
 as long as $\lim_{\|\vartheta\|_2\to \infty}(\zeta_{i}(X_i - \vartheta))_{i=1}^n = 0$.
\end{remark}
}

\subsection{Proof of Lemma~\ref{rem-nonunique}}

\begin{proof}[Proof of Lemma~\ref{rem-nonunique}] By Lemma~\ref{lem-invariance}, the fitted values $\hat{L}_1 = \hat{L}_2 = \hat{L}_3$ are equal. By Lemma~\ref{lem-characterization}, the atoms of $\npmle$ occur at the global modes of $\widehat\psi_n = f_{H, \sigma^2I_2}$, where $H = \frac{1}{3}\sum_{i=1}^3 \delta_{X_i}$. Since $\hat{L}_1=\hat{L}_2 = \hat{L}_3$, the fitted values are also equal to the global maximum of $\widehat\psi_n$, i.e. 
\[\hat{L}_i 
= \max_xf_{H,\sigma^2I_2}(x)
= \frac{2^{2/3}\log 2}{3\pi}
\]
for each $i=1,2,3$. Note that $\hat{L}_i = f_{\delta_0, \sigma^2I_2}(X_i)$ for all $X_i$, so $\npmle = \delta_0$ is an NPMLE. Now let $\npmle' = \frac{1}{3}\sum_{i=1}^n\delta_{X_i/2}$. It suffices to check the fitted values of $\npmle'$ at the observations. For $i=1$,
\begin{align*}
f_{\npmle', \sigma^2I_2}(X_1) 
&= \frac{1}{3}\sum_{i=1}^3\varphi_{\sigma^2I_2}(X_1 - X_i/2)  \\
&= \frac{4\log 2}{9\pi}(2^{-(4/3)(1/4)}+ 2^{-(4/3)(7/4)} + 2^{-(4/3)(7/4)}) \\
& =\frac{2^{2/3}\log 2}{3\pi} = \hat{L}_1.
\end{align*}
Similarly, for $i=2$, 
\begin{align*}
f_{\npmle', \sigma^2I_2}(X_2) 
&= \frac{1}{3}\sum_{i=1}^3\varphi_{\sigma^2I_2}(X_2 - X_i/2) \\
&= \frac{4\log 2}{9\pi}(2^{-(4/3)(7/4)}+ 2^{-(4/3)(1/4)} + 2^{-(4/3)(7/4)}) \\
&=\frac{2^{2/3}\log 2}{3\pi} = \hat{L}_2,
\end{align*}
and, for $i=3$,
\begin{align*}
f_{\npmle', \sigma^2I_2}(X_3) 
&= \frac{1}{3}\sum_{i=1}^3\varphi_{\sigma^2I_2}(X_3 - X_i/2) \\
&= \frac{4\log 2}{9\pi}(2^{-(4/3)(7/4)}+ 2^{-(4/3)(7/4)} + 2^{-(4/3)(1/4)}) \\
&=\frac{2^{2/3}\log 2}{3\pi} = \hat{L}_3.
\end{align*}
This verifies that $\npmle' = \frac{1}{3}\sum_{i=1}^n\delta_{X_i/2}$ is also an NPMLE, so every convex combination $\alpha\npmle + (1-\alpha)\npmle'$ is an NPMLE.
\end{proof}

\blue{
\subsection{Proof of Lemma~\ref{lem-ridgeline}}

\begin{proof}[Proof of Lemma~\ref{lem-ridgeline}] The fact that $\cz\subset \cm$ follows \citet[][Theorem 1]{ray2005topography}. Observe that $\cm$ is compact as it is the continuous image of the simplex, a compact set. 

For the last claim, let $p = x^*(\alpha)$, and let $Y_i = \Sigma_i^{-1}(X_i-p)$. Note $p=x^*(\alpha)$ implies $0 = \sum_{i=1}^n \alpha_iY_i$:
\[
p = x^*(\alpha)
\iff \sum_{i=1}^n \alpha_i\Sigma_i^{-1}p = \sum_{i=1}^n \alpha_i\Sigma_i^{-1}X_i
\iff 0 = \sum_{i=1}^n \alpha_iY_i.
\]
By Carath\'eodory's theorem, there is some $\alpha'\in \Delta_{n-1}$ with at most $d+1$ nonzeros such that $0=\sum_{i=1}^n \alpha_i'Y_i$. Rearranging, we find that $p = x^*(\alpha')$.
\end{proof}
}

\subsection{Proof of Proposition~\ref{prop-support}}

\begin{proof}[Proof of Proposition~\ref{prop-support}] In the proportional covariances case $\Sigma_i= c_i\Sigma$, we have 
\begin{align*}
x^*(\alpha)
&= \left(\sum_{i=1}^n \alpha_i\Sigma_i^{-1}\right)^{-1}\sum_{i=1}^n\alpha_i\Sigma_i^{-1}X_i \\
&= \sum_{i=1}^n\frac{\alpha_i/c_i}{\sum_{\iota=1}^n \alpha_\iota/c_\iota}X_i 
\end{align*}
As $\alpha$ ranges over the simplex, so does $\left(\frac{\alpha_i/c_i}{\sum_{\iota=1}^n \alpha_\iota/c_\iota}\right)_{i=1}^n$. Thus $\cm = \conv(\{X_1,\dots,X_n\})$, proving (i). If each~$\Sigma_i$ is diagonal, letting $x^*_j(\alpha)$ denote the $j^\text{th}$ coordinate of $x^*(\alpha)\in \R^\p$,
\[
x^*_j(\alpha) = \sum_{i=1}^n \frac{\alpha_i (\Sigma_i)_{jj}}{\sum_{i'=1}^n \alpha_{i'} (\Sigma_{i'})_{jj}}X_{ij} \in \left[\min_{i\in \{1, \dots, n\}} X_{ij}, \max_{i\in \{1, \dots, n\}} X_{ij}\right],
\]
proving (ii). For (iii), using concavity of the minimum eigenvalue,
\begin{align*}
\|x^*(\alpha)-x\|_2 
&= \left\|\left(\sum_{i=1}^n \alpha_i \Sigma_i^{-1}\right)^{-1}\sum_{i=1}^n \alpha_i \Sigma_i^{-1}(X_i-x)\right\|_2 \\
&\le \left\|\left(\sum_{i=1}^n \alpha_i \Sigma_i^{-1}\right)^{-1}\right\|_2 \left\|\sum_{i=1}^n \alpha_i \Sigma_i^{-1}(X_i-x)\right\|_2 \\
&\le \left(\sum_{i=1}^n \alpha_i \overline{k}^{-1}\right)^{-1} \sum_{i=1}^n \alpha_i\underline{k}^{-1}\left\|X_i-x\right\|_2 \le  \kappa r
\end{align*}
so $\M\subseteq \mathbb{B}_{\kappa r}(x)$.
\end{proof}

\subsection{Proof of Lemma~\ref{lem-invariance}}

\begin{proof}[Proof of Lemma~\ref{lem-invariance}] By the change of variables formula,
\begin{align*}
f_{T_\#\aprior, \Sigma_i'}(X_i') 
&= \int \varphi_{U_0\Sigma_iU_0^\mathsf{T}}(U_0X_i + x_0 - \theta)\diff T_\#\aprior(\theta) \\
&= \int \varphi_{U_0\Sigma_iU_0^\mathsf{T}}(U_0X_i + x_0 - T(\theta))\diff \aprior(\theta) \\
&= \int \varphi_{\Sigma_i}(X_i - \theta)\diff \aprior(\theta) = f_{\aprior, \Sigma_i}(X_i),
\end{align*}
completing the proof.
\end{proof}

\subsection{Proof of Proposition~\ref{prop-approximation}}

\begin{proof}[Proof of Proposition~\ref{prop-approximation}] Write $\npmle = \sum_{j=1}^{\khat} \hat{w}_j\delta_{\hat{a}_j}$, and for each $j\in[\khat]$, let $C_j\in \ch$ such that $\hat{a}_j\in C_j$. Next, define a positive measure $H_j$ supported on the corners of $C_j$ such that $H_j(C_j) = \hat{w}_j$ and 
\begin{align}\label{eq-moment-match-for-approx-npmle}
\int_{C_j}u\diff H_j(u) = \hat{w}_j\hat{a}_j = \int_{C_j}u\diff \npmle^j(u),
\end{align}
where $\npmle^j \coloneqq \hat{w}_j\delta_{\hat{a}_j}$. Now fix $u\in C_j$ and $i\in [n]$, and let $x_j = \Sigma_i^{-1/2}(X_i - \hat{a}_j)$ and $t = \Sigma_i^{-1/2}(u-\hat{a}_j)$. By the moment identity~\eqref{eq-moment-match-for-approx-npmle} and by~\citet[][A.27]{jiang2009general},
\begin{align*}
\int_{C_j} \varphi_{\Sigma_i}(X_i-u)\diff \npmle^j(u) &- \int_{C_j} \varphi_{\Sigma_i}(X_i-u)\diff H_j(u) \\
&\le \int_{C_j} \<x_j, t\>^2\varphi_{\Sigma_i}(X_i-u)\diff \npmle^j(u)
+\int_{C_j} \left(e^{\|t\|_2^2/2}-1\right)\varphi_{\Sigma_i}(X_i-u)\diff H_j(u) \\
&\le \underline{k}^{-2}D^2\p\delta^2\int_{C_j} \varphi_{\Sigma_i}(X_i-u)\diff \npmle^j(u)
+\left(e^{\underline{k}\p\delta^2/2}-1\right)\int_{C_j}\varphi_{\Sigma_i}(X_i-u)\diff H_j(u).
\end{align*}
Let $H = \sum_{j=1}^{\khat}H_j$. Summing the above inequality over~$j$, 
\[
f_{\npmle, \Sigma_i}(X_i) - f_{H, \Sigma_i}(X_i) 
\le \underline{k}^{-2}D^2\p\delta^2f_{\npmle, \Sigma_i}(X_i)
+\left(e^{\underline{k}\p\delta^2/2}-1\right)f_{H, \Sigma_i}(X_i).
\]
Since $H$ is supported on $\ca$, by optimality of $\npmle^\ca$, 
\[
\prod_{i=1}^n f_{\npmle^{\ca}, \Sigma_i}(X_i) \ge \prod_{i=1}^n f_{H, \Sigma_i}(X_i).
\]
Combining our findings,  
\[
\prod_{i=1}^n f_{\npmle^{\ca}, \Sigma_i}(X_i) 
\ge e^{-n\underline{k}^{-2}\p\delta^2/2}\left(1-\underline{k}^{-2}D^2\p\delta^2\right)^n\prod_{i=1}^n f_{\npmle, \Sigma_i}(X_i).
\]
Using the elementary inequality $1-x\ge e^{-2x}$ for $x\le 3/4$, we obtain 
\[
\prod_{i=1}^n f_{\npmle^{\ca}, \Sigma_i}(X_i) 
\ge \exp\left(-n\underline{k}^{-2}\p\delta^2/2 -2n\underline{k}^{-2}D^2\p\delta^2\right)\prod_{i=1}^n f_{\npmle, \Sigma_i}(X_i).
\]
for $\delta \le \sqrt{\frac{3}{4\p}}\underline{k}D^{-1}$.
\end{proof}

\section{Proof of Theorem \ref{thm-density-estimation}}\label{sec-proofs-density-estimation}

The following notation will be used throughout this section:
\begin{enumerate}
\item $\mathbb{B}_r(x) = \{y\in \R^\p : \|x-y\|_2\le r\}$ denotes a closed ball in $\R^\p$.
\item For a positive integer $m$, let $[m] = \{1, \dots, m\}$.
\item Given a pseudo-metric space $(M, \rho)$ and $\eps > 0$, let $N(\eps, M, \rho)$ denote the $\eps$-covering number, i.e. the smallest positive integer~$N$ such that there exist $x_1,\dots,x_N\in M$ such that~\[M\subset\bigcup_{i=1}^N\{y : \rho(y, x_i)\le \eps\}.\] 
Any such a set $\{x_i\}_{i=1}^N$ is known as an $\eps$-net or $\eps$-cover of $M$ under the pseudo-metric $\rho$. When~$M$ is a subset of Euclidean space we write $N(\eps, M)$ instead of $N(\eps, M, \|\cdot\|_2)$.
\item We use the shorthand $f_{\aprior, \bullet} = (f_{\aprior, \Sigma_i})_{i=1}^n$, the matrices $\Sigma_1,\dots, \Sigma_n$ being viewed as fixed. Let \[\F = \left\{f_{G, \bullet} : G\in \Ps(\R^\p)\right\}.\]
\item For $S\subset\R^\p$ and $M > 0$, let $S^M$ denote the $M$-enlargement $S^M = \{x\in \R^\p : \mathfrak{d}_S( x) \le M\}$. 
\item Define the semi-norm
\[
\|f_{G, \bullet} - f_{H, \bullet}\|_{\infty, S^M}
\coloneqq \max_{1\le i\le n} \sup_{x\in S^M} |f_{G, \Sigma_i}(x) - f_{H, \Sigma_i}(x)|.
\]
Similarly, define 
\[
\|f_{G, \bullet} - f_{H, \bullet}\|_{\nabla, S^M}
\coloneqq \max_{1\le i\le n} \sup_{x\in S^M} |\nabla f_{G, \Sigma_i}(x) - \nabla f_{H, \Sigma_i}(x)|.
\]
\end{enumerate}

Our proof generalizes and builds upon prior techniques for analyzing the Hellinger accuracy of the NPMLE \citep{zhang2009generalized, saha2020nonparametric, jiang2020general}. The basic structure of our argument is to recognize, given the approximation~\eqref{eq-hellinger-accuracy-approximate-npmle} in the likelihood, that we may trivially rewrite the large deviation probability for the NPMLE as a joint probability
\[
\P\bigg(\bar{h}(f_{\npmle, \bullet}, f_{\trueprior, \bullet}) \gtrsim_{\p, \overline{k}, \underline{k}}  t\eps_n\bigg)
= \P\bigg(\bar{h}(f_{\npmle, \bullet}, f_{\trueprior, \bullet}) \gtrsim_{\p, \overline{k}, \underline{k}}  t\eps_n, \prod_{i=1}^n \frac{f_{\npmle,\Sigma_i}(X_i)}{f_{\trueprior,\Sigma_i}(X_i)}\ge \exp\left(-c_{\p, \overline{k}, \underline{k}}n\eps_n^2\right)\bigg).
\]
If $\npmle$ were a fixed probability measure $\aprior_0$ such that $\bar{h}(f_{\aprior_0, \bullet}, f_{\trueprior, \bullet}) \gtrsim_{\p, \overline{k}, \underline{k}}  t\eps_n$, the right-hand side of the last display similarly simplifies as
\[
\P\bigg(\bar{h}(f_{\aprior_0, \bullet}, f_{\trueprior, \bullet}) \gtrsim_{\p, \overline{k}, \underline{k}}  t\eps_n, \prod_{i=1}^n \frac{f_{\aprior_0,\Sigma_i}(X_i)}{f_{\trueprior,\Sigma_i}(X_i)}\ge \exp\left(-c_{\p, \overline{k}, \underline{k}}n\eps_n^2\right)\bigg)
=\P\bigg(\prod_{i=1}^n \frac{f_{\aprior_0,\Sigma_i}(X_i)}{f_{\trueprior,\Sigma_i}(X_i)}\ge \exp\left(-c_{\p, \overline{k}, \underline{k}}n\eps_n^2\right)\bigg).
\]
Since $\npmle$ is not fixed, we first approximate it using a covering argument, and then bound the right-hand side of the previous display using Markov's inequality.

\begin{proof}[Proof of Theorem \ref{thm-density-estimation}]
Suppose for some $\gamma_n$ the NPMLE satisfies 
\[
\prod_{i=1}^n \frac{f_{\npmle,\Sigma_i}(X_i)}{f_{\trueprior,\Sigma_i}(X_i)}\ge \exp\left((\beta-\alpha)n\gamma_n^2\right) \mathrm{~for~some~} 0<\beta < \alpha < 1.
\]
We bound the probability
\[\P\bigg(\bar{h}(f_{\npmle, \bullet}, f_{\trueprior, \bullet}) \ge  t\gamma_n\bigg)\]
for $t > 1$. 

Take $\{f_{H_j, \bullet}\}_{j=1}^N\subset \F$ to be an $\eta$-net of $\F$ under $\|\cdot\|_{\infty, S^M}$. For each $j$, let $H_{0,j}$ be a distribution satisfying
\[
\|f_{H_{0,j}, \bullet} - f_{H_j, \bullet}\|_{\infty, S^M}\le \eta 
\text{ and }
\bar{h}(f_{H_{0,j}, \bullet}, f_{\trueprior, \bullet})\ge t\gamma_n
\]
and $J = \{j\in [N] : H_{0,j}\text{ exists}\}$. By construction of the $\eta$-net, there is $j^*\in [N]$ such that \[\|f_{H_{j^*}, \bullet} - f_{\npmle, \bullet}\|_{\infty, S^M}\le \eta.\] On the event $\{\bar{h}(f_{\npmle, \bullet}, f_{\trueprior, \bullet}) \ge  t\gamma_n\}$, the NPMLE $\npmle$ acts as a witness that $j^*\in J$, so by the triangle inequality
\begin{align}
\|f_{H_{0,j^*}, \bullet} - f_{\npmle, \bullet}\|_{\infty, S^M}\le 2\eta.
\end{align}
This gives
\[
f_{\npmle, \Sigma_i}(x)
\le \begin{cases}
f_{H_{0,j^*},\Sigma_i}(x) + 2\eta, & \text{if } x\in S^M \\
\frac{1}{\sqrt{(2\pi)^\p |\Sigma_i|}}, & \text{otherwise}.
\end{cases}
\]
Defining $v(x) = \eta 1_{x\in S^M} + \eta\left(\frac{M}{\mathfrak{d}_S( x)}\right)^{\p+1}1_{x\not\in S^M}$, we have 
\begin{align}
\exp((\beta-\alpha)nt^2\gamma_n^2) 
\le \max_{j\in J}\left[\prod_{i=1}^n \frac{f_{H_{0,j},\Sigma_i}(X_i) + 2v(X_i)}{f_{\trueprior,\Sigma_i}(X_i)}\right] \cdot \left[\prod_{i : X_i\not\in S^M}\frac{1}{\sqrt{(2\pi)^\p |\Sigma_i|} \cdot 2v(X_i)}\right]
\end{align}
on the event $\{\bar{h}(f_{\npmle, \bullet}, f_{\trueprior, \bullet}) \ge  t\gamma_n\}$. Hence
\begin{align}
\P&\bigg(\bar{h}(f_{\npmle, \bullet}, f_{\trueprior, \bullet}) \ge  t\gamma_n\bigg) \\
\label{eq-term1}
&\le \P\bigg(\max_{j\in J}\prod_{i=1}^n \frac{f_{H_{0,j},\Sigma_i}(X_i) + 2v(X_i)}{f_{\trueprior,\Sigma_i}(X_i)}\ge \exp(-\alpha nt^2\gamma_n^2) \bigg) \\
\label{eq-term2}
&~~~+ \P\bigg(\prod_{i : X_i\not\in S^M}\frac{1}{\sqrt{(2\pi)^\p |\Sigma_i|} \cdot 2v(X_i)}\ge \exp(\beta nt^2\gamma_n^2) \bigg) 
\end{align}
By a union bound and Markov's inequality, the first term \eqref{eq-term1} is bounded by
\begin{align}
e^{\alpha nt^2\gamma_n^2/2}\sum_{j\in J}\E \prod_{i=1}^n \sqrt{\frac{f_{H_{0,j},\Sigma_i}(X_i) + 2v(X_i)}{f_{\trueprior,\Sigma_i}(X_i)}}
\end{align}
Writing out the expectation, 
\begin{align*}
\prod_{i=1}^n \E \sqrt{\frac{f_{H_{0,j},\Sigma_i}(X_i) + 2v(X_i)}{f_{\trueprior,\Sigma_i}(X_i)}} 
&= \exp\left(\sum_{i=1}^n \log\E \sqrt{\frac{f_{H_{0,j},\Sigma_i}(X_i) + 2v(X_i)}{f_{\trueprior,\Sigma_i}(X_i)}} \right) \\
&\le \exp\left(\sum_{i=1}^n\left\{\int \sqrt{f_{H_{0,j},\Sigma_i} + 2v}\sqrt{f_{\trueprior, \Sigma_i}}  - 1\right\}\right) \\
&\le \exp\left(-\frac{nt^2\gamma_n^2}{2}+n\sqrt{2\int v}\right) 
\end{align*}
Putting together the pieces, the first term \eqref{eq-term1} is bounded by
\begin{equation}
\begin{aligned}
\P&\bigg(\max_{j\in J}\prod_{i=1}^n \frac{f_{H_{0,j},\Sigma_i}(X_i) + 2v(X_i)}{f_{\trueprior,\Sigma_i}(X_i)}\ge e^{-\alpha nt^2\gamma_n^2} \bigg) \\
&\le \exp\left(-\left(1-\alpha\right)\frac{nt^2\gamma_n^2}{2} + \log N+n\sqrt{2\int v}\right)
\end{aligned}
\end{equation}
For the second term \eqref{eq-term2}, observe by Markov's inequality
\begin{align*}
\P&\bigg(\prod_{i : X_i\not\in S^M}\frac{1}{\sqrt{(2\pi)^\p |\Sigma_i|} \cdot 2v(X_i)}\ge \exp(\beta nt^2\gamma_n^2) \bigg) \\
&\le \exp\left(-\frac{\beta nt^2\gamma_n^2}{2\log n}\right) \E\left\{\prod_{i : X_i\not\in S^M}\left|\frac{1}{\sqrt{|2\pi\Sigma_i|} \cdot 2v(X_i)}\right|\right\}^{1/2\log n} \\
&= \exp\left(-\frac{\beta nt^2\gamma_n^2}{2\log n}\right) \E\left\{\prod_{i =1}^n\left(\frac{\mathfrak{d}_S( X_i)}{|2\pi\Sigma_i|^{1/(2\p+2)} \cdot (2\eta)^{1/(\p+1)}M}\right)^{1_{\mathfrak{d}_S( X_i)\ge M)}}\right\}^{(\p+1)/2\log n} 
\end{align*}
To reduce clutter write $a = \frac{1}{\underline{k}^{\p/(2\p+2)}\eta^{1/(\p+1)} M}$ and $\lambda = \frac{\p+1}{2\log n}$. The above expectation is further upper bounded by
\begin{align*}
\E\left\{\prod_{i=1}^n\left(a\mathfrak{d}_S( X_i)\right)^{1_{\mathfrak{d}_S( X_i)\ge M}}\right\}^{\lambda}
&= \prod_{i=1}^n\E\left(a\mathfrak{d}_S( X_i)\right)^{\lambda 1_{\mathfrak{d}_S( X_i)\ge M}} \\
&\le \prod_{i=1}^n\left(1+a^\lambda \E\left[\mathfrak{d}_S( X_i)^{\lambda} 1_{\mathfrak{d}_S( X_i)\ge M}\right]\right) \\
&\le \exp\left(a^\lambda \sum_{i=1}^n\E\left[\mathfrak{d}_S( X_i)^{\lambda} 1_{\mathfrak{d}_S( X_i)\ge M}\right]\right) \\
&\le \exp\left(na^\lambda \left\{C_{\p}M^{\p+\lambda-2} \overline{k}^{1-\p/2}e^{-M^2/(8\overline{k})} + M^\lambda \left(\frac{2\mu_q}{M}\right)^q\right\}\right)
\end{align*}
The last inequality follows from Lemma \ref{lem-lip}. Note we need 
\[
\frac{\p+1}{2(1\land q)} \le \log n,
\]
to ensure $\lambda \le 1\land q$. Taking $M \ge \sqrt{8\overline{k}\log n}$, we have $e^{-M^2/(8\overline{k})} \le \frac{1}{n}$, so 
\begin{align*}
\E\left\{\prod_{i=1}^n\left(a\mathfrak{d}_S( X_i)\right)^{1_{\mathfrak{d}_S( X_i)\ge M}}\right\}^{\lambda}
&\le \exp\left((aM)^\lambda\left[C_{\p}M^{\p-2} \overline{k}^{1-\p/2} + n\left(\frac{2\mu_q}{M}\right)^q\right]\right)
\end{align*}
Noting $(aM)^\lambda = \left(\underline{k}^{\p/2}\eta\right)^{-1/(2\log n)}$, choose $\eta = \frac{n^{-2}}{\underline{k}^{\p/2}}$, so $(aM)^\lambda = e$. We directly apply \citet[][Suppl. Lemma A.7]{saha2020nonparametric} for the integral 
\[\int v \le C_{\p} \eta\mathrm{Vol}(S^M).\] To bound the metric entropy, i.e. $\log N$ where $N$ denotes the size of our $\eta$-net $\{f_{H_j, \bullet}\}_{j=1}^N\subset \F$, we apply Lemma~\ref{lem-entropy} 
\[\log N
=\log N\left(\eta, \F, \|\cdot\|_{\infty, S^M}\right)
\le C_{\p} N\left(u, (S^M)^{u}\right) \left(\log\frac{c_{\p,\overline{k}, \underline{k}}}{\eta}\right)^2,\] 
where the scalar $u$ in the above display corresponds to $a$ used in the lemma. Assuming $4n\ge (2\pi)^{\p/2}$,
\[
u
= \sqrt{-2\overline{k}\log \left(((2\pi \underline{k})^{\p/2} \frac{\eta}{4}\right)}
\ge \sqrt{2\overline{k}\log n}
\]
Similarly $u\le \sqrt{6\overline{k}\log n}$, so
\[
N\left(u, (S^M)^u\right)
\le N\left(\sqrt{2\overline{k}\log n}, (S^M)^{\sqrt{6\overline{k}\log n}}\right)
\le C_{\p, \overline{k}}\mathrm{Vol}(S^{2M})(\log n)^{-\p/2}
\]
Combining our findings,
\begin{align*}
\P&\bigg(\bar{h}(f_{\npmle, \bullet}, f_{\trueprior, \bullet}) \ge  t\gamma_n\bigg) \\
&\le \exp\left(-\left(1-\alpha\right)\frac{nt^2\gamma_n^2}{2} + C_{\p,\overline{k}, \underline{k}} (\log n)^{\p/2+1}\mathrm{Vol}(S^{2M}) +C_{\p}\sqrt{\underline{k}^{-\p/2}\mathrm{Vol}(S^M)}\right) \\
&~~~+ \exp\left(-\frac{\beta}{\log n}\frac{nt^2\gamma_n^2}{2} + C_{\p}M^{\p-2} \overline{k}^{1-\p/2} + en \inf_{q\ge (\p+1)/(2\log n)}\left(\frac{2\mu_q}{M}\right)^q\right)
\end{align*}
for any $t > 1$. Absorbing the dependence on $d, \underline{k}$ and $\overline{k}$ into constants, take $\eps_n^2 = \eps_n^2(M, S, \trueprior)$ such that 
\begin{align*}
\max&\left\{(\log n)^{\p/2+1}\mathrm{Vol}(S^{2M}), \sqrt{\mathrm{Vol}(S^M)}, M^{\p-2} , en \inf_{q\ge (\p+1)/(2\log n)}\left(\frac{2\mu_q}{M}\right)^q\right\}  \\
&\lesssim_{\p,\overline{k}, \underline{k}}n\eps_n^2(M, S, \trueprior)
\end{align*}
If we then take $\gamma_n^2  = \frac{C_{\p,\overline{k}, \underline{k}}\eps_n^2(M, S, \trueprior)}{4\min(1-\alpha,\beta)}$,
\[
\P\bigg(\bar{h}(f_{\npmle, \bullet}, f_{\trueprior, \bullet}) \ge  t\gamma_n\bigg)
\le 2\exp\left(-\frac{(1-\alpha)\land \beta}{4\log n} nt^2\gamma_n^2\right) 
\]
This proves~\eqref{eq-hellinger-accuracy-whp}. To prove~\eqref{eq-hellinger-accuracy-expectation}, integrate the tail from~\eqref{eq-hellinger-accuracy-whp},
\begin{align*}
\E\frac{\bar{h}^2(f_{\npmle, \bullet}, f_{\trueprior, \bullet})}{\gamma_n^2}
&\le 1 + \int_1^\infty \P\bigg(\frac{\bar{h}^2(\npmle, \trueprior)}{\gamma_n^2} \ge s\bigg)\diff s \\
&\le 1 + \int_1^\infty 4tn^{-t^2}\diff t = 1+\frac{2}{n\log n}\le 3
\end{align*}
for $n > 1$, completing the proof.
\end{proof}

We now state and prove the lemmas needed in the proof of Theorem \ref{thm-density-estimation}.

\begin{lemma}\label{lem-lip} Let $\truevec\sim \trueprior$ and $Z\sim \mathcal{N}(0, I_\p)$ independently, and $Y = \truevec + \Sigma^{1/2} Z$, where $\underline{k}I_\p\preceq \Sigma\preceq \overline{k}I_\p$. Then
\[
\E\left[\mathfrak{d}_S(Y)^{\lambda} 1_{\mathfrak{d}_S(Y)\ge M}\right]
\le C_\p M^{\p+\lambda-2} \overline{k}^{1-\p/2}e^{-M^2/(8\overline{k})} + M^\lambda \left(\frac{2\mu_q}{M}\right)^q,
\]
for any $\lambda\in (0, 1\land q]$, where $\mu_q$ is the $q^\mathrm{th}$-moment of $\mathfrak{d}_S( \truevec)$ under $\truevec\sim \trueprior$.
\end{lemma}
\begin{proof} Since distance $\mathfrak{d}_S$ is $1$-Lipschitz,
\begin{align}\label{eq-lem-lip-basic}
\E\left[\mathfrak{d}_S(Y)^{\lambda} 1_{\mathfrak{d}_S(Y)\ge M}\right]
\le \E\left[(2\|\Sigma^{1/2}Z\|_2)^{\lambda} 1_{2\|\Sigma^{1/2}Z\|_2\ge M}\right] + \E\left[(2\mathfrak{d}_S( \truevec))^{\lambda} 1_{2\mathfrak{d}_S( \truevec)\ge M}\right]
\end{align}
For the first term on the RHS of \eqref{eq-lem-lip-basic},
\begin{align*}
\E\left[(2\|\Sigma^{1/2}Z\|_2)^{\lambda} 1_{2\|\Sigma^{1/2}Z\|_2\ge M}\right]
&\le M^\lambda\E\left[\left(\frac{\|\Sigma^{1/2}Z\|_2}{M/2}\right)^{\lambda} 1_{\|\Sigma^{1/2}Z\|_2\ge M/2}\right] \\
&\le 2M^{\lambda-1}\E\left[\|\Sigma^{1/2}Z\|_2 1_{\|\Sigma^{1/2}Z\|_2\ge M/2}\right] \\
&\le 2M^{\lambda-1}\overline{k}^{1/2}\E\left[\|Z\|_2 1_{\|Z\|_2 \ge M/(2\overline{k}^{1/2})}\right] \\
&\le 2C_{\p}M^{\lambda-1}\overline{k}^{1/2} \left(\frac{M}{\overline{k}^{1/2}}\right)^{\p-1}e^{-M^2/(8\overline{k})} \\
&=  C_{\p}M^{\p+\lambda-2} \overline{k}^{1-\p/2}e^{-M^2/(8\overline{k})} 
\end{align*}
The penultimate inequality uses $\|\Sigma^{1/2}Z\|_2 \le \overline{k}^{1/2}\|Z\|_2$, and the last inequality directly uses~\citet[][Suppl. Lemma A.6]{saha2020nonparametric}.

Since $\lambda < q$, applying H\"older to the second term on the RHS of \eqref{eq-lem-lip-basic} yields
\[
\E\left[(2\mathfrak{d}_S( \truevec))^{\lambda} 1_{2\mathfrak{d}_S( \truevec)\ge M}\right]
\le M^\lambda \left(\frac{2\mu_q}{M}\right)^q 
\]
\end{proof}

\begin{lemma} (Moment matching, part i) Let $G, H\in \Ps(\R^\p)$. Suppose $A\subset \R^\p$ is such that \[\mathbb{B}_a(x)\subseteq A\subseteq \mathbb{B}_{ca}(x)\] 
for some $c\ge 1$, and that
\[
\int_A \avec_1^{k_1}~\cdots~\avec_\p^{k_\p} dG(\avec)
=\int_A \avec_1^{k_1}~\cdots~\avec_\p^{k_\p} dH(\avec),
\text{ for } k_1,\dots,k_\p\in [2m+1],
\]
for some $m\ge 1$. Then
\[
\max_{1\le i\le n} \left|f_{G, \Sigma_i}(x) - f_{H, \Sigma_i}(x)\right|
\le \frac{1}{(2\pi \underline{k})^{\p/2}} \left(\frac{ec^2a^2}{2\overline{k}(m+1)}\right)^{m+1}  + \frac{e^{-a^2/(2\overline{k})}}{(2\pi \underline{k})^{\p/2}}.
\]
\end{lemma}
\begin{proof} For each $i\in [n]$, write
\[
f_{G,\Sigma_i}(x) - f_{H, \Sigma_i}(x) 
= \int_A\varphi_{\Sigma_i}(x-\avec) (dG(\avec) - dH(\avec))
+\int_{A^\mathsf{c}}\varphi_{\Sigma_i}(x-\avec) (dG(\avec) - dH(\avec))
\]
On $A^\mathsf{c}$, $\|x-\avec\|_2\ge a$, so \[\varphi_{\Sigma_i}(x-\avec) \le \frac{e^{-a^2/(2\overline{k})}}{(2\pi \underline{k})^{\p/2}}.\]
Write the pdf as $\varphi_{\Sigma_i}(z) = P_i(z) + R_i(z)$ where $P_i$ is a polynomial of degree $2m$ and the remainder $R_i$ satisfies
\[
|R_i(z)|
\le (2\pi \underline{k})^{-\p/2} \left(\frac{e\|z\|_2^2}{2\overline{k}(m+1)}\right)^{m+1}
\]
By hypothesis, $\int_AP_i(x-\avec) (\diff G(\avec) - \diff H(\avec))=0$, so
\begin{align*}
\left|\int_A\varphi_{\Sigma_i}(x-\avec) (\diff G(\avec) - \diff H(\avec))\right|
&\le \left|\int_AR_i(x-\avec) (\diff G(\avec) - \diff H(\avec))\right| \\
&\le \frac{1}{(2\pi \underline{k})^{\p/2}} \left(\frac{ec^2a^2}{2\overline{k}(m+1)}\right)^{m+1} 
\end{align*}
completing the proof.
\end{proof}

\begin{lemma}\label{lem-moment-match-ii} (Moment matching, part ii) For any $G\in\Ps(\R^\p)$, there is a discrete distribution $H$ supported on $S^a$ with at most
\[
l\coloneqq (2\lfloor13.5a^2/\overline{k}\rfloor+2)^\p N(a, S^a) + 1
\]
atoms such that 
\[
\|f_{G, \bullet} - f_{H, \bullet}\|_{\infty, S^a}
\le \left(1+\frac{1}{\sqrt{2\pi}}\right)(2\pi \underline{k})^{-\p/2} e^{-a^2/(2\overline{k})}.
\]
\end{lemma}
\begin{proof} The idea is to choose $H$ to match moments, and then apply the previous lemma. The proof is identical to \citet[][Suppl. Lemma~D.3]{saha2020nonparametric}, except that we take $m \coloneqq \lfloor \frac{27a^2}{2\overline{k}}\rfloor$.
\end{proof}

\begin{lemma}\label{lem-entropy} There exists positive constants $C_{\p}$ and $c_{\p,\overline{k}, \underline{k}}$  depending on $\p, \overline{k}, \underline{k}$ alone such that for every compact set $S\subset\R^\p$, $M > 0$ and $\eta\in (0, e^{-1}\land 4(2\pi\underline{k})^{-\p/2}),$ we have
\begin{align}
\log N(\eta,\F,\|\cdot\|_{\infty,S})
\le C_{\p} N(a, S^a) \left(\log\frac{c_{\p,\overline{k}, \underline{k}}}{\eta}\right)^{\p+1}
\end{align}
\end{lemma}
\begin{proof} The idea here is to take $f_{G, \bullet}\in \F$ (induced by some $G\in \Ps(\R^\p)$), approximate $G$ by a discrete distribution $H$, and then further approximate that discrete distribution with another discrete distribution over a fixed set of atoms and weights. So let $G\in \Ps(\R^\p)$, and apply the previous Lemma to obtain a discrete distribution $H$ supported on $S^a$ with at most $l$ atoms such that
\[
\|f_{G, \bullet} - f_{H, \bullet}\|_{\infty, S^a}
\le \left(1+\frac{1}{\sqrt{2\pi}}\right)(2\pi \underline{k})^{-\p/2} e^{-a^2/(2\overline{k})}.
\]
Let $\mathcal{C}$ denote a minimal $\zeta$-net of $S^a$, and let $H'$ approximate each atom of $H$ with its closest element from $\mathcal{C}$. Writing $H= \sum_j w_j \delta_{a_j}$ and $H'= \sum w_j \delta_{b_j}$, we have
\begin{align*}
\|f_{H, \bullet} - f_{H', \bullet}\|_{\infty, S^a}
&= \max_{i\in[n]}\sup_{x\in S^a} |f_{H,\Sigma_i}(x) - f_{H',\Sigma_i}(x)|  \\
&\le \max_{i\in[n]}\sup_{x\in S^a} \sum_{j}w_j |\varphi_{\Sigma_i}(x-a_j) - \varphi_{\Sigma_i}(x-b_j)| \\
&\le \zeta\max_{i\in[n]}\sup_z \|\nabla\varphi_{\Sigma_i}(z)\|_2 
= \zeta\max_{i\in[n]}\sup_z \varphi_{\Sigma_i}(z)\|\Sigma_i^{-1}z\|_2 \\
&\le \zeta \underline{k}^{-1}(2\pi \underline{k})^{-\p/2} \max_{i\in[n]}\sup_t \exp\left(-t^2/2\overline{k}\right)t \\
&\le \zeta \underline{k}^{-1}(2\pi \underline{k})^{-\p/2} (\overline{k}/e)^{1/2} 
\end{align*}
Let $\mathcal{D}$ denote a minimal $\xi$-net of $\Delta^{l-1}$ in the $\ell_1$ norm, and approximate the weights $w$ by their closest element $v\in \mathcal{D}$. Writing $H'' = \sum_j v_j\delta_{b_j}$,
\begin{align*}
\|f_{H', \bullet} - f_{H'', \bullet}\|_{\infty, S^a}
&= \max_{i\in[n]}\sup_{x\in S^a} |f_{H',\Sigma_i}(x) - f_{H'',\Sigma_i}(x)|  \\
&\le \max_{i\in[n]}\sup_{x\in S^a} \sum_{j}|w_j - v_j| |\varphi_{\Sigma_i}(x-b_j)| \le (2\pi \underline{k})^{-\p/2} \xi.
\end{align*}
Applying triangle inequality to the past three displays, 
\begin{align}
\|f_{G, \bullet} - f_{H'', \bullet}\|_{\infty, S^a}
\le (2\pi \underline{k})^{-\p/2}\left[2 e^{-a^2/(2\overline{k})}
+ \zeta \underline{k}^{-1}(\overline{k}/e)^{1/2}+ \xi\right]
\end{align}
Letting $\xi = (2\pi \underline{k})^{\p/2} \frac{\eta}{4}$, $\zeta = \xi\underline{k}(\overline{k}/e)^{-1/2}$, and $a = \sqrt{2\overline{k}\log \xi^{-1}}$ yields $\|f_{G, \bullet} - f_{H'', \bullet}\|_{\infty, S^a}\le \eta$. In order to take $a$ as such we need $\xi < 1$, or equivalently $\eta < 4(2\pi\underline{k})^{-\p/2}$.

The number of possible $H''$ is 
\[
|\mathcal{C}|\cdot|\mathcal{D}| = N(\xi, \Delta^{l-1}) \binom{N(\zeta, S^a)}{l}
\le \left[\left(1+\frac{2}{\xi}\right)\frac{e N(\zeta, S^a)}{l}\right]^l 
\]
From the previous Lemma, $l \ge N(a, S^a)/\overline{k}^\p$, so 
\[
\frac{N(\zeta, S^a)}{l}
\le \overline{k}^\p\frac{N(\zeta, S^a)}{N(a, S^a)}
\le \overline{k}^\p\left(1+\frac{a}{\zeta}\right)^\p
= \overline{k}^\p\left(1+\kappa\frac{2}{\sqrt{e}\xi^{3/2}}\right)^\p
\le C_{\p} \frac{\overline{k}^{2\p}}{\underline{k}^{\p+3\p^2/4}}\left(\frac{1}{\eta}\right)^{3p/2}
\]
Thus, 
\[
\log N(\eta,\F,\|\cdot\|_{\infty,S})
\le C_{\p}N(a, S^a) \log \left(\frac{1}{\underline{k}^{\p/2}}\left(\frac{\overline{k}^{2\p}}{\underline{k}^{3(\p/2+1)}}\lor 1\right) \frac{e}{\eta}\right)^{\p+1}
\]
\end{proof}

\section{Proof of Theorem \ref{thm-denoising}}\label{sec-proofs-denoising}

Throughout the proof we will group sequences of the form $\avec_1,\dots,\avec_n$ into $n\times \p$ matrices $\theta$, so that, for instance, the regret $\frac{1}{n}\sum_{i=1}^n\E\|\estvec_i - \orvec_i\|_2^2$ in the statement of the theorem may be rewritten as the expected squared Frobenius norm $\frac{1}{n}\E\|\estvec - \orvec\|_F^2$, where $\|\avec\|_F^2 = \sum_{i=1}^n\sum_{j=1}^\p\avec_{ij}^2$. Additionally, we use the same notation introduced at the start of Appendix~\ref{sec-proofs-density-estimation}.

\subsection{Regularizing the Bayes rule}

In evaluating $\estvec$, an apparent difficulty is that the denominator in Tweedie's formula can be arbitrarily small. However, we show that the likelihood is lower bounded at each of the observations. \blue{By~\eqref{eq-approximate-npmle-2}, for any fixed $j\in [n]$,}
\begin{align*}
\prod_{i=1}^nf_{\npmle,\Sigma_i}(X_i) 
&\ge \prod_{i=1}^n f_{n^{-1} \delta_{X_i} + (1-n^{-1})\npmle, \Sigma_i}(X_i) \\
&\ge n^{-1}\varphi_{\Sigma_i}(0)(1-n^{-1})^{n-1}\prod_{i: i\ne j} f_{\npmle, \Sigma_i}(X_i).
\end{align*}
Cancelling terms for $i\in [n]\setminus \{j\}$, we conclude
\begin{align}\label{eq-npmle-obs-bound}
f_{\npmle,\Sigma_j}(X_j) 
&\ge \frac{1}{en\sqrt{|2\pi\Sigma_j|}}.
\end{align}
Given this, it is natural to define the regularized empirical Bayes and oracle Bayes rules 
\begin{align}
\estvec_{\rho, i} 
&= X_i + \Sigma_i\frac{\nabla f_{\npmle, \Sigma_i}(X_i)}{f_{\npmle, \Sigma_i}(X_i) \lor (\rho / \sqrt{|\Sigma_i|})} \\
\orvec_{\rho, i}
&= X_i + \Sigma_i\frac{\nabla f_{\trueprior, \Sigma_i}(X_i)}{f_{\trueprior, \Sigma_i}(X_i) \lor (\rho / \sqrt{|\Sigma_i|})}.
\end{align}
By the lower bound~\eqref{eq-npmle-obs-bound} we know that $\estvec_{\rho} = \estvec$ when $\rho \le \rho_0\coloneqq \frac{1}{en(2\pi)^{\p/2}}$. In particular,
\begin{align}\label{eq-conversion-to-regularized}
\|\estvec - \orvec\|_F
&= \|\estvec_\rho - \orvec\|_F 
\le \|\estvec_\rho - \orvec_\rho\|_F  + \|\orvec_\rho - \orvec\|_F.
\end{align}
The first term $\|\estvec_\rho - \orvec_\rho\|_F$ represents the regret between regularized rules, which prevents the denominator in Tweedie's formula from blowing up. The second term represents the cost of introducing a small amount of regularization in the oracle Bayes rule. 

\subsection{Regularization error of oracle Bayes}

Let us first consider the second term $\|\orvec_\rho - \orvec\|_F$ on the RHS of the bound~\eqref{eq-conversion-to-regularized}. Fixing $i\in [n]$, let $\trueprior_i$ denote the distribution of $\xi_i = \Sigma_i^{-1/2}\truevec_i$ where $\truevec_i\sim \trueprior$. Then we may write \blue{$X_i = \Sigma_i^{1/2} \tilde{X}_i$ where $\tilde{X}_i\sim f_{\trueprior_i, I_{\p}}$}. Note how the scale change affects the terms in Tweedie's formula:
\begin{equation}\label{eq-scaled-changes}
\begin{aligned}
f_{\trueprior, \Sigma_i}(X_i) 
&= \EE_{\vartheta_i\sim \trueprior}\left[\frac{1}{\sqrt{|2\pi\Sigma_i|}}\exp\left(-\frac{1}{2}(X_i-\vartheta_i)'\Sigma_i^{-1}(X_i-\vartheta_i)\right)\right] \\
&= \frac{1}{\sqrt{|\Sigma_i|}} \EE_{\xi_i\sim \trueprior_i}\left[\varphi_{I_{\p}}(\Sigma_i^{-1/2}X_i - \xi_i)\right] = \frac{1}{\sqrt{|\Sigma_i|}}f_{\trueprior_i, I_{\p}}(\blue{\tilde{X}_i}) \\
\nabla f_{\trueprior, \Sigma_i}(X_i) 
&= \EE_{\vartheta_i\sim \trueprior}\left[\Sigma_i^{-1}(\vartheta_i - X_i)\frac{1}{\sqrt{|2\pi\Sigma_i|}}\exp\left(-\frac{1}{2}(X_i-\vartheta_i)'\Sigma_i^{-1}(X_i-\vartheta_i)\right)\right] \\
&= \frac{1}{\sqrt{|\Sigma_i|}}\Sigma_i^{-1/2}\nabla f_{\trueprior_i, I_{\p}}(\blue{\tilde{X}_i}) 
\end{aligned}
\end{equation}

In particular, Tweedie's formula, even in its regularized form, is scale equivariant:
\begin{align}\label{eq-scaled-tweedie}
\orvec_{\rho, i} 
&= \Sigma_i^{1/2}\left(\blue{\tilde{X}_i} + \frac{\nabla f_{\trueprior_i, I_{\p}}(\blue{\tilde{X}_i})}{f_{\trueprior_i, I_{\p}}(\blue{\tilde{X}_i}) \lor \rho}\right) 
\end{align}
In this form, \citet[][Lemma 4.3]{saha2020nonparametric} directly applies. Specifically, defining 
\[
\Delta(\aprior, \rho) \coloneqq \int \left(1-\frac{f_{\aprior, I_{\p}}}{f_{\aprior, I_{\p}}\lor \rho}\right)^2\frac{\|\nabla f_{\aprior, I_{\p}}\|_2^2}{f_{\aprior, I_{\p}}},
\] 
for any $\rho \le \rho_0$ and for all compact sets $S_1, \dots, S_n\subset \R^\p$,
\begin{equation}
\begin{aligned}
\E\|\orvec_\rho - \orvec\|_F^2
&= \sum_{i=1}^n\E\|\orvec_{\rho, i} - \orvec_i\|_2^2 
\le \bar{k}\sum_{i=1}^n\Delta(\trueprior_i, \rho) \\
&\le \bar{k}\sum_{i=1}^n \left\{C_{\p} N\left(\frac{4}{L(\rho)}, S_i\right)L^\p(\rho)\rho + \p\trueprior_i(S_i^\mathsf{c})\right\},
\end{aligned}
\end{equation}
where $L(\rho) \coloneqq \sqrt{-\log((2\pi)^\p\rho^2)}$ and $N$ denotes the usual covering number in the Euclidean norm. Choosing $\rho = (2\pi)^{-\p/2}/n$ and $S_i = \Sigma_i^{-1/2}S^M$, 
\[
\E\|\orvec_\rho - \orvec\|_F^2
\le \bar{k}n \left\{C_{\p} N\left(\frac{4}{\sqrt{\log n}}, \Sigma_{i}^{-1/2}S^M\right)\frac{(\log n)^{\p/2}}{n} + \p\trueprior((S^M)^\mathsf{c})\right\}.
\]

Let $x_1,\dots, x_m$ denote a $t$-net of $S^M$. Let $y\in S_i$ and $x = \Sigma_i^{1/2}y$. There is some $j$ s.t. $\|x_j-x\|_2\le t$. Let $y_j = \Sigma_i^{-1/2}x_j$. Then 
\[t^2\ge(x_j-x)'(x_j-x) = (y_j-y)'\Sigma_i(y_j-y) \ge \underline{k}\|y_j-y\|_2^2\]
so $y_1,\dots,y_m$ is a $t/\underline{k}^{1/2}$-net of $S_i$. This shows $N(t/\underline{k}^{1/2}, S_i) \le N(t, S^M)$. By \citet[][Suppl. Lemma F.6]{saha2020nonparametric} and Markov's inequality,
\begin{align}\label{eq-regulization-err}
\E\|\orvec_\rho - \orvec\|_F^2
\le C_{\p}\bar{k}n \left\{\underline{k}^{-\p/2}\mathrm{Vol}\left(S^{1}\right)M^\p\frac{(\log n)^{\p}}{n} + \inf_{q\ge(\p+1)/2\log n}\left(\frac{2\mu_q}{M}\right)^q\right\}.
\end{align}

\subsection{Regret of regularized rules}

Now we consider the first term $\|\estvec_\rho - \orvec_\rho\|_F$ on the RHS of the bound~\eqref{eq-conversion-to-regularized}. First, we will introduce some additional notation. For $\delta > 0$ let $A_\delta = \left\{\bar{h}^2\left(f_{\npmle, \bullet}, f_{\trueprior, \bullet}\right)  \le \delta\right\}$. Given a compact set $S\subset \R^\p$, define another metric
\[
m^S(\aprior, \aprior')
\coloneqq \max_{i\in [n]}\sup_{x : \mathfrak{d}_S(x)\le M}\left\|\frac{\Sigma_i\nabla f_{\aprior, \Sigma_i}(x)}{f_{\aprior, \Sigma_i}(x) \lor (\rho / \sqrt{|\Sigma_i|})} - \frac{\Sigma_i\nabla f_{\aprior', \Sigma_i}(x)}{f_{\aprior', \Sigma_i}(x) \lor (\rho / \sqrt{|\Sigma_i|})}\right\|_2
\]
Let $\aprior^{(1)},\dots,\aprior^{(N)}$ denote a minimal $\eta^*$-covering of $\left\{G : \bar{h}^2(f_{G, \bullet}, f_{\trueprior, \bullet})  \le \delta\right\}$ in the metric $m^S$. For $j\in [N]$ similarly define an $n\times \p$ matrix $\estvec^{(j)}_\rho$ where the $i^\mathrm{th}$ row is given by $X_i + \Sigma_i\frac{\nabla f_{\aprior^{(j)}, \Sigma_i}(X_i)}{f_{\aprior^{(j)}, \Sigma_i}(X_i) \lor (\rho / \sqrt{|\Sigma_i|})}$. We bound the regret as $\|\estvec_\rho - \orvec_\rho\|_F \le \sum_{t=1}^4\zeta_t$, where
\begin{equation}
\begin{aligned}
\zeta_1 
&\coloneqq \|\estvec_\rho - \orvec_\rho\|_F 1_{A_\delta^\mathsf{c}} \\
\zeta_2
&\coloneqq \left(\|\estvec_\rho - \orvec_\rho\|_F - \max_{j\in [N]}\|\estvec_\rho^{(j)} - \orvec_\rho\|_F\right)_+ 1_{A_\delta} \\
\zeta_3
&\coloneqq \max_{j\in [N]}\left(\|\estvec_\rho^{(j)} - \orvec_\rho\|_F - \E\|\estvec_\rho^{(j)} - \orvec_\rho\|_F\right)_+ \\
\zeta_4
&\coloneqq \max_{j\in [N]}\E\|\estvec_\rho^{(j)} - \orvec_\rho\|_F
\end{aligned}
\end{equation}
We will control the second moment of each $\zeta_t$. Here's our rough overview. $\zeta_1$ uses Theorem~\ref{thm-density-estimation} to show the NPMLE places small probability on $A_\delta^\mathsf{c}$; $\zeta_2$ uses the fact that (on $A_\delta$) the cover $\{\aprior^{(j)}\}$ must have some element that is close to the NPMLE in $m^S$; $\zeta_3$ follows from Gaussian concentration of measure; and $\zeta_4$ bounds each expectation individually and uses closeness in Hellinger. 

\subsubsection{Bounding \texorpdfstring{$\E\zeta_1^2$}{the first term}}

By the scaled Tweedie's formula \eqref{eq-scaled-tweedie} 
\begin{align*}
\|\estvec_{\rho, i} - \orvec_{\rho, i}\|_2^2
&\le \bar{k} \left\|\frac{\nabla f_{\widehat{\aprior}_i, I_{\p}}(\blue{\tilde{X}_i})}{f_{\widehat{\aprior}_i, I_{\p}}(\blue{\tilde{X}_i}) \lor \rho} - \frac{\nabla f_{\trueprior_i, I_{\p}}(\blue{\tilde{X}_i})}{f_{\trueprior_i, I_{\p}}(\blue{\tilde{X}_i}) \lor \rho}\right\|_2^2 
\end{align*}
\citet[][Suppl. Lemma~F.1]{saha2020nonparametric} provides
\begin{align}
\E\zeta_1^2
\le 4 \bar{k} n \log\left(\frac{(2\pi)^\p}{\rho^2}\right) \PP(A_\delta^\mathsf{c}).
\end{align}
By Theorem~\ref{thm-density-estimation}, there is a constant $C_{\p,\underline{k}, \bar{k}} > 0$ such that $\delta = C_{\p,\underline{k}, \bar{k}}\eps_n^2(M, S, \trueprior)$ satisfies $\PP(A_\delta^\mathsf{c})\le 2/n$. Hence
\begin{align}\label{eq-zeta1}
\E\zeta_1^2
\le 48 \bar{k} \log\left(n\right).
\end{align}

\subsubsection{Bounding \texorpdfstring{$\E\zeta_2^2$}{the second term}}

Observe 
\begin{align*}
\zeta_2^2
&\le 1_{A_\delta}\min_{j\in [N]}\|\estvec_\rho - \estvec_\rho^{(j)}\|_F^2 \\
&= 1_{A_\delta} \min_{j\in [N]}\sum_{i=1}^n\left\|\frac{\Sigma_i\nabla f_{\npmle, \Sigma_i}(X_i)}{f_{\npmle, \Sigma_i}(X_i) \lor (\rho / \sqrt{|\Sigma_i|})} - \frac{\Sigma_i\nabla f_{\aprior^{(j)}, \Sigma_i}(X_i)}{f_{\aprior^{(j)}, \Sigma_i}(X_i) \lor (\rho / \sqrt{|\Sigma_i|})}\right\|_2^2
\end{align*}
On $A_\delta$, we may take $j$ such that $m^S(\npmle, \aprior^{(j)})\le \eta^*$. For each $i$, consider two cases, where $X_i\in S^M$ and where $X_i\not\in S^M$. When $X_i\in S^M$ bound the above $\|\cdot\|_2$ by the supremum over all $x\in S^M$. When $X_i\not\in S^M$ bound the regularized rules as before. This yields
\begin{align}
\zeta_2^2
&\le 1_{A_\delta}\left(\#\{i : X_i\in S^M\} (\eta^*)^2 + \#\{i : X_i\not\in S^M\}4\bar{k}\log\left(\frac{(2\pi)^\p}{\rho^2}\right)\right)
\end{align}
so in particular
\begin{align}
\E\zeta_2^2
&\le n(\eta^*)^2 + 4\bar{k}\log\left(\frac{(2\pi)^\p}{\rho^2}\right)\sum_{i=1}^n\PP\left(\mathfrak{d}_S(X_i)\ge M\right).
\end{align}
To bound the probabilities on the RHS, write $X_i = \avec_i+\Sigma^{1/2}Z_i$. By Lemma~\ref{lem-lip}, taking $\lambda\downarrow 0$,
\begin{align}\label{eq-zeta2}
\E\zeta_2^2/n
&\le (\eta^*)^2 + 4\bar{k}\log\left(\frac{(2\pi)^\p}{\rho^2}\right)\left(C_{\p}\frac{M^{\p-2}}{n}\bar{k}^{1-\p/2} + \inf_{q\ge(\p+1)/2\log n}\left(\frac{2\mu_q}{M}\right)^q\right).
\end{align}

\subsubsection{Bounding \texorpdfstring{$\E\zeta_3^2$}{the third term}}\label{sec-bounding-zeta3}

\blue{Fix $j\in [N]$. By the scaled Tweedie's formula \eqref{eq-scaled-tweedie},
\begin{align*}
\|\estvec_{\rho, i}^{(j)} - \orvec_{\rho, i}\|_2^2
&\le \bar{k} \left\|\frac{\nabla f_{\aprior_i^{(j)}, I_{\p}}(\tilde{X}_i)}{f_{\aprior_i^{(j)}, I_{\p}}(\tilde{X}_i) \lor \rho} - \frac{\nabla f_{\trueprior_i, I_{\p}}(\tilde{X}_i)}{f_{\trueprior_i, I_{\p}}(\tilde{X}_i) \lor \rho}\right\|_2^2 \\
&\le 2\bar{k}\left(\left\|\frac{\nabla f_{\aprior_i^{(j)}, I_{\p}}(\tilde{X}_i)}{f_{\aprior_i^{(j)}, I_{\p}}(\tilde{X}_i) \lor \rho}\right\|_2^2 + \left\|\frac{\nabla f_{\trueprior_i, I_{\p}}(\tilde{X}_i)}{f_{\trueprior_i, I_{\p}}(\tilde{X}_i) \lor \rho}\right\|_2^2\right) 
\end{align*}
By \citet[][Suppl. Lemma~F.1]{saha2020nonparametric}, 
\[
\|\estvec_{\rho, i}^{(j)} - \orvec_{\rho, i}\|_2
\le 2\bar{k}^{1/2}L(\rho).
\]
Let $V_i^{(j)} = \frac{\|\estvec_{\rho, i}^{(j)} - \orvec_{\rho, i}\|_2}{2\bar{k}^{1/2}L(\rho)}$, so $\|\estvec_{\rho, i}^{(j)} - \orvec_{\rho, i}\|_F = 2\bar{k}^{1/2}L(\rho)\|V^{(j)}\|_2$. Since $(V_1^{(j)}, \ldots, V_n^{(j)})$ are independent random variables in~$[0,1]$, it follows from Theorem~6.10 of \citet{boucheron2003concentration} that 
\[
\PP\left(\|V^{(j)}\|_2 \ge \EE \|V^{(j)}\|_2 + t\right)\le \exp\left(- t^2/2\right).
\]
Thus
\[
\PP\left(\zeta_3 \ge s\right)\le N\exp\left(- \frac{s^2}{8\bar{k}L^2(\rho)}\right).
\]
Integrating the tail gives
\begin{align}\label{eq-zeta3}
\E\zeta_3^2 \le 8\bar{k}L^2(\rho)\log(eN).
\end{align}
}

\subsubsection{Bounding \texorpdfstring{$\E\zeta_4^2$}{the final term}}

Again by the scaled Tweedie's formula \eqref{eq-scaled-tweedie},
\begin{align*}
\E\|\estvec_\rho^{(j)} - \orvec_\rho\|_F
&\le\sqrt{\E\|\estvec_{\rho}^{(j)} - \orvec_{\rho}\|_F^2}
\le \sqrt{\bar{k}\sum_{i=1}^n \E_{\blue{\tilde{X}_i}\sim f_{\trueprior_i, I_{\p}}}\left\|\frac{\nabla f_{\aprior^{(j)}_i, I_{\p}}(\blue{\tilde{X}_i})}{f_{\aprior^{(j)}_i, I_{\p}}(\blue{\tilde{X}_i}) \lor \rho} - \frac{\nabla f_{\trueprior_i, I_{\p}}(\blue{\tilde{X}_i})}{f_{\trueprior_i, I_{\p}}(\blue{\tilde{X}_i}) \lor \rho}\right\|_2^2}
\end{align*}
\citet[][Lemma E.1]{saha2020nonparametric} bounds the above expectation, yielding
\begin{align}\label{eq-e1-analogue}
\left(\E\|\estvec_\rho^{(j)} - \orvec_\rho\|_F\right)^2
&\le C_{\p}\bar{k}\sum_{i=1}^n \max\left\{\left(\frac{L^2(\rho)}{2}\right)^{3}, \left|\log h\left(f_{\trueprior_i, I_{\p}}, f_{\aprior^{(j)}_i, I_{\p}}\right)\right|\right\} h^2\left(f_{\trueprior_i, I_{\p}}, f_{\aprior^{(j)}_i, I_{\p}}\right).
\end{align}
By a change of variables, 
\[
h^2\left(f_{\trueprior_i, I_{\p}}, f_{\aprior^{(j)}_i, I_{\p}}\right)
= h^2\left(f_{\trueprior, \Sigma_i}, f_{\aprior^{(j)}, \Sigma_i}\right).
\]
Using the shorthand $h^2_i = h^2\left(f_{\trueprior, \Sigma_i}, f_{\aprior^{(j)}, \Sigma_i}\right)$ and using $\rho = (2\pi)^{-\p/2}/n$, 
\begin{equation}
\begin{aligned}
\left(\E\|\estvec_\rho^{(j)} - \orvec_\rho\|_F\right)^2
&\le C_{\p}\bar{k}\sum_{i=1}^n \max\left\{\left(\log n\right)^{3}, -\log h_i\right\} h_i^2 \\
&= C_{\p}\bar{k}\left(\sum_{i : \left(\log n\right)^{3} \ge -\log h_i} \left(\log n\right)^{3} h_i^2 
+ \sum_{i : \left(\log n\right)^{3} < -\log h_i} -(\log h_i) h_i^2 \right) \\
&\le C_{\p}\bar{k}\left(n\left(\log n\right)^{3} \delta
+ \sum_{i : \left(\log n\right)^{3} < \log h_i^{-1}} (\log h_i^{-1}) h_i^2 \right),
\end{aligned}
\end{equation}
where in the last step we used $\frac{1}{n}\sum_{i=1}^n h^2_i = \bar{h}^2(f_{\trueprior,\bullet}, f_{\aprior^{(j)}, \bullet}) \le \delta$. To bound the second term, note for $n\ge 6$, $(\log n)^3 \ge 3\log n$, implying $h_i \le n^{-3}$ for all $i$ such that $\left(\log n\right)^{3} < \log h_i^{-1}$. Since $h_i\log h_i^{-1}\le e^{-1}$ for all $h_i\in [0, 1]$, 
\[
\sum_{i : \left(\log n\right)^{3} < \log h_i^{-1}} (\log h_i^{-1}) h_i^2 \le 
\sum_{i : \left(\log n\right)^{3} < \log h_i^{-1}} \frac{1}{en^3} \le \frac{1}{en^2}.
\]
The first term dominates, so 
\begin{align}\label{eq-zeta4}
\E \zeta_4^2 \le C_{\p}\bar{k} n(\log n)^3\delta. 
\end{align}

\subsubsection{Bounding the metric entropy \texorpdfstring{$\log N$}{}}

We will actually bound the larger covering number $\log N(\eta^*, \Ps(\R^\p), m^S)$ of the space of all probability measures $\Ps(\R^\p)$ in the metric $m^S$. For any measure $G$ we let $G_i$ denote the measure scaled by $\Sigma_i^{-1/2}$ as in the scaled Tweedie formula. For $G, H \in \Ps(\R^\p)$, 
\begin{align*}
m^S(G,H)
&\coloneqq \max_{i\in [n]}\sup_{x : \mathfrak{d}_S(x)\le M}\left\|\frac{\Sigma_i\nabla f_{G, \Sigma_i}(x)}{f_{G, \Sigma_i}(x) \lor (\rho / \sqrt{|\Sigma_i|})} - \frac{\Sigma_i\nabla f_{H, \Sigma_i}(x)}{f_{H, \Sigma_i}(x) \lor (\rho / \sqrt{|\Sigma_i|})}\right\|_2 \\
&\le \max_{i\in [n]}\sup_{x : \mathfrak{d}_S(x)\le M}\left\|\frac{\Sigma_i\nabla f_{G, \Sigma_i}(x)}{f_{G, \Sigma_i}(x) \lor (\rho / \sqrt{|\Sigma_i|})} 
- \frac{\Sigma_i\nabla f_{G, \Sigma_i}(x)}{f_{H, \Sigma_i}(x) \lor (\rho / \sqrt{|\Sigma_i|})} \right\|_2 \\
&~~~~~+\max_{i\in [n]}\sup_{x : \mathfrak{d}_S(x)\le M}\left\|\frac{\Sigma_i\nabla f_{G, \Sigma_i}(x)}{f_{H, \Sigma_i}(x) \lor (\rho / \sqrt{|\Sigma_i|})} 
- \frac{\Sigma_i\nabla f_{H, \Sigma_i}(x)}{f_{H, \Sigma_i}(x) \lor (\rho / \sqrt{|\Sigma_i|})}\right\|_2 \\
&\le \max_{i\in [n]}\sup_{x : \mathfrak{d}_S(x)\le M}\left\|\frac{\Sigma_i\nabla f_{G, \Sigma_i}(x)}{f_{G, \Sigma_i}(x) \lor (\rho / \sqrt{|\Sigma_i|})}\right\|_2
\frac{\left|f_{G, \Sigma_i}(x) \lor (\rho / \sqrt{|\Sigma_i|}) - f_{H, \Sigma_i}(x) \lor (\rho / \sqrt{|\Sigma_i|})\right|}{f_{H, \Sigma_i}(x) \lor (\rho / \sqrt{|\Sigma_i|})} \\
&~~~~~+\max_{i\in [n]}\sup_{x : \mathfrak{d}_S(x)\le M}\left\|\frac{\Sigma_i(\nabla f_{G, \Sigma_i}(x) - \nabla f_{H, \Sigma_i}(x))}{f_{H, \Sigma_i}(x) \lor (\rho / \sqrt{|\Sigma_i|})}\right\|_2 
\end{align*}
For the first term, by \citet[][Suppl. Lemma~F.1]{saha2020nonparametric},
\[
\left\|\frac{\Sigma_i\nabla f_{G, \Sigma_i}(x)}{f_{G, \Sigma_i}(x) \lor (\rho / \sqrt{|\Sigma_i|})}\right\|_2 
= \left\|\Sigma_i^{1/2}\frac{\nabla f_{\trueprior_i, I_{\p}}(\Sigma_i^{-1/2}x)}{f_{\trueprior_i, I_{\p}}(\Sigma_i^{-1/2}x) \lor \rho} \right\|_2
\le \bar{k}^{1/2}L(\rho).
\]
Replacing $f\lor (\rho/\sqrt{|\Sigma_i|})$ with $\rho/\sqrt{|\Sigma_i|}$ in the denominator can only make the denominator smaller, so
\begin{align}\label{eq-bound-tweedie-err}
m^S(G,H)
&\le \bar{k}^{1/2}\rho^{-1}L(\rho)\max_{i\in [n]}\sup_{x : \mathfrak{d}_S(x)\le M}\sqrt{|\Sigma_i|}\left|f_{G, \Sigma_i}(x)  - f_{H, \Sigma_i}(x) \right| \\
&~~~~~+\rho^{-1}\max_{i\in [n]}\sup_{x : \mathfrak{d}_S(x)\le M}\sqrt{|\Sigma_i|}\left\|\Sigma_i(\nabla f_{G, \Sigma_i}(x) - \nabla f_{H, \Sigma_i}(x))\right\|_2  \\
&\le \bar{k}^{\p/2+1/2}\rho^{-1}L(\rho)\|f_{G, \bullet}-f_{H, \bullet}\|_{\infty, S^M}
+ \bar{k}^{\p/2+1}\rho^{-1}\|f_{G,\bullet}-f_{H, \bullet}\|_{\nabla, S^M}
\end{align}
In particular, letting 
\[
\eta^* = \left(\bar{k}^{\p/2+1/2}L(\rho) + \bar{k}^{\p/2+1}\right)\frac{\eta}{\rho}
\]
we have 
\[
\log N(\eta^*, \Ps(\R^\p), m^S)
\le \log N(\eta/2, \F, \|\cdot\|_{\infty, S^M}) + \log N(\eta/2, \F, \|\cdot\|_{\nabla, S^M}).
\]
We already have a bound on $\log N(\eta/2,\F,\|\cdot\|_{\infty,S^M})$ in Lemma~\ref{lem-entropy}, and we bound the other term similarly in Lemma~\ref{lem-entropy-2} below. Combining these bounds, 
\[
\log N(\eta^*, \Ps(\R^\p), m^S)
\le C_{\p} N(a, S^{M+a})\left(\log\frac{c_{\p, \underline{k}, \bar{k}}}{\eta}\right)^{\p+1},
\]
where $a = \sqrt{-2\bar{k}\log \left(\sqrt{\underline{k}\land 1}\frac{(2\pi\underline{k})^{\p/2}}{5}\eta\right)}$. Take $\eta = \rho/n = (2\pi)^{-\p/2}/n^2$. 
\[
a 
= \sqrt{4\bar{k}\log n + 2\bar{k}\log \left(\frac{5}{\sqrt{\underline{k}\land 1}\underline{k}^{\p/2}}\right)} \in \left[\sqrt{2\bar{k}\log n}, \sqrt{6\bar{k}\log n}\right],
\]
provided $1/n\le \frac{5}{\sqrt{\underline{k}\land 1}\underline{k}^{\p/2}}\le n$. Hence by \citet[][Suppl. Lemma~F.6]{saha2020nonparametric} (see the argument on page 6) gives 
\[
\log N \le c_{\p, \underline{k}, \bar{k}}(\log n)^{1+\p/2}\mathrm{Vol}(S^{2M}).
\]

\begin{lemma}\label{lem-entropy-2} For all compact $S\subset \R^\p$, $M > 0$ and $\eta > 0$ sufficiently small, 
\[
\log N(\eta, \F, \|\cdot\|_{\nabla, S^M}) \le C_{\p} N(a, S^a)\left(\log\frac{c_{\p, \underline{k}, \bar{k}}}{\eta}\right)^{\p+1}
\]
where $a = \sqrt{2\bar{k}\log \frac{c_{\p, \overline{k}, \underline{k}}'}{\eta}}$.
\end{lemma}

\begin{proof}
Fix $G\in \Ps(\R^\p)$. By Lemma~\ref{lem-moment-match-ii},
there is a discrete measure $H$ supported on $S^a$ with at most
\[
l\coloneqq \left(2\lfloor 13.5a^2/\bar{k}\rfloor+2\right)^\p N(a, S^a) + 1
\]
atoms such that 
\[
\|f_{G, \bullet}-f_{H, \bullet}\|_{\nabla, S^a}
\le a\left(1+\frac{3}{\sqrt{2\pi}}\right)(2\pi\underline{k})^{-\p/2}e^{-a^2/(2\bar{k})}
\]
Now let $\cc$ denote a minimal $\alpha$-net of $S^a$. Write $H=\sum_j w_j\delta_{a_j}$, and define $H' =\sum_j w_j\delta_{b_j}$ where $b_j\in \cc$ is the closest element to $a_j$. Then 
\begin{align*}
\left\|\nabla f_{H, \Sigma_i}(x)-\nabla f_{H', \Sigma_i}(x)\right\|_2
&\le \sum_j w_j \left\|\nabla \varphi_{\Sigma_i}(x-a_j)-\nabla \varphi_{\Sigma_i}(x-b_j)\right\|_2 \\
&\le \underline{k}^{-1/2}|\Sigma_i|^{-1/2} \sum_j w_j \left\|\nabla\varphi\left(\Sigma_i^{-1/2}(x-a_j)\right)- \nabla\varphi\left(\Sigma_i^{-1/2}(x-b_j)\right)\right\|_2 \\
&\le  \frac{\underline{k}^{-\p/2-1/2}\alpha}{(2\pi)^{\p/2}}\left[1+\frac{2}{e} + \frac{\alpha}{\sqrt{\underline{k}e}}\right]
\end{align*}

Now let $\cd$ denote a minimal $\beta$-net of $\Delta_{l-1}$ under $\|\cdot\|_1$. Let $H'' = \sum_jw_j'\delta_{b_j}$ where $\|w'-w\|_1\le \beta$. Then 
\[
\left\|\nabla f_{H', \Sigma_i}(x)-\nabla f_{H'', \Sigma_i}(x)\right\|_2
\le \beta \sup_u\|\nabla\varphi_{\Sigma_i}(u)\|_2
\le \frac{\underline{k}^{-\p/2-1/2}\beta}{(2\pi)^{\p/2}\sqrt{e}}.
\]
By triangle inequality,
\[
\left\|f_{G,\bullet}- f_{H'', \bullet}\right\|_{\nabla, S^M}
\le (2\pi\underline{k})^{-\p/2}\left[\left(1+\frac{3}{\sqrt{2\pi}}\right)ae^{-a^2/(2\bar{k})} + \frac{\alpha}{\sqrt{\underline{k}}}\left[1+\frac{2}{e} + \frac{\alpha}{\sqrt{\underline{k}e}}\right] + \frac{\beta}{\sqrt{\underline{k}e}}\right]
\]
Taking $a = \sqrt{2\bar{k}\log \alpha^{-1}}\ge 1$ and $\alpha = \beta = \sqrt{k\land 1}\frac{(2\pi\underline{k})^{\p/2}}{5}\eta$, 
\[
\left\| f_{G, \bullet}- f_{H'', \bullet}\right\|_{\nabla, S^M}
\le \frac{5a\alpha}{\sqrt{k\land 1}}(2\pi\underline{k})^{-\p/2} = a\eta
\]
The proof is completed following same steps as the proof of Lemma 5.
\end{proof}

\subsection{Putting together the pieces}

Combining~\eqref{eq-regulization-err},~\eqref{eq-zeta1},~\eqref{eq-zeta2},~\eqref{eq-zeta3}, and~\eqref{eq-zeta4} and pulling out any constants depending on $\p, \underline{k},$ or $\bar{k}$, 
\begin{equation}
\begin{aligned}
\E\|\estvec - \orvec\|_F^2/n
&\le (5/n)\left[\E\|\orvec_\rho - \orvec\|_F^2 + \sum_{t=1}^4 \E\zeta_t^2\right] \\
&\le c_{\p, \underline{k}, \bar{k}}\bigg(\eps_n^2(M, S, \trueprior)(\sqrt{\log n})^{\p-2} \\
&~~~~~~~~~~~~~+\frac{\log n}{n} \\
&~~~~~~~~~~~~~+(\eta^*)^2 + \eps_n^2(M, S, \trueprior) \\
&~~~~~~~~~~~~~+\log(eN)\frac{\log n}{n} \\ 
&~~~~~~~~~~~~~+\eps_n^2(M, S, \trueprior)(\sqrt{\log n})^{6} \bigg) \\
&\le c_{\p, \underline{k}, \bar{k}}\eps_n^2(M, S, \trueprior)(\sqrt{\log n})^{(\p-2)\lor 6}
\end{aligned}
\end{equation}
This completes the proof of Theorem \ref{thm-denoising}.

\section{Proofs of Theorems~\ref{thm-deconvolution} and~\ref{thm-deconvolution-point-mass}}\label{sec-proofs-deconvolution}

\begin{proof}[{Proof of Theorem~\ref{thm-deconvolution}}]

We will relate the Wasserstein distance to the average Hellinger distance, so we rely on the tools of \citet[][proof of theorem 2]{nguyen2013convergence}. Fix a symmetric density $K$ whose Fourier transform $\widetilde{K}$ is bounded with support on $[-1,1]^\p$. For any $\delta > 0$ define the scaled kernel $K_\delta(x) = \frac{1}{\delta^\p}K(x/\delta)$. By the triangle inequality,
\[
W_2(\trueprior, \npmle)
\le W_2(\trueprior, \trueprior * K_\delta) + W_2(\trueprior * K_\delta, \npmle * K_\delta) + W_2(\npmle * K_\delta, \npmle).
\]
For the first and third terms, bound the minimum over all couplings by the strong coupling:
\[
W_2^2(G, G*K_\delta)
= \min_{\avec\sim G, \avec'\sim G, \eps\sim K}\E\|\avec-(\avec'+\delta\eps)\|_2^2
\le \delta^2\E_{\eps\sim K}\|\eps\|_2^2,
\]
where the inequality follows from choosing the coupling where $\avec=\avec'$ almost surely. Letting $m_2(K) = \E_{\eps\sim K}\|\eps\|_2^2$ denote the second moment of the (unscaled) kernel, we have 
\begin{align}\label{eq-triangle-Wasserstein}
W_2(\trueprior, \npmle)
\le 2\sqrt{m_2(K)}\delta + W_2(\trueprior * K_\delta, \npmle * K_\delta).
\end{align}
For the second term, \citet[][Theorem 6.15]{villani2008optimal} yields
\[
W_2^2(\trueprior * K_\delta, \npmle * K_\delta)
\le 2\int\|x\|_2^2\diff\left|\trueprior * K_\delta - \npmle * K_\delta\right|(x)
\]
By \citet[][Lemma 6]{nguyen2013convergence}, for any $s > 2$ such that $m_s(K) = \E\|\eps\|_2^s < \infty$,  
\begin{align*}
W_2^2(\trueprior * K_\delta, \npmle * K_\delta)
&\le 4\left\|\trueprior * K_\delta - \npmle * K_\delta\right\|_{L_1}^{(s-2)/s}R^{2/s} \\ 
&\le 4\left[2\mathrm{Vol}(B_1)^{s/(\p+2s)}R^{\p/(\p+2s)}\left\|\trueprior * K_\delta - \npmle * K_\delta\right\|_{L_2}^{2s/(\p+2s)}\right]^{(s-2)/s} R^{2/s}\\
&= 4\cdot 2^{(s-2)/s}\mathrm{Vol}(B_1)^{(s-2)/(2s+\p)}\cdot R^{\p(s-2)/(s(\p+2s)) + 2/s}\left\|\trueprior * K_\delta - \npmle * K_\delta\right\|_{L_2}^{2(s-2)/(\p+2s)} \\
&\le 8\sqrt{\mathrm{Vol}(B_1)}\cdot R^{\p(s-2)/(s(\p+2s)) + 2/s}\left\|\trueprior * K_\delta - \npmle * K_\delta\right\|_{L_2}^{2(s-2)/(\p+2s)} 
\end{align*}
where $R \coloneqq \E_{\truevec\sim \trueprior, \eps\sim K}\|\truevec+\delta\eps\|_2^s + \E_{\avec\sim \npmle, \eps\sim K}\|\avec+\delta\eps\|_2^s$. 

For moments in the term $R$, use $\E\|\avec+\delta\eps\|_2^s
\le 2^s\left(\E\|\avec\|_2^s + \delta^sm_s(K)\right),$ so 
\[R\le 2^s(m_s(\trueprior) + m_s(\npmle) + 2\delta^sm_s(K)).\] The quantity $m_s(K)$ is regarded as a constant depending only on $s > 2$ and $\p$. By assumption, the support of $\npmle$ is contained in the minimum bounding box of the observations, which is further contained in $[-U, U]^\p$ where $U = \max_{i, j}|X_{ij}| \le L+\max_{i, j}|X_{ij} - \truevec_{ij}|$. Since $X_{ij} - \truevec_{ij} \simind \mathcal{N}(0, (\Sigma_{i})_{jj})$, we have by a standard concentration argument that
\[
U\le L+ 4\sqrt{\overline{k}\log n}
\]
with probability at least $1-\frac{2d}{n^8}$. Hence, with the same probability
\[
m_s(\npmle)
= \E_{\npmle}\|\avec\|_2^s
\le \p^{s/2}\E_{\npmle}\|\avec\|_\infty^s
\le \p^{s/2}\left(L+ 4\sqrt{\overline{k}\log n}\right)^s.
\]
This same bound holds for $m_s(\trueprior)$.

For the $\|\cdot\|_{L_2}$ norm $\left\|\trueprior * K_\delta - \npmle * K_\delta\right\|_{L_2}$, let $g_\delta^{(i)}$ denote the inverse Fourier transform of $\widetilde{K_\delta}/\widetilde{\varphi_{\Sigma_i}}$, so that $G*K_\delta = f_{G, \Sigma_i}*g_\delta^{(i)}$. Hence, by Proposition~8.49 of \citet{folland1999real}, we have for each~$i = 1,\dots,n$,
\[
\left\|\trueprior * K_\delta - \npmle * K_\delta\right\|_{L_2}
\le 2d_{TV}(f_{\trueprior, \Sigma_i}, f_{\npmle, \Sigma_i})\|g_\delta^{(i)}\|_{L_2}.
\]
Using Plancherel's theorem and the fact that $\widetilde{K}$ is bounded on its support of $[-1,1]^\p$,
\begin{align*}
\|g_\delta^{(i)}\|_{L_2}^2 
&= \int_{\R^\p} \frac{\widetilde{K}(\delta\omega)^2}{\widetilde{\varphi_{\Sigma_i}}(\omega)^2}\diff \omega
\le C_\p \int_{[-1/\delta, 1/\delta]^\p} \widetilde{\varphi_{\Sigma_i}}(\omega)^{-2}\diff \omega \\
&= C_\p \int_{[-1/\delta, 1/\delta]^\p} \exp(\omega'\Sigma_i\omega)\diff \omega
= C_\p \prod_{j=1}^\p\int_{-1/\delta}^{1/\delta} \exp((\Sigma_i)_{jj}\omega_j^2)\diff \omega_j \\
&\le C_\p \prod_{j=1}^\p e^{2(\Sigma_i)_{jj}\delta^{-2}}\int_{-1/\delta}^{1/\delta} \exp(-(\Sigma_i)_{jj}\omega_j^2)\diff \omega_j 
\le C_\p  \left(\frac{\pi}{\underline{k}}\right)^{\p/2}e^{2\p\overline{k}\delta^{-2}}.
\end{align*}
Averaging over $i=1,\dots,n$,
\begin{align*}
\left\|\trueprior * K_\delta - \npmle * K_\delta\right\|_{L_2}
&\le C_\p  \underline{k}^{-\p/2}e^{2\p\overline{k}\delta^{-2}} \cdot \frac{1}{n}\sum_{i=1}^nd_{TV}(f_{\trueprior, \Sigma_i}, f_{\npmle, \Sigma_i}) \\
&\le C_\p  \underline{k}^{-\p/2}e^{2\p\overline{k}\delta^{-2}} \bar{h}(\trueprior, \npmle).
\end{align*}
Combining our calculations following~\eqref{eq-triangle-Wasserstein}, we have
\begin{equation}
\begin{aligned}
W_2(\trueprior, \npmle)
&\le C_{\p, s}\inf_{\delta\in (0, 1)}\bigg\{ \delta + \left(2^s\left(\p^{s/2}\left(L+ 4\sqrt{\overline{k}\log n}\right)^s + 2\delta^sm_s(K)\right)\right)^{3\p/(2\p+4s)} \\ 
&~~~~~~~~~~~~~~~~~~~~~~\times\left(\underline{k}^{-\p/2}e^{2\p\overline{k}\delta^{-2}} \bar{h}(\trueprior, \npmle) \right)^{(s-2)/(\p+2s)}  \bigg\}.
\end{aligned}
\end{equation}
Assume $n$ is large enough that $4\sqrt{\overline{k}\log n} \ge L$ and $2^s\p^{s/2}\left(4\sqrt{\overline{k}\log n}\right)^s \ge 2m_s(K)$, so
\begin{equation}
\begin{aligned}
W_2(\trueprior, \npmle)
&\le C_{\p, s}\inf_{\delta\in (0, 1)}\bigg\{ \delta + \left(\overline{k}\log n\right)^{3s\p/(4(\p+2s))} \left(\underline{k}^{-\p/2}e^{2\p\overline{k}\delta^{-2}} \bar{h}(\trueprior, \npmle) \right)^{(s-2)/(\p+2s)}  \bigg\}.
\end{aligned}
\end{equation}
Choosing $\delta^{-2} = -\frac{1}{4\overline{k}\p}\log \bar{h}(\trueprior, \npmle)$ (provided $\delta < 1$) and $s=d+2$ yields
\begin{equation}
\begin{aligned}
W_2^2(\trueprior, \npmle)
&\le C_{\p}\bigg\{ \frac{\overline{k}\p}{-\log \bar{h}(\trueprior, \npmle)} + \left(\overline{k}\log n\right)^{d/2} \left(\underline{k}^{-\p}\bar{h}(\trueprior, \npmle) \right)^{1/12}  \bigg\}.
\end{aligned}
\end{equation}
$\eps_n^2(M, S, \trueprior)$ defined in~\eqref{eq-eps}, with $S = [-L, L]^\p$ and $M = \sqrt{10\overline{k}\log n}$ gives 
\[\eps_n^2= \frac{\left(4\sqrt{10}\left(L^2\lor\overline{k}\right)\right)^\p}{n}(\log n)^{\p+1}.\]
By Theorem~\ref{thm-density-estimation},
\begin{align*}
W_2^2(\trueprior, \npmle)
&\le C_{\p}\bigg\{ \frac{\overline{k}\p}{\log n -\log C_{\p,\overline{k}, \underline{k},L}t^2(\log n)^{d+1}} + \left(\overline{k}\log n\right)^{d/2} \left(C_{\p,\overline{k}, \underline{k},L}t^2\frac{(\log n)^{d+1}}{n} \right)^{1/24}  \bigg\},
\end{align*}
with probability at least $1-2n^{-t^2}$. Take $t^2=8$. For $n$ sufficiently large the first term dominates, $\delta < 1$, and $\log n -\log C_{\p,\overline{k}, \underline{k},L}8(\log n)^{d+1} \ge (\log n)/2$.
\end{proof}

\begin{proof}[Proof of Theorem~\ref{thm-deconvolution-point-mass}]
Take $\mu = 0$ by location equivariance (see Lemma~\ref{lem-invariance}). Write $\npmle = \sum_{j=1}^{\khat}\hat{w}_j\delta_{\hat{a}_j}$. Since~$\trueprior = \delta_0$ is a point mass,
\[W_2^2(\npmle, \trueprior) = \E_{\vartheta\sim \npmle}\|\vartheta\|_2^2 = \sum_{j=1}^{\khat}\hat{w}_j\|\hat{a}_j\|_2^2.\] 
We relate this to the marginal density~$f_{\npmle, \bullet}$ via
\begin{align*}
\int f_{\npmle, \Sigma_i}(x)\|x\|_2^2\diff x 
&= \sum_{j=1}^{\khat}\hat{w}_j \int \varphi_{\Sigma_i}(x-\hat{a}_j)\|x\|_2^2\diff x \\
&= \sum_{j=1}^{\khat}\hat{w}_j \int \varphi_{\Sigma_i}(x)(\|x\|_2^2+\|\hat{a}_j\|_2^2 +2\langle x,\hat{a}_j\rangle)\diff x \\
&= \int f_{\trueprior, \Sigma_i}(x)\|x\|_2^2\diff x  + W_2^2(\npmle, \trueprior).
\end{align*}
Hence for any~$i\in \{1,\dots, n\}$,
\begin{align*}
W_2^2(\npmle, \trueprior) 
&= \int (f_{\npmle, \Sigma_i}(x) - f_{\trueprior, \Sigma_i}(x))\|x\|_2^2\diff x \\
&\le 2h(f_{\npmle, \Sigma_i}, f_{\trueprior, \Sigma_i}) 
\left(\int\left(f_{\npmle, \Sigma_i}(x) + f_{\trueprior, \Sigma_i}(x)\right)\|x\|_2^4\diff x\right)^{1/2}.
\end{align*}
Averaging over~$i\in\{1, \dots, n\}$,
\[
W_2^2(\npmle, \trueprior) 
\le 2\bar{h}\left(f_{\npmle, \bullet}, f_{\trueprior, \bullet}\right) 
\max_{i=1:n}\left(\int\left(f_{\npmle, \Sigma_i}(x) + f_{\trueprior, \Sigma_i}(x)\right)\|x\|_2^4\diff x\right)^{1/2}.
\]
Applying Theorem~\ref{thm-density-estimation} with $S= \{0\}$ and $M = \sqrt{10\overline{k}\log n}$, 
\[
\bar{h}^2\left(f_{\npmle, \bullet}, f_{\trueprior, \bullet}\right) \lesssim_{\p,\overline{k}, \underline{k}}t^2\frac{(\log n)^{\p+1}}{n}
\]
with probability at least $1-2n^{-t^2}$ for all $t\ge1$. 

For the remaining terms, 
\[
\int f_{\trueprior, \Sigma_i}(x)\|x\|_2^4\diff x 
= \E\|X_i\|_2^4
= \E_{Z\sim\cn(0, I_\p)}\|\Sigma_i^{1/2}Z\|_2^4
\le \overline{k}^{2}\E_{A\sim \chi_\p^2}A^2
= \overline{k}^{2}\p(\p+2)
\]
and
\begin{align*}
\int f_{\npmle, \Sigma_i}(x)\|x\|_2^4\diff x
&= \sum_j\hat{w}_j\int \varphi_{\Sigma_i}(x)\|x+\hat{a}_j\|_2^4\diff x \\
&\le 8\sum_j\hat{w}_j\int \varphi_{\Sigma_i}(x)\Big(\|x\|_2^4+\|\hat{a}_j\|_2^4\Big)\diff x \\
&\le 8\overline{k}^{2}\p(\p+2) + 8\sum_j\hat{w}_j\|\hat{a}_j\|_2^4.
\end{align*}
By our assumption on the support that each~$\hat{a}_j\in \mathbb{B}_{\kappa r}(\bar{X})$, each~$\hat{a}_j$ equivalently satisfies 
\[\|\hat{a}_j - \bar{X}\|_2^4\le (\overline{k}/\underline{k})^4\max_i\|X_i-\bar{X}\|_2^4.\]
Noting that $X_i-\bar{X} \sim \cn\left(0, (1-2n^{-1})\Sigma_i + n^{-1}\bar{\Sigma}\right)$ with $\bar{\Sigma} = n^{-1}\sum_{j=1}^n\Sigma_j$, we bound
\begin{align*}
\sum_j\hat{w}_j\|\hat{a}_j\|_2^4
&\le 8\left(\|\bar{X}\|_2^4 + \max_j\|\hat{a}_j - \bar{X}\|_2^4\right) \\
&\le 8\left(\|\bar{X}\|_2^4 + \kappa^4\max_{i\in[n]}\|X_i-\bar{X}\|_2^4\right) \\ 
&\le_{\mathrm{st}} 8\overline{k}^2\left(n^{-2}A_0^2 + \kappa^4\max_{i\in[n]}A_i^2\right) \\
&\le 16\frac{\overline{k}^6}{\underline{k}^4}\max_{i=0:n}A_i^2,
\end{align*}
where $A_0,A_1, \dots, A_n\sim \chi_\p^2$ are possibly dependent, and $\le_{\mathrm{st}}$ denotes stochastic inequality. 

For $t\ge1$, we use the following tail bound \citep[see][Lemma~1]{laurent2000adaptive}
\[
\P\left(\max_{i=0:n}A_i^2 \ge 60t^2(\log n)^2\right)
\le n^{-t^2},
\]
where we have used the assumption in Theorem~\ref{thm-density-estimation} that $n \ge (2\pi)^{\p/2}$ to eliminate the dependence on~$\p$. We have thus shown that
\[
\max_{i=1:n}\left(\int\left(f_{\npmle, \Sigma_i}(x) + f_{\trueprior, \Sigma_i}(x)\right)\|x\|_2^4\diff x\right)^{1/2} 
\le \left(9\overline{k}^{2}\p(\p+2) + 400\frac{\overline{k}^6}{\underline{k}^4}t^2(\log n)^2\right)^{1/2} 
\lesssim \frac{\overline{k}^3}{\underline{k}^2}t(\log n)
\]
with probability at least $1-n^{-t^2}$. Combining with a union bound over our earlier estimate, 
\[
W_2(\npmle, \trueprior) \lesssim_{\p,\overline{k}, \underline{k}}t^{3/2}\frac{(\log n)^{(\p+3)/4}}{n^{1/4}}
\]
with probability at least $1-3n^{-t^2}$ for all $t\ge 1$. 
\end{proof}


\begin{thebibliography}{}

\bibitem[Abolfathi et~al., 2018]{abolfathi2018fourteenth}
Abolfathi, B., Aguado, D., Aguilar, G., Prieto, C.~A., Almeida, A., Ananna, T.~T., Anders, F., Anderson, S.~F., Andrews, B.~H., Anguiano, B., et~al. (2018).
\newblock The fourteenth data release of the {S}loan {D}igital {S}ky {S}urvey: first spectroscopic data from the extended {B}aryon {O}scillation {S}pectroscopic {S}urvey and from the second phase of the {A}pache {P}oint {O}bservatory {G}alactic {E}volution {E}xperiment.
\newblock {\em The Astrophysical Journal Supplement Series}, 235(2):42.

\bibitem[Aibar et~al., 2023]{leukemiasEset}
Aibar, S., Fontanillo, C., and De~Las~Rivas, J. (2023).
\newblock {\em leukemiasEset: Leukemia's microarray gene expression data (expressionSet)}.
\newblock R package version 1.36.0.

\bibitem[Akritas and Bershady, 1996]{AB96}
Akritas, M.~G. and Bershady, M.~A. (1996).
\newblock Linear regression for astronomical data with measurement errors and intrinsic scatter.
\newblock {\em The Astrophysical Journal}, 470(2):706.

\bibitem[Am\'{e}ndola et~al., 2020]{amendola2020maximum}
Am\'{e}ndola, C., Engstr\"{o}m, A., and Haase, C. (2020).
\newblock Maximum number of modes of {G}aussian mixtures.
\newblock {\em Inf. Inference}, 9(3):587--600.

\bibitem[Anderson et~al., 2018]{anderson2018improving}
Anderson, L., Hogg, D.~W., Leistedt, B., Price-Whelan, A.~M., and Bovy, J. (2018).
\newblock Improving {G}aia parallax precision with a data-driven model of stars.
\newblock {\em The Astronomical Journal}, 156(4):145.

\bibitem[Banerjee et~al., 2023]{banerjee2023nonparametric}
Banerjee, T., Fu, L.~J., James, G.~M., Mukherjee, G., and Sun, W. (2023).
\newblock Nonparametric empirical {B}ayes estimation on heterogeneous data.
\newblock {\em arXiv preprint arXiv:2002.12586}.

\bibitem[B\"{o}hning, 1985]{bohning1985numerical}
B\"{o}hning, D. (1985).
\newblock Numerical estimation of a probability measure.
\newblock {\em J. Statist. Plann. Inference}, 11(1):57--69.

\bibitem[B{\"o}hning, 1999]{bohning1999computer}
B{\"o}hning, D. (1999).
\newblock {\em Computer-assisted analysis of mixtures and applications: meta-analysis, disease mapping and others}, volume~81.
\newblock CRC press.

\bibitem[B\"{o}hning, 2003]{bohning2003algorithm}
B\"{o}hning, D. (2003).
\newblock The {EM} algorithm with gradient function update for discrete mixtures with known (fixed) number of components.
\newblock {\em Stat. Comput.}, 13(3):257--265.

\bibitem[Boucheron et~al., 2013]{boucheron2003concentration}
Boucheron, S., Lugosi, G., and Massart, P. (2013).
\newblock {\em Concentration inequalities}.
\newblock Oxford University Press, Oxford.
\newblock A nonasymptotic theory of independence, With a foreword by Michel Ledoux.

\bibitem[Bovy et~al., 2011]{bovy2011extreme}
Bovy, J., Hogg, D.~W., and Roweis, S.~T. (2011).
\newblock Extreme deconvolution: inferring complete distribution functions from noisy, heterogeneous and incomplete observations.
\newblock {\em Ann. Appl. Stat.}, 5(2B):1657--1677.

\bibitem[Carroll and Hall, 1988]{carroll1988optimal}
Carroll, R.~J. and Hall, P. (1988).
\newblock Optimal rates of convergence for deconvolving a density.
\newblock {\em J. Amer. Statist. Assoc.}, 83(404):1184--1186.

\bibitem[Chen, 2023]{chen2023empirical}
Chen, J. (2023).
\newblock Empirical {B}ayes when estimation precision predicts parameters.
\newblock {\em arXiv preprint arXiv:2212.14444}.

\bibitem[Deb et~al., 2022]{deb2021two}
Deb, N., Saha, S., Guntuboyina, A., and Sen, B. (2022).
\newblock Two-component mixture model in the presence of covariates.
\newblock {\em J. Amer. Statist. Assoc.}, 117(540):1820--1834.

\bibitem[Dedecker and Michel, 2013]{dedecker2013minimax}
Dedecker, J. and Michel, B. (2013).
\newblock Minimax rates of convergence for {W}asserstein deconvolution with supersmooth errors in any dimension.
\newblock {\em J. Multivariate Anal.}, 122:278--291.

\bibitem[Delaigle and Meister, 2008]{delaigle2008density}
Delaigle, A. and Meister, A. (2008).
\newblock Density estimation with heteroscedastic error.
\newblock {\em Bernoulli}, 14(2):562--579.

\bibitem[Dempster et~al., 1977]{dempster1977maximum}
Dempster, A.~P., Laird, N.~M., and Rubin, D.~B. (1977).
\newblock Maximum likelihood from incomplete data via the {EM} algorithm.
\newblock {\em J. Roy. Statist. Soc. Ser. B}, 39(1):1--38.
\newblock With discussion.

\bibitem[DerSimonian and Laird, 1986]{dersimonian1986meta}
DerSimonian, R. and Laird, N. (1986).
\newblock Meta-analysis in clinical trials.
\newblock {\em Controlled clinical trials}, 7(3):177--188.

\bibitem[Dicker and Zhao, 2016]{dicker2016high}
Dicker, L.~H. and Zhao, S.~D. (2016).
\newblock High-dimensional classification via nonparametric empirical {B}ayes and maximum likelihood inference.
\newblock {\em Biometrika}, 103(1):21--34.

\bibitem[Doss et~al., 2020]{doss2020optimal}
Doss, N., Wu, Y., Yang, P., and Zhou, H.~H. (2020).
\newblock Optimal estimation of high-dimensional {G}aussian mixtures.
\newblock {\em arXiv preprint arXiv:2002.05818}.

\bibitem[Dyson, 1926]{dyson1926method}
Dyson, F. (1926).
\newblock A method for correcting series of parallax observations.
\newblock {\em Monthly Notices of the Royal Astronomical Society}, 86:686.

\bibitem[Dytso et~al., 2020]{dytso2019capacity}
Dytso, A., Yagli, S., Poor, H.~V., and Shamai, S. (2020).
\newblock The capacity achieving distribution for the amplitude constrained additive {G}aussian channel: an upper bound on the number of mass points.
\newblock {\em IEEE Trans. Inform. Theory}, 66(4):2006--2022.

\bibitem[Efron, 2011]{efron2011tweedie}
Efron, B. (2011).
\newblock Tweedie's formula and selection bias.
\newblock {\em J. Amer. Statist. Assoc.}, 106(496):1602--1614.

\bibitem[Efron, 2012]{efron2012large}
Efron, B. (2012).
\newblock {\em Large-scale inference: empirical {B}ayes methods for estimation, testing, and prediction}, volume~1.
\newblock Cambridge University Press.

\bibitem[Efron, 2014]{efron2014two}
Efron, B. (2014).
\newblock Two modeling strategies for empirical {B}ayes estimation.
\newblock {\em Statist. Sci.}, 29(2):285--301.

\bibitem[Efron, 2019]{efron2019bayes}
Efron, B. (2019).
\newblock Bayes, oracle {B}ayes and empirical {B}ayes.
\newblock {\em Statist. Sci.}, 34(2):177--201.

\bibitem[Efron and Hastie, 2016]{efron2016computer}
Efron, B. and Hastie, T. (2016).
\newblock {\em Computer age statistical inference}, volume~5.
\newblock Cambridge University Press.

\bibitem[Efron and Morris, 1972a]{efron1972empirical}
Efron, B. and Morris, C. (1972a).
\newblock Empirical {B}ayes on vector observations: an extension of {S}tein's method.
\newblock {\em Biometrika}, 59:335--347.

\bibitem[Efron and Morris, 1972b]{efron1972limiting}
Efron, B. and Morris, C. (1972b).
\newblock Limiting the risk of {B}ayes and empirical {B}ayes estimators. {II}. {T}he empirical {B}ayes case.
\newblock {\em J. Amer. Statist. Assoc.}, 67:130--139.

\bibitem[Efron and Morris, 1973a]{efron1973combining}
Efron, B. and Morris, C. (1973a).
\newblock Combining possibly related estimation problems.
\newblock {\em J. Roy. Statist. Soc. Ser. B}, 35:379--421.

\bibitem[Efron and Morris, 1973b]{efron1973stein}
Efron, B. and Morris, C. (1973b).
\newblock Stein's estimation rule and its competitors---an empirical {B}ayes approach.
\newblock {\em J. Amer. Statist. Assoc.}, 68:117--130.

\bibitem[Fan, 1991]{fan1991optimal}
Fan, J. (1991).
\newblock On the optimal rates of convergence for nonparametric deconvolution problems.
\newblock {\em Ann. Statist.}, 19(3):1257--1272.

\bibitem[Feng and Dicker, 2018]{feng2018approximate}
Feng, L. and Dicker, L.~H. (2018).
\newblock Approximate nonparametric maximum likelihood for mixture models: a convex optimization approach to fitting arbitrary multivariate mixing distributions.
\newblock {\em Comput. Statist. Data Anal.}, 122:80--91.

\bibitem[Folland, 1999]{folland1999real}
Folland, G.~B. (1999).
\newblock {\em Real analysis: modern techniques and their applications}, volume~40.
\newblock John Wiley \& Sons.

\bibitem[Gu and Koenker, 2016]{gu2016problem}
Gu, J. and Koenker, R. (2016).
\newblock On a problem of {R}obbins.
\newblock {\em Int. Stat. Rev.}, 84(2):224--244.

\bibitem[Gu and Koenker, 2017]{gu2017rebayes}
Gu, J. and Koenker, R. (2017).
\newblock Rebayes: an {R} package for empirical {B}ayes mixture methods.
\newblock Technical report, cemmap working paper.

\bibitem[Haferlach et~al., 2010]{haferlach2010clinical}
Haferlach, T., Kohlmann, A., Wieczorek, L., Basso, G., Te~Kronnie, G., B{\'e}n{\'e}, M.-C., De~Vos, J., Hern{\'a}ndez, J.~M., Hofmann, W.-K., Mills, K.~I., et~al. (2010).
\newblock Clinical utility of microarray-based gene expression profiling in the diagnosis and subclassification of leukemia: report from the {International Microarray Innovations in Leukemia Study Group}.
\newblock {\em Journal of clinical oncology}, 28(15):2529.

\bibitem[Heinrich and Kahn, 2018]{heinrich2018strong}
Heinrich, P. and Kahn, J. (2018).
\newblock Strong identifiability and optimal minimax rates for finite mixture estimation.
\newblock {\em Ann. Statist.}, 46(6A):2844--2870.

\bibitem[Ho and Nguyen, 2016]{ho2016strong}
Ho, N. and Nguyen, X. (2016).
\newblock On strong identifiability and convergence rates of parameter estimation in finite mixtures.
\newblock {\em Electron. J. Stat.}, 10(1):271--307.

\bibitem[Hoff, 2009]{hoff2009first}
Hoff, P.~D. (2009).
\newblock {\em A first course in {B}ayesian statistical methods}, volume 580.
\newblock Springer.

\bibitem[Hogg et~al., 2010]{HDB10}
Hogg, D.~W., Myers, A.~D., and Bovy, J. (2010).
\newblock Inferring the eccentricity distribution.
\newblock {\em The Astrophysical Journal}, 725(2):2166.

\bibitem[Ignatiadis and Sen, 2023]{ignatiadis2023empirical}
Ignatiadis, N. and Sen, B. (2023).
\newblock Empirical partially {B}ayes multiple testing and compound $\chi^2$ decisions.
\newblock {\em arXiv preprint arXiv:2303.02887}.

\bibitem[James and Stein, 1960]{james1961}
James, W. and Stein, C. (1960).
\newblock Estimation with quadratic loss.
\newblock In {\em Proc. 4th {B}erkeley {S}ympos. {M}ath. {S}tatist. and {P}rob., {V}ol. {I}}, pages 361--379. Univ. California Press, Berkeley-Los Angeles, Calif.

\bibitem[Jiang et~al., 2011]{jiang2011best}
Jiang, J., Nguyen, T., and Rao, J.~S. (2011).
\newblock Best predictive small area estimation.
\newblock {\em J. Amer. Statist. Assoc.}, 106(494):732--745.

\bibitem[Jiang, 2020]{jiang2020general}
Jiang, W. (2020).
\newblock On general maximum likelihood empirical {B}ayes estimation of heteroscedastic {IID} normal means.
\newblock {\em Electron. J. Stat.}, 14(1):2272--2297.

\bibitem[Jiang and Zhang, 2009]{jiang2009general}
Jiang, W. and Zhang, C.-H. (2009).
\newblock General maximum likelihood empirical {B}ayes estimation of normal means.
\newblock {\em Ann. Statist.}, 37(4):1647--1684.

\bibitem[Jiang and Zhang, 2010]{jiang2010empirical}
Jiang, W. and Zhang, C.-H. (2010).
\newblock Empirical {B}ayes in-season prediction of baseball batting averages.
\newblock In {\em Borrowing strength: theory powering applications---a {F}estschrift for {L}awrence {D}. {B}rown}, volume~6 of {\em Inst. Math. Stat. (IMS) Collect.}, pages 263--273. Inst. Math. Statist., Beachwood, OH.

\bibitem[Johnstone, 2019]{johnstone2019gaussian}
Johnstone, I.~M. (2019).
\newblock Gaussian estimation: {S}equence and wavelet models.
\newblock \url{https://imjohnstone.su.domains/GE_08_09_17.pdf}.

\bibitem[Kelly, 2012]{kelly2012measurement}
Kelly, B.~C. (2012).
\newblock Measurement error models in astronomy.
\newblock In {\em Statistical challenges in modern astronomy V}, pages 147--162. Springer.

\bibitem[Kiefer and Wolfowitz, 1956]{kiefer1956consistency}
Kiefer, J. and Wolfowitz, J. (1956).
\newblock Consistency of the maximum likelihood estimator in the presence of infinitely many incidental parameters.
\newblock {\em Ann. Math. Statist.}, 27:887--906.

\bibitem[Kim and Guntuboyina, 2022]{kim2022minimax}
Kim, A. K.~H. and Guntuboyina, A. (2022).
\newblock Minimax bounds for estimating multivariate {G}aussian location mixtures.
\newblock {\em Electron. J. Stat.}, 16(1):1461--1484.

\bibitem[Kim et~al., 2020]{kim2020fast}
Kim, Y., Carbonetto, P., Stephens, M., and Anitescu, M. (2020).
\newblock A fast algorithm for maximum likelihood estimation of mixture proportions using sequential quadratic programming.
\newblock {\em J. Comput. Graph. Statist.}, 29(2):261--273.

\bibitem[Koenker and Gu, 2017]{koenker2017rebayes}
Koenker, R. and Gu, J. (2017).
\newblock {REB}ayes: Empirical {B}ayes mixture methods in {R}.
\newblock {\em Journal of Statistical Software}, 82(8):1--26.

\bibitem[Koenker and Mizera, 2014]{koenker2014convex}
Koenker, R. and Mizera, I. (2014).
\newblock Convex optimization, shape constraints, compound decisions, and empirical {B}ayes rules.
\newblock {\em J. Amer. Statist. Assoc.}, 109(506):674--685.

\bibitem[Kohlmann et~al., 2008]{kohlmann2008international}
Kohlmann, A., Kipps, T.~J., Rassenti, L.~Z., Downing, J.~R., Shurtleff, S.~A., Mills, K.~I., Gilkes, A.~F., Hofmann, W.-K., Basso, G., Dell’Orto, M.~C., et~al. (2008).
\newblock An international standardization programme towards the application of gene expression profiling in routine leukaemia diagnostics: the {Microarray Innovations in LEukemia study prephase}.
\newblock {\em British journal of haematology}, 142(5):802--807.

\bibitem[Laird, 1978]{laird1978nonparametric}
Laird, N. (1978).
\newblock Nonparametric maximum likelihood estimation of a mixed distribution.
\newblock {\em J. Amer. Statist. Assoc.}, 73(364):805--811.

\bibitem[Lashkari and Golland, 2008]{lashkari2008convex}
Lashkari, D. and Golland, P. (2008).
\newblock Convex clustering with exemplar-based models.
\newblock In {\em Advances in neural information processing systems}, pages 825--832.

\bibitem[Laurent and Massart, 2000]{laurent2000adaptive}
Laurent, B. and Massart, P. (2000).
\newblock Adaptive estimation of a quadratic functional by model selection.
\newblock {\em Ann. Statist.}, 28(5):1302--1338.

\bibitem[Lesperance and Kalbfleisch, 1992]{lesperance1992algorithm}
Lesperance, M.~L. and Kalbfleisch, J.~D. (1992).
\newblock An algorithm for computing the nonparametric {MLE} of a mixing distribution.
\newblock {\em Journal of the American Statistical Association}, 87(417):120--126.

\bibitem[Lindsay, 1995]{lindsay1995mixture}
Lindsay, B.~G. (1995).
\newblock Mixture models: theory, geometry and applications.
\newblock In {\em NSF-CBMS regional conference series in probability and statistics}. JSTOR.

\bibitem[Lindsay and Roeder, 1993]{lindsay1993uniqueness}
Lindsay, B.~G. and Roeder, K. (1993).
\newblock Uniqueness of estimation and identifiability in mixture models.
\newblock {\em Canad. J. Statist.}, 21(2):139--147.

\bibitem[Liu and Zhu, 2007]{liu2007partially}
Liu, L. and Zhu, Y. (2007).
\newblock Partially projected gradient algorithms for computing nonparametric maximum likelihood estimates of mixing distributions.
\newblock {\em J. Statist. Plann. Inference}, 137(7):2509--2522.

\bibitem[Majewski et~al., 2017]{majewski2017apache}
Majewski, S.~R., Schiavon, R.~P., Frinchaboy, P.~M., Prieto, C.~A., Barkhouser, R., Bizyaev, D., Blank, B., Brunner, S., Burton, A., Carrera, R., et~al. (2017).
\newblock The {A}pache {P}oint {O}bservatory {G}alactic {E}volution {E}xperiment ({APOGEE}).
\newblock {\em The Astronomical Journal}, 154(3):94.

\bibitem[Marriott, 2002]{marriott2002local}
Marriott, P. (2002).
\newblock On the local geometry of mixture models.
\newblock {\em Biometrika}, 89(1):77--93.

\bibitem[Meister, 2009]{meister2009deconvolution}
Meister, A. (2009).
\newblock {\em Deconvolution problems in nonparametric statistics}, volume 193.
\newblock Springer Science \& Business Media.

\bibitem[Morris, 1983]{morris1983parametric}
Morris, C.~N. (1983).
\newblock Parametric empirical {B}ayes inference: theory and applications.
\newblock {\em J. Amer. Statist. Assoc.}, 78(381):47--65.
\newblock With discussion.

\bibitem[Nguyen, 2013]{nguyen2013convergence}
Nguyen, X. (2013).
\newblock Convergence of latent mixing measures in finite and infinite mixture models.
\newblock {\em Ann. Statist.}, 41(1):370--400.

\bibitem[Pfanzagl, 1988]{pfanzagl1988consistency}
Pfanzagl, J. (1988).
\newblock Consistency of maximum likelihood estimators for certain nonparametric families, in particular: mixtures.
\newblock {\em J. Statist. Plann. Inference}, 19(2):137--158.

\bibitem[Polyanskiy and Wu, 2020]{polyanskiy2020self}
Polyanskiy, Y. and Wu, Y. (2020).
\newblock Self-regularizing property of nonparametric maximum likelihood estimator in mixture models.
\newblock {\em arXiv preprint arXiv:2008.08244}.

\bibitem[Ratcliffe et~al., 2020]{ratcliffe2020tracing}
Ratcliffe, B.~L., Ness, M.~K., Johnston, K.~V., and Sen, B. (2020).
\newblock Tracing the assembly of the {M}ilky {W}ay’s disk through abundance clustering.
\newblock {\em The Astrophysical Journal}, 900(2):165.

\bibitem[Ray and Lindsay, 2005]{ray2005topography}
Ray, S. and Lindsay, B.~G. (2005).
\newblock The topography of multivariate normal mixtures.
\newblock {\em Ann. Statist.}, 33(5):2042--2065.

\bibitem[Redner and Walker, 1984]{redner1984mixture}
Redner, R.~A. and Walker, H.~F. (1984).
\newblock Mixture densities, maximum likelihood and the {EM} algorithm.
\newblock {\em SIAM Rev.}, 26(2):195--239.

\bibitem[Rigollet and Weed, 2018]{rigollet2018entropic}
Rigollet, P. and Weed, J. (2018).
\newblock Entropic optimal transport is maximum-likelihood deconvolution.
\newblock {\em C. R. Math. Acad. Sci. Paris}, 356(11-12):1228--1235.

\bibitem[Ritchie and Murray, 2019]{ritchie2019scalable}
Ritchie, J.~A. and Murray, I. (2019).
\newblock Scalable extreme deconvolution.
\newblock {\em arXiv preprint arXiv:1911.11663}.

\bibitem[Robbins, 1950]{robbins1950generalization}
Robbins, H. (1950).
\newblock A generalization of the method of maximum likelihood-estimating a mixing distribution.
\newblock In {\em Annals of Mathematical Statistics}, volume~21, pages 314--315.

\bibitem[Robbins, 1956]{robbins1956proceedings}
Robbins, H. (1956).
\newblock An empirical {B}ayes approach to statistics.
\newblock In {\em Proceedings of the {T}hird {B}erkeley {S}ymposium on {M}athematical {S}tatistics and {P}robability, 1954--1955, vol. {I}}, pages 157--163. Univ. California Press, Berkeley-Los Angeles, Calif.

\bibitem[Saha and Guntuboyina, 2020]{saha2020nonparametric}
Saha, S. and Guntuboyina, A. (2020).
\newblock On the nonparametric maximum likelihood estimator for {G}aussian location mixture densities with application to {G}aussian denoising.
\newblock {\em Ann. Statist.}, 48(2):738--762.

\bibitem[Staudenmayer et~al., 2008]{staudenmayer2008density}
Staudenmayer, J., Ruppert, D., and Buonaccorsi, J.~P. (2008).
\newblock Density estimation in the presence of heteroscedastic measurement error.
\newblock {\em J. Amer. Statist. Assoc.}, 103(482):726--736.

\bibitem[Tan, 2016]{tan2016steinized}
Tan, Z. (2016).
\newblock Steinized empirical {B}ayes estimation for heteroscedastic data.
\newblock {\em Statist. Sinica}, 26(3):1219--1248.

\bibitem[van~de Geer, 2000]{geer2000empirical}
van~de Geer, S. (2000).
\newblock {\em Empirical Processes in {M}-estimation}, volume~6.
\newblock Cambridge university press.

\bibitem[Villani, 2008]{villani2008optimal}
Villani, C. (2008).
\newblock {\em Optimal transport: old and new}, volume 338.
\newblock Springer Science \& Business Media.

\bibitem[Wang, 2007]{wang2007fast}
Wang, Y. (2007).
\newblock On fast computation of the non-parametric maximum likelihood estimate of a mixing distribution.
\newblock {\em J. R. Stat. Soc. Ser. B Stat. Methodol.}, 69(2):185--198.

\bibitem[Weinstein et~al., 2018]{weinstein2018group}
Weinstein, A., Ma, Z., Brown, L.~D., and Zhang, C.-H. (2018).
\newblock Group-linear empirical {B}ayes estimates for a heteroscedastic normal mean.
\newblock {\em J. Amer. Statist. Assoc.}, 113(522):698--710.

\bibitem[Wong and Shen, 1995]{wong1995probability}
Wong, W.~H. and Shen, X. (1995).
\newblock Probability inequalities for likelihood ratios and convergence rates of sieve {MLE}s.
\newblock {\em Ann. Statist.}, 23(2):339--362.

\bibitem[Wu and Yang, 2020]{wu2020optimal}
Wu, Y. and Yang, P. (2020).
\newblock Optimal estimation of {G}aussian mixtures via denoised method of moments.
\newblock {\em Ann. Statist.}, 48(4):1981--2007.

\bibitem[Xie et~al., 2012]{xie2012sure}
Xie, X., Kou, S.~C., and Brown, L.~D. (2012).
\newblock S{URE} estimates for a heteroscedastic hierarchical model.
\newblock {\em J. Amer. Statist. Assoc.}, 107(500):1465--1479.

\bibitem[Zhang, 1990]{zhang1990fourier}
Zhang, C.-H. (1990).
\newblock Fourier methods for estimating mixing densities and distributions.
\newblock {\em Ann. Statist.}, 18(2):806--831.

\bibitem[Zhang, 2009]{zhang2009generalized}
Zhang, C.-H. (2009).
\newblock Generalized maximum likelihood estimation of normal mixture densities.
\newblock {\em Statist. Sinica}, 19(3):1297--1318.

\end{thebibliography}
\end{document}